%% file: RK-51.tex
\documentclass[a4paper]{amsart}

\usepackage{latexsym}
\usepackage{times}
\usepackage{amsfonts}
\usepackage{amssymb}
\usepackage{amsthm}
\usepackage{amsmath}
\usepackage{mathrsfs}
\usepackage{mathtools}
\usepackage{enumitem}
\usepackage{environ}
\usepackage{stmaryrd}
\theoremstyle{plain}
 \newtheorem{theorem}{\bf Theorem}[section]
 \newtheorem{lemma}[theorem]{Lemma}
 \newtheorem{corollary}[theorem]{Corollary}
 \newtheorem{claim}[theorem]{Claim}
\theoremstyle{definition}
 \newtheorem{definition}[theorem]{Definition}
 
  \newtheorem{q}[theorem]{Question}
 
 \newtheorem{remark}[theorem]{Remark}

\numberwithin{equation}{section}

\newcommand{\TS}{\textstyle}

\renewcommand{\a}{\alpha}
\renewcommand{\b}{\beta}
\newcommand{\g}{\gamma}

\renewcommand{\d}{\delta}

\renewcommand{\o}{\omega}

\DeclareFontFamily{U}{wncy}{}
 \DeclareFontShape{U}{wncy}{m}{n}{<->wncyr10}{}
 \DeclareSymbolFont{mcy}{U}{wncy}{m}{n}
 \DeclareMathSymbol{\dc}{\mathord}{mcy}{"64}

\DeclareFontFamily{U}{wncy}{}
 \DeclareFontShape{U}{wncy}{m}{n}{<->wncyr10}{}
 \DeclareSymbolFont{mcy}{U}{wncy}{m}{n}
 \DeclareMathSymbol{\bc}{\mathord}{mcy}{"62}

\newcommand{\ps}{\emptyset}

\newcommand{\sledi}{\Rightarrow}
\newcommand{\rest}{\upharpoonright}
\newcommand{\akko}{\Leftrightarrow}

\newcommand{\forallbutfin}{{\forall}^{\infty}}
\newcommand{\existsinf}{{\exists}^{\infty}}

\newcommand{\Q}{\mathbb Q}

\renewcommand{\P}{\mathbb P}
\renewcommand{\c}{\mathfrak c}
\newcommand{\z}{\mathfrak z}
\newcommand{\cp}{\mathcal P}

\newcommand{\cu}{\mathcal U}
\newcommand{\cv}{\mathcal V}

\newcommand{\seq}[1]{\left<#1\right>}
\newcommand{\set}[1]{\left\{#1\right\}}
\newcommand{\abs}[1]{\left\vert#1\right\vert}
\newcommand{\br}[1]{\left(#1\right)}

\newcommand{\lp}{\left\llbracket}
\newcommand{\rp}{\right\rrbracket}

\newcommand{\ran}{\operatorname{ran}}

\newcommand{\id}{\operatorname{id}}

\newcommand{\cf}{\operatorname{cof}}

\newcommand{\sset}{\operatorname{set}}

\renewcommand{\int}{\operatorname{int}}

\newcommand{\fin}{\operatorname{FIN}}
\newcommand{\BS}{{\omega}^{\omega}}
\newcommand{\MA}{\mathrm{MA}}
\newcommand{\CH}{\mathrm{CH}}
\newcommand{\ZFC}{\mathrm{ZFC}}

\NewEnviron{killcontents}{}
\author[B. Kuzeljevic]{Borisa Kuzeljevic}
\address{Department of Mathematics, National University of Singapore, Singapore 119076.}
\address{Mathematical Institute SANU, Kneza Mihaila 36, 11001 Belgrade, Serbia.}
\curraddr{}
\email{borisa@mi.sanu.ac.rs}
\urladdr{http://www.mi.sanu.ac.rs/$\sim$borisa}
\dedicatory{}

\author[D. Raghavan]{Dilip Raghavan}
\address{Department of Mathematics, National University of Singapore, Singapore 119076.}
\curraddr{}
\email{raghavan@math.nus.edu.sg}
\urladdr{http://www.math.nus.edu.sg/$\sim$raghavan}
\dedicatory{}

\title[A long chain of P-points]{A long chain of P-points}
\keywords{ultrafilter, Rudin-Keisler order, Tukey reducibility, P-point}
\subjclass[2010]{03E50, 03E05, 03E35, 54D80}
\thanks{Both authors were partially supported by NUS research grant number R-146-000-161-133. The first author was partially supported by the MPNTR grant ON174006}
\date{\today}

\begin{document}

\begin{abstract}
The notion of a $\delta$-generic sequence of P-points is introduced in this paper.
It is proved assuming the Continuum Hypothesis that for each $\delta < {\omega}_{2}$, any $\delta$-generic sequence of P-points can be extended to an ${\omega}_{2}$-generic sequence.
This shows that the Continuum Hypothesis implies that there is a chain of P-points of length ${\c}^{+}$ with respect to both Rudin-Keisler and Tukey reducibility.
The proofs can be easily adapted to get such a chain of length ${\c}^{+}$ under a more general hypothesis like Martin's Axiom.
These results answer an old question of Andreas Blass.
\end{abstract}
\maketitle
\section{Introduction}
In his 1973 paper on the structure of P-points \cite{blassrk}, Blass posed the following question:
\begin{q}[Question 4 of \cite{blassrk}]\label{q:blass}
 What ordinals can be embedded into the class of P-points when equipped with the ordering of Rudin-Keisler reducibility assuming Martin's Axiom?   
\end{q}
Recall that an ultrafilter $\mathcal U$ on $\o$ is called a \emph{P-point} if for any $\set{a_n:n < \o} \subset \mathcal U$ there is $a \in \mathcal U$ such that $a \subset^* a_n$, for every $n<\o$.
All filters $\cu$ occurring in this paper are assumed to be \emph{proper} -- meaning that $0 \notin \cu$ -- and \emph{non-principal} -- meaning that $\cu$ extends the filter of co-finite sets.
It is not hard to see that an ultrafilter $\cu$ is a P-point if and only if every $f \in \BS$ becomes either constant or finite-to-one on a set in $\cu$.
Recall also the well-known Rudin-Keisler ordering on P-points.
\begin{definition} \label{def:rk}
Let $\cu$ and $\cv$ be ultrafilters on $\o$.
We say that $\cu \le_{RK} \cv $, i.e.\@ $\cu$ is \emph{Rudin-Keisler (RK) reducible} to $\cv$ or $\cu$ is \emph{Rudin-Keisler (RK) below} $\cv$, if there is $f\in\o^{\o}$ such that $A \in \cu \akko f^{-1}(A)\in \cv$ for every $A \subset \o$.
We say that $\cu \; {\equiv}_{RK} \; \cv$, i.e.\@ $\cu$ is \emph{RK equivalent} to $\cv$, if $\cu \; {\leq}_{RK} \; \cv$ and $\cv \; {\leq}_{RK} \; \cu$.
\end{definition}
It is worth noting here that the class of P-points is downwards closed with respect to this order.
In other words, if $\cu$ is a P-point, then every ultrafilter that is RK below $\cu$ is also a P-point.
It should also be noted that the existence of P-points cannot be proved in $\ZFC$ by a celebrated result of Shelah (see \cite{PIF}).
Hence it is natural to assume some principle that guarantees the existence of ``many'' P-points when studying their properties under the RK or other similar orderings.
Common examples of such principles include the Continuum Hypothesis ($\CH$), Martin's axiom ($\MA$), and Martin's Axiom for $\sigma$-centered posets ($\MA(\sigma-\textrm{centered})$).

Blass showed in \cite{blassrk} that the ordinal ${\omega}_{1}$ can be embedded into the P-points with respect to the RK ordering, if $\MA(\sigma-\textrm{centered})$ holds.
In particular under $\CH$, the ordinal $\c = {2}^{{\aleph}_{0}}$ embeds into the P-points with respect to RK reducibility.
Note that no ultrafilter $\cv$ can have more than $\c$ predecessors in the RK order.
This is because for each $f \in \BS$, there can be at most one ultrafilter $\cu$ for which $f$ witnesses the relation $\cu \; {\leq}_{RK} \; \cv$.
Therefore there can be no RK-chain of P-points of length ${\c}^{+} + 1$.
Thus the strongest possible positive answer to Question \ref{q:blass} is that the ordinal ${\c}^{+}$ embeds into the P-points under the RK ordering. 

Though there have not been many advances directly pertaining to  Question \ref{q:blass} after \cite{blassrk}, several results have dealt with closely related issues.
Rosen~\cite{rosen} showed assuming $\CH$ that the ordinal ${\omega}_{1}$ occurs as an RK initial segment of the P-points.
In other words, he produced a strictly increasing RK chain of P-points of length ${\omega}_{1}$ that is downwards closed under the relation ${\leq}_{RK}$ up to RK equivalence.
Laflamme ~\cite{lacomplete} further investigated well-ordered initial segments of the P-points under the RK ordering.
For each countable ordinal $\alpha$, he produced a forcing notion ${\P}_{\alpha}$ that generically adds an RK initial segment of the P-points of order type $\alpha$.
He also gave combinatorial characterizations of the generics added by these forcing notions.

Dobrinen and Todorcevic~\cite{dt} considered the Tukey ordering on P-points.
Recall that for any $\mathcal X,\mathcal Y\subset \cp(\o)$, a map $\phi:\mathcal X\to \mathcal Y$ is said to be \emph{monotone} if for every $a,b\in \mathcal X$, $a\subset b$ implies $\phi(a)\subset \phi(b)$, while $\phi$ is said to be \emph{cofinal} in $\mathcal Y$ if for every $b\in \mathcal Y$ there is $a\in \mathcal X$ so that $\phi(a)\subset b$.
\begin{definition}\label{d2}
 We say that $\cu \; {\leq}_{T} \;\cv$, i.e.\@ $\cu$ is \emph{Tukey reducible} to $\cv$ or $\cu$ is \emph{Tukey below} $\cv$, if there is a monotone $\phi: \cv \to \cu$ which is cofinal in $\cu$.
 We say that $\cu \; {\equiv}_{T} \; \cv$, i.e.\@ $\cu$ is \emph{Tukey equivalent} to $\cv$, if $ \cu \; {\leq}_{T} \; \cv$ and $\cv \; {\leq}_{T} \; \cu$.
\end{definition}
It is not hard to see that $\cu \; {\leq}_{RK} \; \cv$ implies $\cu \; {\leq}_{T} \; \cv$, and it was proved by Raghavan and Todorcevic in \cite{tukey} that $\CH$ implies the existence of P-points $\cu$ and $\cv$ such that $\cv \; {<}_{RK} \; \cu$, but $\cv \; {\equiv}_{T} \; \cu$.
Their result showed that the orders ${\leq}_{T}$ and ${\leq}_{RK}$ can diverge in a strong sense even within the realm of P-points, although by another result from \cite{tukey}, the two orders coincide within the realm of selective ultrafilters.
In \cite{dt}, Dobrinen and Todorcevic showed that
every P-point has only $\c$ Tukey predecessors by establishing the following useful fact.
\begin{theorem}[Dobrinen and Todorcevic~\cite{dt}] \label{thm:dt}
 If $\cv$ is a P-point and $\cu$ is any ultrafilter with $\cu \; {\leq}_{T} \; \cv$, then there is a continuous monotone $\phi: \cp(\omega) \rightarrow \cp(\omega)$ such that $\phi \rest \cv: \cv \rightarrow \cu$ is a monotone map that is cofinal in $\cu$.
\end{theorem}
They used this in \cite{dt} to embed the ordinal ${\omega}_{1}$ into the class of P-points equipped with the ordering of Tukey reducibility assuming $\MA(\sigma-\textrm{centered})$.
Question 54 of \cite{dt} asks whether there is a strictly increasing Tukey chain of P-points of length ${\c}^{+}$. 
In \cite{dttrans1} and \cite{dttrans2}, Dobrinen and Todorcevic proved some analogues of Laflamme's results mentioned above for the Tukey order.
In particular, they showed that each countable ordinal occurs as a Tukey initial segment of the class of P-points, assuming $\MA(\sigma-\textrm{centered})$.
Raghavan and Shelah proved in \cite{pomegaembed} that $\MA(\sigma-\textrm{centered})$ implies that the Boolean algebra $\cp(\omega) \slash \fin$ equipped with its natural ordering embeds into the P-points with respect to both the RK and Tukey orders.
In particular, for each $\alpha < {\c}^{+}$, the ordinal $\alpha$ embeds into the P-points with respect to both of these orders. 

In this paper, we give a complete answer to Question \ref{q:blass} by showing that the ordinal ${\c}^{+}$ can be embedded into the P-points under RK reducibility.
Our chain of P-points of length ${\c}^{+}$ will also be strictly increasing with respect to Tukey reducibility, so we get a positive answer to Question 54 of \cite{dt} as well.
The construction will be presented assuming $\CH$ for simplicity.
However the same construction can be run under $\MA$ with some fairly straightforward modifications.
We will try to point out these necessary modifications at the appropriate places in the proofs below.
We will make use of Theorem \ref{thm:dt} in our construction to ensure that our chain is also strictly increasing in the sense of Tukey reducibility.
However the continuity of the monotone maps will not be important for us.
Rather any other fixed collection of $\c$ many monotone maps from $\cp(\omega)$ to itself which is large enough to catch all Tukey reductions from any P-point will suffice.
For instance, it was proved in \cite{tukey} that the collection of monotone maps of the first Baire class suffice to catch all Tukey reductions from any basically generated ultrafilter, which form a larger class of ultrafilters than the P-points.
So we could equally well have used monotone maps of the first Baire class in our construction.

A powerful machinery for constructing objects of size ${\aleph}_{2}$ under $\diamondsuit$ was introduced by Shelah, Laflamme, and Hart in \cite{sh:162}.
This machinery can be used to build a chain of P-points of length ${\omega}_{2}$ that is strictly increasing with respect to both RK and Tukey reducibility assuming $\diamondsuit$.
More generally, the methods in \cite{sh:162} allow for the construction of certain types of objects of size ${\lambda}^{+}$ from a principle called ${\mathord{\mathrm{Dl}}}_{\lambda}$, which is closely related to ${\diamondsuit}_{\lambda}$.
Shelah's results in \cite{sh:922} imply that ${\mathord{\mathrm{Dl}}}_{\c}$ follows from $\MA$ when $\c > {\aleph}_{1}$ and is a successor cardinal.  
Thus the methods of \cite{sh:162}, when combined with the results of \cite{sh:922}, can also be used to get a chain of P-points of length ${\c}^{+}$ when $\c > {\aleph}_{1}$, $\c$ is a successor cardinal, and $\MA$ holds.
However the techniques from \cite{sh:162} are inadequate to treat the case when only $\CH$ holds\footnote{Personal communication with Shelah.
The second author thanks Shelah for discussing these issues with him.}.
\section{Preliminaries} \label{sec:prelim}
We use standard notation. 
``$\forallbutfin x$ \ldots'' abbreviates the quantifier ``for all but finitely many $x$ \ldots'' and ``$\existsinf x$ \ldots'' stands for ``there exist infinitely many $x$ such that \dots''.
${[\omega]}^{\omega}$ refers to the collection of all infinite subsets of $\omega$, and ${[\omega]}^{< \omega}$ is the collection of all finite subsets of $\omega$.
The symbol $\subset^*$ denotes the relation of containment modulo a finite set: $a\subset^* b$ iff $a\setminus b$ is finite. 

Even though the final construction uses $\CH$, none of the preliminary lemmas rely on it.
In fact, $\CH$ will only be used in Section \ref{sec:long}.
The results in Sections \ref{sec:prelim}--\ref{sec:top} are all in $\ZFC$, and $\CH$ is needed later on to ensure that these results are sufficient to carry out the final construction and that they are applicable to it.
So we do not need to make any assumptions about cardinal arithmetic at the moment.

One of the difficulties in embedding various partially ordered structures into the P-points is that, unlike the class of all ultrafilters on $\omega$, this class is not $\c$-directed under ${\leq}_{RK}$.
It is not hard to prove in $\ZFC$ that if $\{{\cu}_{i}: i < \c\}$ is an arbitrary collection of ultrafilters on $\omega$, then there is an ultrafilter $\cu$ on $\omega$ such that $\forall i < \c\left[{\cu}_{i} \; {\leq}_{RK} \; \cu\right]$.
However it is well-known that there are two P-points $\cu$ and $\cv$ with no RK upper bound that is a P-point under $\CH$ (see \cite{blassrk}). 
Even if we restrict ourselves to chains, it is easy to construct, assuming $\CH$, an RK chain of P-points $\langle {\cu}_{i}: i < {\omega}_{1} \rangle$ which has no P-point upper bound.
The strategy for ensuring that our chains are always extensible is to make each ultrafilter ``very generic'' (with respect to some partial order to be defined in Section \ref{sec:top}).
The same strategy was used in \cite{pomegaembed}, but with one crucial difference.
Only $\c$ many ultrafilters were constructed in \cite{pomegaembed} and so all of the ultrafilters in question could be built simultaneously in $\c$ steps.
But by the time we get to, for example, the ultrafilter ${\cu}_{{\omega}_{1}}$ in our present construction, all of the ultrafilters ${\cu}_{i}$, for $i < {\omega}_{1}$, will have been fully determined with no room for further improvements.
Thus the ultrafilters that were built before should have already predicted and satisfied the requirements imposed by ${\cu}_{{\omega}_{1}}$, and indeed by all of the ultrafilters to come in future.
This is possible because there are only ${\omega}_{1}$ possible initial segments of ultrafilters.
More precisely, each of the ${\omega}_{2}$ many P-points is generated by a ${\subset}^{\ast}$-descending tower ${A}_{\alpha} = \langle {C}^{\alpha}_{i}: i < {\omega}_{1} \rangle$.
For each $j < {\omega}_{1}$, the collection $\{{A}_{\alpha} \rest j: \alpha < {\omega}_{2}\}$ just has size ${\omega}_{1}$.
This leads to the notion of a $\delta$-generic sequence, which is essentially an RK-chain of P-points of length $\delta$ where every ultrafilter in the sequence has predicted and met certain requirements involving such initial segments of potential future ultrafilters and potential RK maps going from such initial segments into it.
The precise definition is given in Definition \ref{d22}.
Our main result is that such generic sequences can always be extended.
\begin{remark}\label{r1}
In the rest of the paper we will use the following notation:
\begin{itemize}
\item For $A\subset \o$, define $A^0=A$ and $A^1=\o\setminus A$.
\item For $A\in [\o]^{\o}$, $A(m)$ is the $m$th element of $A$ in its increasing enumeration.
\item For $A\in [\o]^{\o}$ and $k,m<\o$ let $A[k,m)=\set{A(l):k\le l<m}$.
\item For $A\in [\o]^{\o}$ and $m<\o$, let $A[m]=\set{A(l):l\ge m}$.
\item For a sequence $c\!=\!\seq{c(\xi):\xi<\a}$, let $\sset(c)=\bigcup_{\xi<\a}c(\xi)$.
\item For a sequence $c\!=\!\seq{c(\xi):\xi\!<\!\a}$ and $\eta\!<\!\a$, let $\sset(c)\lp \eta\rp\!=\!\bigcup_{\eta\le\xi<\a}c(\xi)$.
\item For $m\in \o$, let $s(m)=m(m+1)/2$ and $t(m)=\sum_{s(m)\le k<s(m+1)}(k+1)$.
\item A function $f\in \o^{\o}$ is \emph{increasing} if $\forall n\in \o\ [f(n)\le f(n+1)]$.
\end{itemize}
We also consider, for an ordinal $\a$, triples $\rho\!=\!\seq{\bar D,\bar K, \bar \pi}$ where $\bar D\!=\!\seq{D_n:n\!<\!\a}$ is a sequence of sets  in $\cp(\o)$, $\bar K=\seq{K_{m,n}:m\le n<\a}$ is a sequence in $\o$ and $\bar \pi=\seq{\pi_{n,m}:m\le n<\a}$ is a sequence in $\o^{\o}$. Then, for $n<\a$, denote:
\begin{itemize}
\item $\Delta^{\rho}_n=\bigcup_{m\le n}H^{\rho}_{m,n}$ where $H^{\rho}_{m,n}$ is the set of all $k\in D_m[K_{m,n},K_{m,n+1})$ such that $\pi_{m',m}^{-1}(\set{k})\cap D_{m'}[K_{m',n},K_{m',n+1})=0$ for all $m<m'\le n$;
\item $L^{\rho}_0=0$ and for each $n \in \omega$, $L^{\rho}_{n+1}=L^{\rho}_n+\abs{\Delta^{\rho}_n}$.
\end{itemize}
\end{remark}
Regarding the definition of ${\Delta}^{\rho}_{n}$, the reader should think of the sequence $\bar{K}$ as defining an interval partition of ${D}_{m}$, for each $m$.
Then for any $m \leq n$, ${H}^{\rho}_{m, n}$ consists of the points in the $n$th interval of ${D}_{m}$ that do not have a preimage in the $n$th interval of ${D}_{m'}$ for any $m < m' \leq n$.

Next we recall the notion of a rapid ultrafilter.
All the P-points in our construction will be rapid.
This happens because the requirement of genericity forces our ultrafilters to contain some ``very thin'' sets.
However they cannot be too thin, lest we end up with a Q-point.
Rapidity turns out to be a good compromise.
\begin{definition}\label{d15}
We say that the ultrafilter $\mathcal U$ is \emph{rapid} if for every $f\in\o^{\o}$ there is $X\in \mathcal U$ such that $X(n)\ge f(n)$ for every $n<\o$.
\end{definition}

\begin{lemma}\label{t33}
If $\mathcal U$ is a rapid ultrafilter, then for every $f\in\o^{\o}$ and every $X\in \mathcal U$ there is $Y\in \mathcal U$ such that $Y\subset X$ and $Y(n)\ge f(n)$ for every $n<\o$.
\end{lemma}

\begin{proof}
Let $f\in\o^{\o}$ and $X\in \mathcal U$. Let $Z$ be as in Definition \ref{d15}, i.e. $Z\in \cu$ and $Z(n)\ge f(n)$ for $n<\o$. Let $Y=X\cap Z$ and note $Y\in \cu$. Since $Y\subset Z$ we have $Y(n)\ge Z(n)\ge f(n)$ for every $n<\o$.
\end{proof}

\begin{claim}\label{t102}
Let $\seq{\cu_n:n<\o}$ be a sequence of distinct P-points and $M$ a countable elementary submodel of $H_{\br{2^{\c}}^+}$ containing $\seq{\cu_n:n<\o}$. For $n<\o$ let $A_n\in\cu_n$ be such that $A_n\subset^* A$ for each $A\in\cu_n\cap M$. Then $\abs{A_n\cap A_m}\!<\!\o$ for $m\!<\!n\!<\!\o$.
\end{claim}

\begin{proof}
Fix $m<n<\o$ and pick a set $A_{m,n}$ so that $A_{m,n}\in \cu_m$ and $\o\setminus A_{m,n}\in \cu_n$ (this can be done because $\cu_{m}\neq \cu_n$ for $m\neq n$). By elementarity of $M$, since $\cu_m,\cu_n\in M$ we can assume that $A_{m,n}\in M$. Now we have that $A_m\subset^* A_{m,n}$ and that $A_n\subset^* \o\setminus A_{m,n}$. So $A_m\cap A_n\subset^*0$ implying $\abs{A_m\cap A_n}<\o$.
\end{proof}
Now we introduce one of the basic partial orders used in the construction.
The definition of $\langle \P, \leq \rangle$ is inspired by the many examples of creature forcing in the literature, for example see \cite{creaturebook}.
However there is no notion of norm in this case, or rather the norm is simply the cardinality.
\begin{definition}\label{d13}
Define $\P$ as the set of all functions $c:\o\to [\o]^{<\o}\setminus\set{0}$ such that $\forall n\in\o\ [\abs{c(n)}<\abs{c(n+1)}\wedge \max(c(n))<\min(c(n+1))]$. If $c,d\in \P$, then $c\le d$ if there is $l<\o$ such that $c\le_l d$, where $c\le_l d \akko \forall m\ge l\ \exists n\ge m\ [c(m)\subset d(n)]$.
\end{definition}
Note that if $d \leq c$, then $\sset(d) \; {\subset}^{\ast} \; \sset(c)$.
Each of our ultrafilters will be generated by a tower of the form $\langle \sset({c}_{i}): i < {\omega}_{1} \rangle$ where $\langle {c}_{i}: i < {\omega}_{1} \rangle$ is some decreasing sequence of conditions in $\P$.
This guarantees that each ultrafilter is a P-point.
\begin{remark}\label{r5}
Note that $\seq{\P,\le}$ is a partial order and has the following properties:
\begin{enumerate}
\item\label{i201} For any $c\in \P$ we have $\min(c(n))\ge n$;
\item\label{i204} If $a\le_n b$ and $b\le_m c$, then $a\le_l c$ for $l=\max\set{m,n}$. To see this take any $k\ge l$. There is $k'\ge k\ge l$ such that $a(k)\subset b(k')$. There is also $k''\ge k'\ge k\ge l$ so that $a(k)\subset b(k')\subset c(k'')$ as required.
\item\label{i203} Let $\set{d_n\!:n\!<\!\o}\!\subset\! \P$ be such that $\forall n\!<\!\o\ [d_{n+1}\!\le_{m_n}\! d_n]$. Then $d_{n+1}\!\le_l\! d_0$ for $n<\o$ and $l=\max\set{m_k:k\le n}$. The proof is by induction on $n$. For $n\!=\!0$, $d_1\!\le_{m_0}\!d_0$. Let $n\!>\!0$ and assume the statement is true for all $m\!\le\! n$. Then $d_{n+2}\le_{m_{n+1}}d_{n+1}$ and $d_{n+1}\le_{l_1}d_0$, where $l_1=\max\set{m_k:k\le n}$. So, by (\ref{i204}), $d_{n+2}\le_l d_0$ for $l=\max\set{m_k:k\le n+1}$.
\end{enumerate}
\end{remark}

\begin{definition}\label{d21}
A triple $\seq{\pi,\psi,c}$ is called a \emph{normal triple} if $\pi,\psi\in \o^{\o}$, for every $l\le l'<\o$ we have that $\psi(l)\le \psi(l')$, if $\ran(\psi)$ is infinite, and if $c\in \P$ is such that for $l<\o$ we have $\pi''c(l)=\set{\psi(l)}$ and for $n\in \o\setminus \sset(c)$ we have $\pi(n)=0$.
\end{definition}
The notion of a normal triple will help us ensure that when $\alpha < \beta < {\omega}_{2}$, the RK reduction from ${\cu}_{\beta}$ to ${\cu}_{\alpha}$ is witnessed by a function that is increasing on a set in ${\cu}_{\beta}$.
Thus our sequence of P-points will actually even be a chain with respect to the order ${\leq}^{+}_{RB}$.
Recall that for ultrafilters $\cu$ and $\cv$ on $\omega$, $\cu \; {\leq}^{+}_{RB} \; \cv$ if there is an increasing $f \in \BS$ such that $A \in \cu \iff {f}^{-1}(A) \in \cv$, for every $A \subset \omega$.
More information about the order ${\leq}^{+}_{RB}$ can be found in \cite{LaZh}
\begin{remark}\label{r0}
Suppose that $d\le c$ and $\seq{\pi,\psi,c}$ is a normal triple. There is $N<\o$ such that for every $k,l\in \sset(d)\setminus N$ if $k\le l$, then $\pi(k)\le \pi(l)$.
\end{remark}

\begin{lemma}\label{t14}
Suppose that $\seq{\pi,\psi,b}$ is a normal triple, that $a\subset \pi''\sset(c)\lp n_0\rp$, and that $c\le_{n_0} b$. For $n<\o$ denote $F_n=\set{m<\o:\pi''c(m)=\set{a(n)}}$. Then for $n<\o$: $F_n\setminus n_0\neq 0$, $\abs{F_n}<\o$ and $\max(F_n)<\min(F_{n+1}\setminus n_0)\le \max(F_{n+1})$.
\end{lemma}

\begin{proof}
Fix $n<\o$. By the choice of the set $a$ there are $k\ge n_0$ and $x\in c(k)$ so that $\pi(x)=a(n)$. Since $c\le_{n_0}b$ there is $m\ge k$ so that $c(k)\subset b(m)$ and because $\seq{\pi,\psi,b}$ is a normal triple we know $\pi''b(m)=\set{a(n)}$. So $k\in F_n$ is such that $k\ge n_0$, implying $F_n\setminus n_0\neq 0$. To show that each $F_n$ is finite take any $k\in F_n$. As above, $F_{n+1}\setminus n_0\neq 0$. Let $k'=\min\br{F_{n+1}\setminus n_0}$. We will show $k<k'$. If $k<n_0$ the statement follows. If $k\ge n_0$ then there are $m_1,m_2$ so that $c(k)\!\subset\! b(m_1)$ and $c(k')\!\subset\! b(m_2)$. Since $\pi''b(m_1)\!=\!\set{a(n)}$, $\pi''b(m_2)\!=\!\set{a(n+1)}$ and $\seq{\pi,\psi,b}$ is a normal triple we have $m_1\!<\!m_2$ and consequently $k\!<\!k'$. So $\max(F_n)\!<\!k'$ implying both $\abs{F_n}\!<\!\o$ and $\max(F_n)\!<\!\min(F_{n+1}\setminus n_0)\!\le\! \max(F_{n+1})$.
\end{proof}
Now comes the central definition of the construction.
We will briefly try to explain the intuition behind each of the clauses below.
Clauses (\ref{i2}), (\ref{i4}), and (\ref{i9}) are self explanatory and were commented on earlier.
Clause (\ref{i5}) guarantees that ${\pi}_{\beta, \alpha}$ is an RK map from ${\cu}_{\beta}$ to ${\cu}_{\alpha}$ whenever $\alpha \leq \beta$.
This is because if $\cu,\mathcal V$ are ultrafilters on $\o$ and $f\in\o^{\o}$ is such that $f''b\in \cu$ for every $b\in \mathcal V$, then $f$ witnesses that $\cu \; {\leq}_{RK} \; \cv$.
Clause (\ref{i0}) says that if $\alpha \leq \beta \leq \gamma$, then ${\pi}_{\gamma, \alpha} = {\pi}_{\beta, \alpha} \circ {\pi}_{\gamma, \beta}$ modulo a set in ${\cu}_{\gamma}$.
This type of commuting of RK maps is unavoidable in a chain.
Clause (\ref{i6}) makes the map ${\pi}_{\beta, \alpha}$ increasing on a set in ${\cu}_{\beta}$; this makes ${\cu}_{\alpha} \; {\leq}^{+}_{RB} \; {\cu}_{\beta}$.
The fact that ${\pi}_{\beta, \alpha}$ is constant on ${c}^{\beta}_{i}(n)$ for almost all $n$ is helpful for killing unwanted Tukey maps.

Clauses (\ref{i1}) and (\ref{i8}) deal with the prediction of requirements imposed by future ultrafilters.
To understand (\ref{i1}), suppose for simplicity that $\langle {\cu}_{n}: n < \omega \rangle$ has already been built and that ${\cu}_{\omega}$ is being built.
At a certain stage you have decided to put $\sset(d) \in {\cu}_{\omega}$, for some $d \in \P$, and you have also decided the  sequence of RK maps $\langle {\pi}_{\omega, i}: i \leq n \rangle$, for some $n \in \omega$.
In particular ${\pi}_{\omega, n}''\sset(d) \in {\cu}_{n}$.
Now you wish to decide ${\pi}_{\omega, n + 1}$ and you are permitted to extend $d$ to some ${d}^{\ast} \leq d$ in the process.
But you must ensure that ${\pi}_{\omega, n + 1}''\sset({d}^{\ast}) \in {\cu}_{n + 1}$ and that ${\pi}_{\omega, n }$ commutes through ${\pi}_{\omega, n + 1}$.
Clause (\ref{i1}) says that ${\cu}_{n + 1}$ anticipated this requirement and that there is a $b \in {\cu}_{n + 1}$ (in fact cofinally many $b$) that allows this requirement to be fulfilled.
Next to understand (\ref{i8}), suppose that $\langle {\cu}_{\alpha}: \alpha < {\omega}_{1} \rangle$ has been built and that you are building ${\cu}_{{\omega}_{1}}$.
At some stage you have determined that $\langle \sset({d}_{n}): n < \omega \rangle \subset {\cu}_{{\omega}_{1}}$, for some decreasing sequence of conditions $\langle {d}_{n}: n < \omega \rangle \subset \P$.
You have also determined the sequence $\langle {\pi}_{{\omega}_{1}, n}: n < \omega \rangle$.
In particular $\forall n, m < \omega \left[ {\pi}_{{\omega}_{1}, n}''\sset({d}_{m}) \in {\cu}_{n} \right]$, and each ${\pi}_{{\omega}_{1}, n}$ has the right form on some ${d}_{m}$.
Now you would like to find a ${d}^{\ast} \in \P$ that is below all of the ${d}_{n}$.
You would also like to determine ${\pi}_{{\omega}_{1}, \omega}$.
But you must ensure that ${\pi}_{{\omega}_{1}, \omega}''\sset({d}^{\ast}) \in {\cu}_{\omega}$, that ${\pi}_{{\omega}_{1}, \omega}$ has the appropriate form on ${d}^{\ast}$, and that all of the ${\pi}_{{\omega}_{1}, n}$ commute through ${\pi}_{{\omega}_{1}, \omega}$.
Clause (\ref{i8}) says that ${\cu}_{\omega}$ anticipated this requirement and that there is a $b \in {\cu}_{\omega}$ (in fact cofinally many $b$) enabling you to find such a ${d}^{\ast}$ and ${\pi}_{{\omega}_{1}, \omega}$.
\begin{definition}\label{d22}
Let $\d\le \o_2$ . We call $\seq{\seq{c^{\a}_i:i<\c\wedge\a<\d},\seq{\pi_{\b,\a}:\a\le\b<\d}}$ \emph{$\d$-generic} if and only if:
\begin{enumerate}
\item\label{i2} for every $\a<\d$, $\seq{c_i^{\a}:i<\c}$ is a decreasing sequence in $\P$;
\item\label{i4} for every $\a<\d$, $\mathcal U_{\a}=\set{a\in \cp(\o):\exists i<\c\ [\sset(c^{\a}_i)\subset^*a]}$ is an ultrafilter on $\omega$ and it is a rapid P-point (we say that $\mathcal U_{\a}$ is \emph{generated by} $\seq{c_i^{\a}:i<\c}$);
\item\label{i1} for every $\a<\b<\d$, every normal triple $\seq{\pi_1,\psi_1,b_1}$ and every $d\le b_1$ if $\pi''_1\sset(d)\in\cu_{\a}$, then for every $a\in\cu_{\b}$ there is $b\in\cu_{\b}$ such that $b\subset^* a$ and that there are $\pi,\psi\in\o^{\o}$ and $d^*\le_0 d$ so that $\seq{\pi,\psi,d^*}$ is a normal triple, $\pi''\sset(d^*)=b$ and $\forall k\in \sset(d^*)\ [\pi_1(k)=\pi_{\b,\a}(\pi(k))]$.
\item\label{i9} if $\a<\b<\d$, then $\cu_{\b}\nleq_T\cu_{\a}$.
\item\label{i3} for every $\a<\d$, $\pi_{\a,\a}=\id$ and:
\begin{enumerate}
\item\label{i5} $\forall \a\le \b<\d\ \forall i<\c\ [\pi_{\b,\a}''\sset(c^{\b}_{i})\in \mathcal U_{\a}]$;
\item\label{i0} $\forall \a\le \b\le \g<\d\ \exists i<\c\ \forall^{\infty} k\in \sset(c_i^{\g})\ [\pi_{\g,\a}(k)=\pi_{\b,\a}(\pi_{\g,\b}(k))]$;
\item\label{i6} for $\a<\b<\d$ there are $i<\c$, $b_{\b,\a}\in\P$ and $\psi_{\b,\a}\in \o^{\o}$ such that $\seq{\pi_{\b,\a},\psi_{\b,\a},b_{\b,\a}}$ is a normal triple and $c^{\b}_i\le b_{\b,\a}$;
\end{enumerate}
\item\label{i8} if $\mu<\d$ is a limit ordinal such that $\cf(\mu)=\o$, $X\subset \mu$ is such that $\sup(X)=\mu$, $\seq{d_j:j<\o}$ is a decreasing sequence of conditions in $\P$, $\seq{\pi_{\a}:\a\in X}$ is a sequence of maps in $\o^{\o}$ such that:
    \begin{enumerate}
    \item\label{i21} $\forall \a\in X\ \forall j<\o\ [\pi_{\a}''\sset(d_j)\in \mathcal U_{\a}]$;
    \item\label{i22} $\forall \a,\b\in X\ [\a\le \b\sledi \exists j<\o\ \forall^{\infty} k\!\in\!\sset(d_j)\ [\pi_{\a}(k)\!=\!\pi_{\b,\a}(\pi_{\b}(k))]]$;
    \item\label{i23} for all $\a\in X$ there are $j<\o$, $b_{\a}\in \P$ and $\psi_{\a}\in\o^{\o}$ such that $\seq{\pi_{\a},\psi_{\a},b_{\a}}$ is a normal triple and $d_j\le b_{\a}$;
    \end{enumerate}
    then the set of all $i^*<\c$ such that there are $d^*\in \P$ and $\pi\in\o^{\o}$ satisfying:
    \begin{enumerate}[resume]
    \item\label{i31} $\forall j<\o\ [d^*\le d_j]$ and $\sset(c^{\mu}_{i^*})=\pi''\sset(d^*)$;
    \item\label{i32} $\forall \a\in X\ \forall^{\infty} k\in \sset(d^*)\ [\pi_{\a}(k)=\pi_{\mu,\a}(\pi(k))]$;
    \item\label{i34} there is $\psi$ for which $\seq{\pi,\psi,d^*}$ is a normal triple;
    \end{enumerate}
    is cofinal in $\c$;
\end{enumerate}
\end{definition}
When $\CH$ is replaced with $\MA$, the notion of a $\delta$-generic sequence would be defined for every $\delta \leq \c$.
Clause (\ref{i8}) would need to be strengthened by allowing $\mu$ to be any limit ordinal such that $\cf(\mu) < \c$ and by allowing the decreasing sequence of conditions in $\P$ to be of length $\cf(\mu)$. 
\begin{remark}\label{r4}
Suppose $S=\seq{\seq{c^{\a}_i:i<\c\wedge\a<\d},\seq{\pi_{\b,\a}:\a\le\b<\d}}$ is a $\d$-generic sequence for some limit ordinal $\d\le \o_2$.
For every ordinal $\xi<\d$ let $S\rest\xi$ denote $\seq{\seq{c^{\a}_i:i<\c\wedge\a<\xi},\seq{\pi_{\b,\a}:\a\le\b<\xi}}$. We point out that if $S\rest \xi$ is $\xi$-generic for every $\xi<\d$, then $S$ is $\d$-generic. To see this we check conditions (\ref{i2}-\ref{i8}) of Definition \ref{d22}. Conditions (\ref{i2}) and (\ref{i4}) are true because for a fixed $\a<\d$ we can pick $\xi$ so that $\a<\xi<\d$. Then $S\rest\xi$ witnesses that $\cu_{\a}$ and $\seq{c^{\a}_i:i<\c}$ are as needed. For (\ref{i1}), (\ref{i9}) and (\ref{i6}) take $\a<\b<\d$. There is $\xi$ such that $\b<\xi<\d$ and $S\rest\xi$ witnesses (\ref{i1}), (\ref{i9}) and (\ref{i6}). For (\ref{i5}) and (\ref{i0}) take $\a\le\b\le\g<\d$. Again there is $\xi$ so that $\g<\xi<\d$ and $S\rest\xi$ witnesses both (\ref{i5}) and (\ref{i0}). We still have to prove condition (\ref{i8}), so assume that all the objects are given as in (\ref{i8}). In this case we also pick $\xi$ such that $\mu<\xi<\d$. Then $S\rest\xi$ already has all the information about the assumed objects. So $S\rest\xi$ knows that the set of $i^*<\c$ such that $c^{\mu}_{i^*}$ has the required properties is cofinal in $\c$ which implies that (\ref{i8}) is also satisfied in $S$.
\end{remark}

\section{Main lemmas} \label{sec:mainlemmas}
In this section we prove several crucial lemmas that will be used in Section \ref{sec:top} for proving things about the partial order ${\Q}^{\delta}$ to be defined there.
\begin{definition}\label{d44}
Let $\seq{L,\prec}$ be a finite linear order. For each $i<\abs{L}$ let $L(i,\prec)$ denote the $i$th element of $\seq{L,\prec}$. More formally, if $\seq{L,\prec}$ is a finite linear order, then there is a unique order isomorphism $o:\abs{L}\to L$ and $L(i,\prec)=o(i)$ ($i<\abs{L}$).
\end{definition}
Lemma \ref{t7} will be used to prove Lemmas \ref{t30} and \ref{t101}.
It essentially says that the sets ${D}_{m}$ can be broken into intervals of the form ${D}_{m}[{K}_{m, n}, {K}_{m, n + 1})$, for $m \leq n < \omega$, in such a way that whenever $m < m' \leq n$, the image of ${D}_{m'}[{K}_{m', n}, {K}_{m', n + 1})$ under ${\pi}_{m', m}$ comes after everything in the interval ${D}_{m}[{K}_{m, n - 1}, {K}_{m, n})$. 
The use of the elementary submodel $M$ is only for convenience.
We just need a way of saying that each ${D}_{n}$ diagonalizes a ``large enough'' collection of sets from ${\cu}_{n}$.
$M \cap {\cu}_{n}$ is a convenient way to specify this collection.
The use of Lemma \ref{t7} simplifies the presentation of the proofs of Lemmas \ref{t30} and \ref{t101}.
There is no direct analogue of Lemma \ref{t7} when $\CH$ is replaced by $\MA$.
So the proofs of the analogues of Lemmas \ref{t30} and \ref{t101} under $\MA$ will be less elementary. 
\begin{lemma}\label{t7}
Let $\seq{\cu_n:n<\o}$ be a sequence of distinct rapid P-points.
Assume that $\bar \pi=\seq{\pi_{m,n}:n\le m<\o}$ is a sequence of maps in $\o^{\o}$ such that $\pi_{n,n}=\id$ ($n<\o$) and:
\begin{enumerate}
\item\label{i57} $\forall n\le m<\o\ \forall a\in\mathcal U_m\ [\pi''_{m,n}a\in\mathcal U_n]$;
\item\label{i58} $\forall n\le m\le k<\o\ \exists a\in\mathcal U_k\ \forall l\in a\ [\pi_{k,n}(l)=\pi_{m,n}(\pi_{k,m}(l))]$;
\item\label{i59} $\forall n\le m<\o\ \exists a\in\mathcal U_m\ \forall x,y\in a\ [x\le y\sledi \pi_{m,n}(x)\le \pi_{m,n}(y)]$.
\end{enumerate}
Let $\seq{E_n:n<\o}$ be a sequence such that $E_n\in \cu_n$ ($n<\o$). Suppose also that $f\in \o^{\o}$ is increasing and that $M$ is a countable elementary submodel of $H_{\br{2^{\c}}^+}$ containing $\seq{\cu_n:n<\o}$, $\seq{E_n:n<\o}$, $\bar \pi$ and the map $f$. If $\bar D=\seq{D_n:n<\o}$ is a sequence such that $D_n\in \cu_n$ ($n<\o$) and $D_n\subset^* A$, for every $A\in \cu_n\cap M$, then there are sequences $\seq{C_n:n<\o} \in M$, $\seq{F_n:n<\o} \in M$, $\bar K=\seq{K_{m,n}:m\le n<\o} \subset \omega$ and $g' \in {\omega}^{\omega}$ such that for every $n<\o$ we have $\forall m\le n\ [K_{m,n}>0]$, $C_n,F_n\in\cu_n$, $C_n\subset E_n$, $F_n=C_n\cap \pi_{n+1,n}''C_{n+1}$, $\forall m\le m'\le n\ \forall v\in C_n\ [\pi_{n,m}(v)=\pi_{m',m}(\pi_{n,m'}(v))]$, $\forall m\le n\ \forall v,v'\in C_n\ [v\le v'\sledi \pi_{n,m}(v)\le \pi_{n,m}(v')]$, and letting $\rho=\seq{\bar D,\bar K,\bar \pi}$ (see Remark \ref{r1}):
\begin{enumerate}[resume]
\item\label{i51} $\forall m\le n\ \exists z\in F_n\cap F_n(g'(n))\ [\pi_{n,m}(z)=D_m(K_{m,n}-1)]$ and if $n>0$ then $\forall m<n\ [K_{m,n}>K_{m,n-1}]$;
\item\label{i53} $\forall m\le n\ [D_m[K_{m,n}]\subset \pi''_{n,m}\br{F_n\setminus F_n(g'(n))}]$;
\item\label{i54} $\forall m\le n\ \forall v\in C_n\setminus F_n(g'(n))\ [\pi_{n,m}(v)>D_m(K_{m,n}-1)]$;
\item\label{i55} $\forall m<n\ [D_n[K_{n,n}]\cap D_m=0]$ and if $n>0$, then for every $x\in \Delta^{\rho}_{n-1}$, there is a unique $m<n$ such that $x\in D_{m}[K_{m,n-1},K_{m,n})$;
\item\label{i56} if $n>0$, then define ${\prec}_{n - 1}$ to be the collection of all $\langle x, y \rangle$ such that $x, y \in \Delta^{\rho}_{n-1}$ and $\max \set{z\in F_{n-1}:\pi_{n-1,m}(z)=x}<\max\set{z\in F_{n-1}:\pi_{n-1,m'}(z)=y}$ -- where $m,m'<n$ are unique with the property that $x\in D_{m}[K_{m,n-1},K_{m,n})$ and $y\in D_{m'}[K_{m',n-1},K_{m',n})$; then $\prec_{n-1}$ is a linear order on $\Delta^{\rho}_{n-1}$;
\item\label{i52} $2g'(n)\ge L^{\rho}_n$;
\item\label{i60} if $n > 0$, then define the following notation: let ${R}_{n - 1}=\abs{{\Delta}^{\rho}_{n-1}}$ and for each $j < {R}_{n-1}$, let ${x}^{n-1}_{j}$ be ${\Delta}^{\rho}_{n-1}(j,{\prec}_{n-1})$, let $m(n - 1, j)$ be the unique $m < n$ such that ${x}^{n-1}_{j} \in {D}_{m}[{K}_{m,n-1},{K}_{m,n})$, and let
\begin{align*}
{z}^{n-1}_{j}=\max\set{z\in {F}_{n-1}: {\pi}_{n-1,m(n-1,j)}(z)={x}^{n-1}_{j}}; 
\end{align*}
then there exists $l\ge f\br{{L}^{\rho}_{n-1} + j}$ such that ${z}^{n-1}_{j}={E}_{n-1}(l)$.
\end{enumerate}
\end{lemma}

\begin{proof}
First pick a sequence $\bar A=\seq{A_{n,m,k}:n\le m\le k<\o}\in M$ such that $A_{n,m,k}\in \cu_k$ for each $n\le m\le k<\o$ and $\pi_{k,n}(v)=\pi_{m,n}(\pi_{k,m}(v))$ for every $v\in A_{n,m,k}$. Similarly, pick a sequence $\bar B=\seq{B_{n,m}:n\le m<\o}\in M$ so that $B_{n,m}\in \cu_m$, for every $n\le m<\o$ and $\forall x,y\in B_{n,m}\ [x\le y\sledi \pi_{m,n}(x)\le \pi_{m,n}(y)]$. Let us now define sequence $\seq{C'_k:k<\o}$ so that for every $k<\o$ we have $C'_k=E_k\cap \br{\bigcap_{n\le m\le k}A_{n,m,k}}\cap \br{\bigcap_{m\le k}B_{m,k}}$. Note that $C'_k\in\mathcal U_k$ and that $C'_k\subset E_k$ for every $k<\o$. Moreover, $\seq{C'_k:k<\o}\in M$. Next, choose $\bar C=\seq{C_k:k<\o}\in M$ so that for $k<\o$: $C_k\in \cu_k$, $C_k\subset C'_k$ and $C_k(n)\ge E_k(f(2n))$ for every $n<\o$. Let $F_n$ be the set $C_n\cap \pi''_{n+1,n}C_{n+1}$ ($n<\o$). Note that $\seq{F_n:n<\o}$ belongs to $M$. Note also that for each $m\le m'\le n<\o$, $C_n\subset A_{m,m',n}$. So $\forall v\in C_n\ [\pi_{n,m}(v)=\pi_{m',m}(\pi_{n,m'}(v))]$ as required in the statement of the lemma. Moreover if $m\le n$, then $C_n\subset B_{m,n}$, implying that $\forall v,v'\in C_n\ [v\le v'\sledi \pi_{n,m}(v)\le \pi_{n,m}(v')]$ as needed. Then by the definition of sets $D_n$ ($n<\o$) we have $D_n\subset^* \pi_{m,n}''F_m$ for $m\ge n$. Also, for $n<\o$, let $Y_n$ be minimal so that $\forall m<n\ [D_n[Y_n]\cap D_m=0]$.

Now that we have chosen sets $C_n$ and $F_n$ ($n<\o$) we construct, by induction on $n$, numbers $K_{m,n}$ and $g'(n)$ ($m\le n<\o$). First let $l'$ be the least number such that $D_0[l']\subset F_0$ and let $K_{0,0}=l'+1$. Then define $g'(0)$ to be the least $l$ so that $F_0(l)\ge D_0(K_{0,0})$. Note that properties (\ref{i51}-\ref{i60}) hold for $n=0$. Now assume that for every $m\le m'\le n$ we have defined numbers $K_{m,m'}$ and $g'(m')$. We will define $K_{m,n+1}$ ($m\le n+1$) and $g'(n+1)$. So for every $m\le n$ let $X_{m,n+1}$ be the least number such that $X_{m,n+1}\ge K_{m,n}$ and $D_m[X_{m,n+1}]\subset \pi''_{n+1,m}F_{n+1}$, while $X_{n+1,n+1}$ is defined to be minimal such that $D_{n+1}[X_{n+1,n+1}]\subset F_{n+1}=\pi_{n+1,n+1}''F_{n+1}$ and $X_{n+1,n+1}\ge Y_{n+1}$. Put ${x}^{\ast} = L^{\rho}_n + \sum_{m\le n}(X_{m,n+1}-K_{m,n})$. Now define $K_{0,n+1}$ to be the minimal $l$ such that $l > X_{0,n+1} + {x}^{\ast}$ and that $D_0(l-1)\in \bigcap_{m\le n+1}\pi''_{m,0}\br{D_{m}[X_{m,n+1}]}$. Next define $g'(n+1)$ to be the minimal $l\in\o$ with $\pi_{n+1,0}(F_{n+1}(l))\ge D_0(K_{0,n+1})$. Observe that if $v\in F_{n+1}\cap F_{n+1}(g'(n+1))$, then $\pi_{n+1,0}(v)<D_0(K_{0,n+1})$. Also if $v\in F_{n+1}[g'(n+1)]$, then $\pi_{n+1,0}(v)\ge \pi_{n+1,0}(F_{n+1}(g'(n+1)))\ge D_0(K_{0,n+1})$.
Put $G = {F}_{n + 1}\left[g'(n + 1)\right]$ for convenience.
Now for $0<m\le n+1$ define $K_{m,n+1}$ as the minimal $l\in\o$ so that $D_m[l]\subset \pi_{n+1,m}''G$. We remark that $K_{0,n+1}$ is also minimal such that $D_0[K_{0,n+1}]\subset \pi_{n+1,0}''G$. To see this, take any $w\in D_0[K_{0,n+1}]$. Since $K_{0,n+1}\ge X_{0,n+1}$, there exists $v\in F_{n+1}$ with $\pi_{n+1,0}(v)=w$. By the first observation above, $v\notin F_{n+1}(g'(n+1))$. Hence $v\in G$ showing that $D_0[K_{0,n+1}]\subset \pi_{n+1,0}''G$. By the second observation above, there is no $v\in G$ with $\pi_{n+1,0}(v)=D_0(K_{0,n+1}-1)$. Hence there is no $l<K_{0,n+1}$ satisfying $D_0[l]\subset \pi_{n+1,0}''G$.

Now we prove that (\ref{i51}-\ref{i60}) are fulfilled for $n+1$. We begin with the second part of (\ref{i51}). Fix $m\le n+1$. By the definition of $K_{0,n+1}$, there exists $u\in D_m[X_{m,n+1}]$ such that $\pi_{m,0}(u)=D_0(K_{0,n+1}-1)$. We claim $u<D_m(K_{m,n+1})$. Suppose not. Then $u\in D_m[K_{m,n+1}]$, and so $u=\pi_{n+1,m}(v)$, for some $v\in G$. However $\pi_{n+1,0}(v)=\pi_{m,0}(\pi_{n+1,m}(v))=\pi_{m,0}(u)=D_0(K_{0,n+1}-1)$, contradicting an observation of the previous paragraph. Thus $u<D_m(K_{m,n+1})$, showing that $X_{m,n+1}<K_{m,n+1}$, for all $m\le n+1$.

Before proving the rest of (\ref{i51}), (\ref{i53}), and (\ref{i54}) we make some useful observations. Put $y_m=D_m(K_{m,n+1}-1)$, for each $m\le n+1$. Let $m\le n+1$ be fixed. First since $K_{m,n+1}$ is minimal so that $D_m[K_{m,n+1}]\subset \pi_{n+1,m}''G$, there is no $v\in G$ with $\pi_{n+1,m}(v)=y_m$. Next let $u\in D_m[K_{m,n+1}]$. Then there exists $v\in G$ with $u=\pi_{n+1,m}(v)$, and $\pi_{m,0}(u)=\pi_{m,0}(\pi_{n+1,m}(v))=\pi_{n+1,0}(v)\ge D_0(K_{0,n+1})$. Thus $\pi_{m,0}(u)\ge D_0(K_{0,n+1})$, for every $u\in D_m[K_{m,n+1}]$. For the final observation, consider some $v\in F_{n+1}\cap F_{n+1}(g'(n+1))$. As pointed out before, $\pi_{n+1,0}(v)<D_0(K_{0,n+1})$. Now let $u=\pi_{n+1,m}(v)$. Then $\pi_{m,0}(u)=\pi_{n+1,0}(v)<D_0(K_{0,n+1})$. Applying the previous observation to $u$, we conclude that $\pi_{n+1,m}(v)\notin D_m[K_{m,n+1}]$, for every $v\in F_{n+1}\cap F_{n+1}(g'(n+1))$.

Now the rest of (\ref{i51}), (\ref{i53}), and (\ref{i54}) easily follow from the three observations in the previous paragraph. For the first part of (\ref{i51}), $y_m\in D_m[X_{m,n+1}]$, and so there is $v_m\in F_{n+1}$ with $y_m=\pi_{n+1,m}(v_m)$. By the first observation, $v_m\notin G$. Hence $v_m\in F_{n+1}\cap F_{n+1}(g'(n+1))$, as needed. For (\ref{i53}), let $u\in D_m[K_{m,n+1}]$. Then $u\in D_m[X_{m,n+1}]$ and so there is $v\in F_{n+1}$ with $\pi_{n+1,m}(v)=u$. By the third observation, $v\notin F_{n+1}(g'(n+1))$, as required. For (\ref{i54}), first note that since $v_m,F_{n+1}(g'(n+1))\in B_{m,n+1}$, $\pi_{n+1,m}(F_{n+1}(g'(n+1)))\ge y_m$, and by the first observation, $\pi_{n+1,m}(F_{n+1}(g'(n+1)))>y_m$. Now let $v\in C_{n+1}$. Then $v\in B_{m,n+1}$ and if $v\ge F_{n+1}(g'(n+1))$, then $\pi_{n+1,m}(v)>y_m$, implying (\ref{i54}).

For (\ref{i55}) we have $Y_{n+1}\le X_{n+1,n+1}<K_{n+1,n+1}$, and so $D_{n+1}[K_{n+1,n+1}]\cap D_m=0$, for all $m<n+1$. The second part of (\ref{i55}) easily follows from the definition of $\Delta^{\rho}_n$ and from the induction hypotheses.

For (\ref{i56}), first consider any $x\in \Delta^{\rho}_n$. By (\ref{i55}) applied to $n+1$, let $m<n+1$ be unique so that $x\in D_m[K_{m,n},K_{m,n+1})$. By (\ref{i53}) applied to $n$, there is $z\in F_n$ with $\pi_{n,m}(z)=x$. Also $F_n\subset C_n$ and so $\pi_{n,m}$ is finite-to-one on $F_n$. Therefore $\max\set{z\in F_n: \pi_{n,m}(z)=x}$ is well-defined. Next it is clear that $\prec_n$ is transitive and irreflexive. We check that it is total. Let $x,y\in \Delta^{\rho}_n$ and let $m,m'<n+1$ be unique so that $x\in D_m[K_{m,n},K_{m,n+1})$ and $y\in D_{m'}[K_{m',n},K_{m',n+1})$. We may assume $m\le m'$. If $x$ and $y$ are incomparable under $\prec_n$, then there exists $z\in F_n$ so that $\pi_{n,m}(z)=x$ and $\pi_{n,m'}(z)=y$. As $F_n\subset C_n$, $\pi_{m',m}(y)=\pi_{m',m}(\pi_{n,m'}(z))=\pi_{n,m}(z)=x$. If $m<m'$, then this contradicts the fact that $x\in H^{\rho}_{m,n}$. Therefore $m=m'$, and since $\pi_{m',m}=\id$, $x=y$, implying comparability.

Now we check (\ref{i52}). For each $m\!<\!n+1$ define $H^{\rho}_{m,n,0}\!=\!H^{\rho}_{m,n}\cap D_m[K_{m,n},X_{m,n+1})$ and $H^{\rho}_{m,n,1}=H^{\rho}_{m,n}\cap D_m[X_{m,n+1},K_{m,n+1})$. Define ${x}_{0} = \abs{\bigcup_{m\le n}H^{\rho}_{m,n,0}}$ and ${x}_{1} = \abs{\bigcup_{m\le n}H^{\rho}_{m,n,1}}$. It is clear that $H^{\rho}_{m,n}=H^{\rho}_{m,n,0}\cup H^{\rho}_{m,n,1}$ and that $L^{\rho}_{n+1}\le L^{\rho}_n + {x}_{0} + {x}_{1}$. Also $L^{\rho}_n + {x}_{0} \le {x}^{\ast}$. So to prove (\ref{i52}) it is enough to show both ${x}^{\ast} \le g'(n+1)$ and ${x}_{1} \le g'(n+1)$. For the first inequality, note that $\abs{D_0[X_{0,n+1},K_{0,n+1})} \ge {x}^{\ast}$. For each $u\in D_0[X_{0,n+1},K_{0,n+1})$, there exists $v\in F_{n+1}$ with $\pi_{n+1,0}(v)=u$. By (\ref{i54}) applied to $n+1$, $v\in F_{n+1}(g'(n+1))$. It follows that $g'(n + 1) \ge {x}^{\ast}$. For the second inequality, note first that for each $m\le n$ and $u\in H^{\rho}_{m,n,1}$ we get a $v\in F_{n+1}\cap F_{n+1}(g'(n+1))$ with $\pi_{n+1,m}(v)=u$ by applying the same argument. Now suppose $u\neq u'$, $m,m'\le n$, $u\in H^{\rho}_{m,n,1}$, $u'\in H^{\rho}_{m',n,1}$, $v,v'\in F_{n+1}\cap F_{n+1}(g'(n+1))$, $\pi_{n+1,m}(v)=u$, and $\pi_{n+1,m'}(v')=u'$. We would like to see that $v\neq v'$. Suppose not. We may assume $m\le m'$. Since $v\in C_{n+1}$, $u=\pi_{n+1,m}(v)=\pi_{m',m}(\pi_{n+1,m'}(v))=\pi_{m',m}(u')$. If $m<m'$, then this contradicts the fact that $u\in H^{\rho}_{m,n}$. Hence $m=m'$, whence $u=u'$. This is a contradiction which shows that $v\neq v'$. It follows that $g'(n + 1) \ge {x}_{1}$ as needed.

Finally we come to (\ref{i60}). Fix $j<R_n$. By (\ref{i53}) applied to $n$, $z^n_j=F_n(t^n_j)$ for some $t^n_j\ge g'(n)$. Since $\prec_n$ is a linear order $t^n_j\ge g'(n)+j$. By (\ref{i52}) applied to $n$, $2t^n_j\ge 2g'(n)+2j\ge L^{\rho}_n+j$. Now $F_n\subset C_n\subset C'_n\subset E_n$ and $F_n(t^n_j)\ge C_n(t^n_j)\ge E_n(f(2t^n_j))\ge E_n(f(L^{\rho}_n+j))$ because $f$ is an increasing function. It follows that $z^n_j=E_n(l^n_j)$ for some $l^n_j\ge f(L^{\rho}_n+j)$, as needed.
\end{proof}

\begin{remark}\label{r3}
Note that for each $n < \o$, $H^{\rho}_{n,n} = D_n[K_{n,n},K_{n,n+1})$; so $\Delta^{\rho}_n \neq 0$, and so $L^{\rho}_{n+1}>L^{\rho}_n$.
Note also that ${z}^{n}_{j} \geq {F}_{n}(g'(n))$, for each $n < \omega$ and $j < {R}_{n}$.
\end{remark}

Lemma \ref{t30} will play an important role throughout the next section.
It is essential to the proof that ${\Q}^{\delta}$, which will be defined in Definition \ref{d14}, is countably closed.
It is also used in ensuring that ${\cu}_{\delta}$ is a rapid ultrafilter and that ${\cu}_{\delta}$ satisfies (\ref{i1}) and (\ref{i8}) of Definition \ref{d22}.
\begin{lemma}\label{t30}
Assume that $\d<\o_2$, $\cf(\d)=\o$, $f\in \o^{\o}$ is increasing, $X\subset \d$ is such that $\sup(X)=\d$ and:

\begin{enumerate}
\item\label{i71} the sequence $S\!=\!\seq{\seq{c^{\a}_i:i<\c\wedge\a<\d},\seq{\pi_{\b,\a}:\a\le \b<\d}}$ is $\d$-generic;
\item\label{i72} there are $e\in\P$ and mappings $\seq{\pi_{\d,\a}:\a\in X}$ such that:
\begin{enumerate}
\item\label{i73} $\forall \a\in X\ [\pi_{\d,\a}''\sset(e)\in \mathcal U_{\a}]$;
\item\label{i74} $\forall \a,\b\in X\ [\a\le\b\sledi\forall^{\infty}k\in \sset(e)\ [\pi_{\d,\a}(k)=\pi_{\b,\a}(\pi_{\d,\b}(k))]]$;
\item\label{i75} for all $\a\in X$ there are $b_{\d,\a}$ and $\psi_{\d,\a}$ such that $\seq{\pi_{\d,\a},\psi_{\d,\a},b_{\d,\a}}$ is a normal triple and $e\le b_{\d,\a}$;
\end{enumerate}
\item\label{i76} there is a decreasing sequence $\seq{d_j:j<\o}$ of elements of $\P$ and a sequence of mappings $\seq{\pi_{\a}:\a\in X}$ such that:
\begin{enumerate}
\item\label{i77} $\forall \a\in X\ \forall j<\o\ [\pi_{\a}''\sset(d_j)\in \mathcal U_{\a}]$;
\item\label{i78} $\forall \a,\b\in X\ [\a\le\b\sledi\exists j<\o\ \forall^{\infty} k\!\in\!\sset(d_j)\ [\pi_{\a}(k)\!=\!\pi_{\b,\a}(\pi_{\b}(k))]]$;
\item\label{i79} for all $\a\in X$ there are $j<\o$, $\psi_{\a}$ and $b_{\a}$ such that $\seq{\pi_{\a},\psi_{\a},b_{\a}}$ is a normal triple and $d_j\le b_{\a}$.
\end{enumerate}
\end{enumerate}
Then there are $d^*,e^*\in \P$ and $\pi:\o\to\o$ such that:
\begin{enumerate}[resume]
\item\label{i84} $\forall n<\o\ \exists m\ge f(n)\ [e^*(n)\subset e(m)]$;
\item\label{i80} $\forall j<\o\ [d^*\le d_j]$ and $\sset(e^*)=\pi''\sset(d^*)$;
\item\label{i81} $\forall \a\in X\ \forall^{\infty}k\in \sset(d^*)\ [\pi_{\a}(k)=\pi_{\d,\a}(\pi(k))]$;
\item\label{i82} $\forall \a\in X\ [\pi_{\d,\a}''\sset(e^*)\in \mathcal U_{\a}]$;
\item\label{i83} there is $\psi\in\o^{\o}$ for which $\seq{\pi,\psi,d^*}$ is a normal triple.
\end{enumerate}
\end{lemma}

\begin{proof}
Let $\seq{\d_n:n<\o}\subset X$ be an increasing and cofinal sequence in $\d$. For $m\le n<\o$ choose $j_{m,n}$ and $L^d(m,n)$ so that $\forall k\!\in\! \sset(d_{j_{m,n}})\lp L^d(m,n)\rp\ [\pi_{\d_m}(k)=\pi_{\d_n,\d_m}(\pi_{\d_n}(k))]$. For every $n<\o$ pick $j(n)$, $b_{\d_n}$ and $\psi_{\d_n}$ such that $\seq{\pi_{\d_n},\psi_{\d_n},b_{\d_n}}$ is a normal triple and $d_{j(n)}\le b_{\d_n}$ holds by using (3c). Let $K^d(n)$ be minimal such that $d_{j(n)}\le_{K^d(n)}b_{\d_n}$. Define a strictly increasing sequence $\seq{j_N:N<\o}$ by setting $j_N=\max(\set{j(N)}\cup\set{j_k+1:k\!<\!N}\cup\set{j_{m,N}:m\!\le\! N})$. Let $Q^d(N)$ be minimal such that $d_{j_N}\!\le_{Q^d(N)}d_{j(N)}$, $\forall k<N\ [d_{j_N}\le_{Q^d(N)}d_{j_k}]$ and $\forall m\le N\ [d_{j_N}\le_{Q^d(N)}d_{j_{m,N}}]$. Define $M^d_N\!=\!\max\big(\set{K^d(N),Q^d(N)}\cup\set{L^d(m,N):m\le N}\cup\set{M^d_k:k<N}\big)$.

For each $n<\o$ let $K^e(n)$ be minimal such that $e\le_{K^e(n)} b_{\d,\d_n}$. For $m\le n<\o$ let $L^e(m,n)$ be minimal such that $\forall k\in \sset(e)\lp L^e(m,n)\rp\ [\pi_{\d,\d_m}(k)=\pi_{\d_n,\d_m}(\pi_{\d,\d_n}(k))]$. For $N\!<\!\o$ let $M^e_N\!=\!\max(\set{K^e(N)}\cup\set{L^e(m,N):m\le N}\cup\set{M^e_k:k<N})$.

The proof of the following claim is simple so we leave it to the reader.

\begin{claim}\label{t41}
Let $N<\o$. The following hold:
\begin{enumerate}
\item\label{i01} $\forall m\le n\le N\ \forall k\in \sset(e)\lp M^e_N\rp\ [\pi_{\d,\d_m}(k)=\pi_{\d_n,\d_m}(\pi_{\d,\d_n}(k))]$;
\item\label{i02} $\forall m\le n\le N\ \forall k\in \sset(d_{j_N})\lp M^d_N\rp\ [\pi_{\d_m}(k)=\pi_{\d_n,\d_m}(\pi_{\d_n}(k))]$;
\item\label{i03} $\forall n\le N\ \forall k,l\in \sset(e)\lp M^e_N\rp\ [k\le l\sledi \pi_{\d,\d_n}(k)\le \pi_{\d,\d_n}(l)]$;
\item\label{i04} $\forall n\le N\ \forall k,l\in \sset(d_{j_N})\lp M^d_N\rp\ [k\le l\sledi \pi_{\d_n}(k)\le \pi_{\d_n}(l)]$.
\end{enumerate}
\end{claim}

For every $n<\o$ let $E_n=\pi''_{\d,\d_n}\sset(e)\lp M^e_n\rp\cap\pi''_{\d_n}\sset(d_{j_n})\lp M^d_n\rp\in \cu_{\d_n}$. Let $M$ be a countable elementary submodel of $H_{\br{2^{\c}}^+}$ containing $S$, $\d$, $f$, and sequences $\seq{E_{n}:n<\o}$ and $\seq{\d_n:n<\o}$. For $n<\o$ choose sets $D_n\in \cu_{\d_n}$ as follows: $D_n\subset^* A$ for every $A\in \cu_{\d_n}\cap M$. Note that these sets exist because $\cu_{\d_n}$ is a P-point and $M$ is countable. Define $g\in \o^{\o}$ by $g(n)=\max\set{f(n),t(n),s(n+1)}$, for each $n\in\o$. Note $g\in M$ and that $g$ is increasing. Now Lemma \ref{t7} applies to the sequences $\seq{\cu_{\d_n}:n<\o}$, $\bar\pi=\seq{\pi_{\d_n,\d_m}:m\le n<\o}$, $\seq{E_n:n<\o}$, $\bar D=\seq{D_n:n<\o}$ and the function $g$. Let $\seq{C_n:n<\o}$, $\seq{F_n:n<\o}$, $\bar K=\seq{K_{m,n}:m\le n<\o}$ and $\seq{g'(n):n<\o}$ be as in the conclusion of Lemma \ref{t7}. We denote $\rho=\seq{\bar D,\bar K,\bar \pi}$, numbers $m(n,j)$, numbers $R_n$ and numbers $z^n_j$ as in the conclusion of Lemma \ref{t7}.

At this point, for every $n<\o$, we define set $I_n=\set{m<\o:L^{\rho}_n\le m<L^{\rho}_{n+1}}$. Clearly, $\set{I_n:n<\o}$ is a partition of $\o$. We also have $I_n=\set{L^{\rho}_n+j:j<R_n}$. So every $k<\o$ is of the form $L^{\rho}_n+j$ for some $n<\o$ and $j<R_n$. For each $n<\o$ and $j<R_n$, (\ref{i60}) of Lemma \ref{t7} implies that $z^n_j=E_n(l^n_j)$ for some $l^n_j\ge g(L^{\rho}_n+j)$. Now for a fixed $n<\o$, $\seq{\pi_{\d,\d_n},\psi_{\d,\d_n},b_{\d,\d_n}}$ and $\seq{\pi_{\d_n},\psi_{\d_n},b_{\d_n}}$ are normal triples, $e\le_{M^e_n}b_{\d,\d_n}$, $d_{j_n}\le_{M^d_n}b_{\d_n}$, $E_n\subset \pi_{\d,\d_n}''\sset(e)\lp M^e_n\rp$, and $E_n\subset \pi_{\d_n}''\sset(d_{j_n})\lp M^d_n\rp$. So Lemma \ref{t14} applies and implies that for each $l<\o$, $\zeta_l=\max\!\set{\!m\!<\!\o\!:\!\pi_{\d,\d_n}''e(m)\!=\!\set{E_n(l)}}$ and $\z_l\!=\!\max\!\set{\!m\!<\!\o\!:\!\pi_{\d_n}''d_{j_n}(m)\!=\!\set{E_n(l)}}$ are well-defined, that $\zeta_l<\zeta_{l+1}$, $\z_l<\z_{l+1}$, and that $\zeta_l\ge M^e_n$, $\z_l\ge M^d_n$. Now for each $j<R_n$, define $\zeta^n_j=\zeta_{l^n_j}$ and $\z^n_j=\z_{l^n_j}$. It follows that $\zeta^n_j,\z^n_j\ge g(L^{\rho}_n+j)$. Also if $j<j+1<R_n$, then $z^n_j<z^n_{j+1}$, and so $\zeta^n_j<\zeta^n_{j+1}$ and $\z^n_j<\z^n_{j+1}$. Hence for $j<R_n$, $\abs{e(\zeta^n_j)}\ge g(L^{\rho}_n+j) \ge L^{\rho}_n+j+1$ and $\abs{d_{j_n}(\z^n_j)}\ge g(L^{\rho}_n+j)\ge t(L^{\rho}_n+j)$. Now unfix $n$. For $k<\o$ write $k=L^{\rho}_n+j$ and pick arbitrary $e^*(k)\in [e(\zeta^n_j)]^{k+1}$. Note that this choice of $e^*(k)$ ($k<\o$) ensures that $e^*(k)\subset e(m)$ for $m\ge f(k)$ as required in the statement of the lemma. Similarly for $l<\o$ find the unique $m$ such that $s(m)\le l<s(m+1)$ and write $m=L^{\rho}_n+j$. Pick $d^*(l)\!\in\![d_{j_n}(\z^n_j)]^{l+1}$ in such a way that for $s(m)\!\le\! l\!<\!l+1\!<\!s(m+1)$ we have $\max(d^*(l))\!<\!\min (d^*(l+1))$. This is possible by already proved $\abs{d_{j_n}(\z^n_j)}\!\ge\! t(m)$.

\begin{claim}\label{t8}
For every $k<\o$ we have
\[\TS
\max(e^*(k))<\min(e^*(k+1))\ \mbox{and}\ \max(d^*(k))<\min(d^*(k+1)).
\]
\end{claim}

\begin{proof}
For a fixed $n<\o$, if $j<j+1<R_n$, then $\zeta^n_j<\zeta^n_{j+1}$ and $\z^n_j<\z^n_{j+1}$ and so $\max(e(\zeta^n_j))<\min(e(\zeta^n_{j+1}))$ and $\max(d_{j_n}(\z^n_j))<\min(d_{j_n}(\z^n_{j+1}))$. Moreover, for a fixed $n<\o$ and $j<R_n$, if $s(L^{\rho}_n+j)\le l<l+1<s(L^{\rho}_n+j+1)$, then $\max(d^*(l))<\min(d^*(l+1))$ by definition. Therefore it suffices to show that for $n<\o$, $j<R_n$, and $j'<R_{n+1}$, $\max(e(\zeta^n_j))<\min(e(\zeta^{n+1}_{j'}))$ and $\max(d_{j_n}(\z^n_j))<\min(d_{j_{n+1}}(\z^{n+1}_{j'}))$.

To see the first inequality, we argue by contradiction. Suppose $y\in e(\zeta^n_j), y'\in e(\zeta^{n+1}_{j'})$, and $y'\le y$. As noted above, $\zeta^{n+1}_{j'}\ge M^e_{n+1}$. Thus $y,y'\in \sset(e)\lp M^e_{n+1}\rp$, and so by (\ref{i01}) and (\ref{i03}) of Claim \ref{t41},
\begin{align*}
&\pi_{\d_{n+1},\d_{m(n,j)}}(z^{n+1}_{j'}) = \pi_{\d_{n+1},\d_{m(n,j)}}(\pi_{\d,\d_{n+1}}(y'))=\pi_{\d,\d_{m(n,j)}}(y')\le \pi_{\d,\d_{m(n,j)}}(y) = \\
&\pi_{\d_n,\d_{m(n,j)}}(\pi_{\d,\d_n}(y))=\pi_{\d_n,\d_{m(n,j)}}(z^n_j)=x^n_j\le D_{m(n,j)}(K_{m(n,j),n+1}-1).
\end{align*}
However ${z}^{n + 1}_{j'} \in {C}_{n + 1} \setminus {F}_{n + 1}(g'(n + 1))$. But then by (\ref{i54}) of Lemma \ref{t7} applied to $n+1$ and $m=m(n,j)$, $\pi_{\d_{n+1},\d_{m(n,j)}}(z^{n+1}_{j'})>D_{m(n,j)}(K_{m(n,j),n+1}-1)$. This is a contradiction which proves the first inequality.

The second inequality is also proved by contradiction. So suppose $y\in d_{j_n}(\z^n_j)$, $y'\in d_{j_{n+1}}(\z^{n+1}_{j'})$, and $y'\le y$. As noted above, $\z^n_j\ge M^d_n$ and $\z^{n+1}_{j'}\ge M^d_{n+1}$. Moreover $d_{j_{n+1}}\le_{Q^d(n+1)}d_{j_n}$, $M^d_{n+1}\ge Q^d(n+1)$, and $M^d_{n+1}\ge M^d_n$. So there exists $l\ge M^d_n$ with $d_{j_{n+1}}(\z^{n+1}_{j'})\subset d_{j_n}(l)$. Thus $y'\in \sset(d_{j_{n+1}})\lp M^d_{n+1}\rp$ and $y,y'\!\in\! \sset(d_{j_n})\lp M^d_n\rp$. Therefore by (\ref{i02}) and (\ref{i04}) of Claim \ref{t41},
\begin{align*}
&\pi_{\d_{n+1},\d_{m(n,j)}}(z^{n+1}_{j'})\!=\pi_{\d_{n+1},\d_{m(n,j)}}(\pi_{\d_{n+1}}(y'))=\pi_{\d_{m(n,j)}}(y')\le \pi_{\d_{m(n,j)}}(y)=\\
&\pi_{\d_n,\d_{m(n,j)}}(\pi_{\d_n}(y))=\pi_{\d_n,\d_{m(n,j)}}(z^n_j)=x^n_j\le D_{m(n,j)}(K_{m(n,j),n+1}-1). 
\end{align*}
However $\pi_{\d_{n+1},\d_{m(n,j)}}(z^{n+1}_{j'})>D_{m(n,j)}(K_{m(n,j),n+1}-1)$ as pointed out in the previous paragraph. This is a contradiction which completes the proof.
\end{proof}

So for now we have settled that $e^*\le_0 e$ and that for every $n<\o$ there is $m\ge f(n)$ such that $e^*(n)\subset e(m)$. Define $\pi:\o\to\o$ as follows: for every $k\in \o\setminus \sset(d^*)$ let $\pi(k)=0$, while for every $k\in\sset(d^*)$ let $m$ be unique such that $k\in d^*(m)$ and define $\pi(k)=\sset(e^*)(m)$.

\begin{claim}\label{t92}
The sequences $\seq{e^*(n):n<\o}$ and $\seq{d^*(n):n<\o}$ belong to $\P$ and satisfy the following conditions:
\begin{enumerate}
\item\label{i91} $\forall j<\o\ [d^*\le d_j]$ and $\sset(e^*)=\pi''\sset(d^*)$;
\item\label{i92} $\forall \a\in X\ \forall^{\infty} k\in \sset(d^*)\ [\pi_{\a}(k)=\pi_{\d,\a}(\pi(k))]$;
\item\label{i93} $\forall \a\in X\ [\pi_{\d,\a}''\sset(e^*)\in \cu_{\a}]$;
\item\label{i94} There is $\psi$ such that $\seq{\pi,\psi,d^*}$ is a normal triple.
\end{enumerate}
\end{claim}

\begin{proof}
First note that $e^*$ and $d^*$ belong to $\P$. Next, we prove (\ref{i91}). It suffices to prove that $d^*\le d_{j_n}$, for all $n<\o$. Fix $n<\o$. Take any $l\ge s(L^{\rho}_n)$. Let $m$ be such that $s(m)\le l<s(m+1)$ and $n'$ and $j'<R_{n'}$ such that $m=L^{\rho}_{n'}+j'$. Note $m\ge L^{\rho}_n$ and $n'\ge n$. Then $d^*(l)\subset d_{j_{n'}}(\z^{n'}_{j'})$ and $\z^{n'}_{j'}\ge M^d_{n'}$. As noted earlier $\z^{n'}_{j'}\ge g(L^{\rho}_{n'}+j')=g(m)\ge s(m+1)\ge l$. So since $d_{j_{n'}}\le_{M^d_{n'}}d_{j_n}$, there is $l'\ge \z^{n'}_{j'}\ge l$ so that $d^*(l)\subset d_{j_{n'}}(\z^{n'}_{j'})\subset d_{j_n}(l')$, showing $d^*\le_{s(L^{\rho}_n)}d_{j_n}$. To see that $\sset(e^*)=\pi''\sset(d^*)$ note that by the definition of $\pi$ we have that for every $n<\o$ holds $\pi''d^*(n)=\set{\sset(e^*)(n)}$. So (\ref{i91}) is proved.

Now we prove (\ref{i92}). First we prove $\forall^{\infty}k\in \sset(d^*)\ [\pi_{\d_n}(k)=\pi_{\d,\d_n}(\pi(k))]$, for every $n<\o$. Fix $n<\o$ and consider any $n'\ge n$ and $j'<R_{n'}$. It suffices to show that for any $s(L^{\rho}_{n'}+j')\le l<s(L^{\rho}_{n'}+j'+1)$ and $k\in d^*(l)$, $\pi_{\d_n}(k)=\pi_{\d,\d_n}(\pi(k))$. By definition $d^*(l)\subset d_{j_{n'}}(\z^{n'}_{j'})$ and $\pi(k)\in e^*(L^{\rho}_{n'}+j')\subset e(\zeta^{n'}_{j'})$. Therefore $\pi_{\d_{n'}}(k)=z^{n'}_{j'}=\pi_{\d,\d_{n'}}(\pi(k))$. Also $k\in \sset(d_{j_{n'}})\lp M^d_{n'}\rp$ and $\pi(k)\in \sset(e)\lp M^e_{n'}\rp$ because $\zeta^{n'}_{j'}\ge M^e_{n'}$ and $\z^{n'}_{j'}\ge M^d_{n'}$. Thus by (\ref{i01}) and (\ref{i02}) of Claim \ref{t41}, $\pi_{\d_n}(k)=\pi_{\d_{n'},\d_n}(\pi_{\d_{n'}}(k))=\pi_{\d_{n'},\d_n}(\pi_{\d,\d_{n'}}(\pi(k)))=\pi_{\d,\d_n}(\pi(k))$, as needed. For the more general claim fix $\a\in X$ and find $n<\o$ so that $\a\le \d_n$. By (\ref{i78}) of Lemma \ref{t30}, there exist $i<\o$ and $L_0$ so that $\forall k\in \sset(d_i)\lp L_0\rp\ [\pi_{\a}(k)=\pi_{\d_n,\a}(\pi_{\d_n}(k))]$. Let $L_1$ be minimal with $d^*\le_{L_1}d_i$. By (\ref{i74}) of Lemma \ref{t30} and by the fact that $\sset(e^*)\subset \sset(e)$, there is $L_2$ so that $\forall k\in \sset(e^*)[L_2][\pi_{\d,\a}(k)=\pi_{\d_n,\a}(\pi_{\d,\d_n}(k))]$. Let $L_3$ be so that $\forall k\in \sset(d^*)\lp L_3\rp\ [\pi_{\d_n}(k)=\pi_{\d,\d_n}(\pi(k))]$. Let $L=\max\set{L_0,L_1,L_2,L_3}$. If $k\in \sset(d^*)\lp L\rp$, then $k\in \sset(d_i)\lp L_0\rp$ and $\pi(k)\in \sset(e^*)[L_2]$. So $\pi_{\a}(k)=\pi_{\d_n,\a}(\pi_{\d_n}(k))=\pi_{\d_n,\a}(\pi_{\d,\d_n}(\pi(k)))=\pi_{\d,\a}(\pi(k))$. Thus $\forall k\in \sset(d^*)\lp L\rp\ [\pi_{\a}(k)=\pi_{\d,\a}(\pi(k))]$, proving (\ref{i92}).

Now we come to (\ref{i93}). We first show that for each $m<\o$, $D_m\subset^*\pi_{\d,\d_m}''\sset(e^*)$. Fix $m<\o$. As $\seq{K_{m,n}:m\le n<\o}$ is strictly increasing with $n$, it suffices to show that for each $n\ge m$, $D_m[K_{m,n},K_{m,n+1})\subset \pi_{\d,\d_m}''\sset(e^*)$. Let $n\ge m$ and $u\in D_m[K_{m,n},K_{m,n+1})$ be given. Put $m'=$
\begin{align*}
\max\{m''\le n:m\le m''\ \mbox{and}\ \exists u''\in D_{m''}[K_{m'',n},K_{m'',n+1})\ [\pi_{\d_{m''},\d_m}(u'')=u]\},
\end{align*}
and choose $u'\in D_{m'}[K_{m',n},K_{m',n+1})$ with $\pi_{\d_{m'},\d_m}(u')=u$. We claim that $u'\in H^{\rho}_{m',n}$. Suppose not. Then there exist $m'<m''\le n$ and $u''\in D_{m''}[K_{m'',n},K_{m'',n+1})$ with $\pi_{\d_{m''},\d_{m'}}(u'')=u'$. Now $u''\in D_{m''}[K_{m'',m''}]$ because $K_{m'',m''}\le K_{m'',n}$. So by (\ref{i53}) of Lemma \ref{t7} applied to $m''$, $u''\in F_{m''}\subset C_{m''}$. By one of the properties of $C_{m''}$ listed in Lemma \ref{t7}, $\pi_{\d_{m''},\d_m}(u'')=\pi_{\d_{m'},\d_m}(\pi_{\d_{m''},\d_{m'}}(u''))=\pi_{\d_{m'},\d_m}(u')=u$. However this contradicts the choice of $m'$. Thus $u'\in H^{\rho}_{m',n}\subset \Delta^{\rho}_n$. So let $j<R_n$ be so that $u'=x^n_j$. Note that $m(n,j)=m'$. Also $z^n_j\in F_n\subset C_n$, and so $\pi_{\d_n,\d_m}(z^n_j)=\pi_{\d_{m'},\d_m}(\pi_{\d_n,\d_{m'}}(z^n_j))=\pi_{\d_{m'},\d_m}(u')=u$. Now if $k\in e^*({L}^{\rho}_{n} + j)\subset e(\zeta^n_j)$, then $k\in \sset(e)\lp M^e_n\rp$ because $\zeta^n_j\ge M^e_n$. So by (\ref{i01}) of Claim \ref{t41}, $\pi_{\d,\d_m}(k)=\pi_{\d_n,\d_m}(\pi_{\d,\d_n}(k))=\pi_{\d_n,\d_m}(z^n_j)=u$, showing that $u\in \pi_{\d,\d_m}''\sset(e^*)$. This concludes the proof that $D_m\subset^* \pi_{\d,\d_m}''\sset(e^*)$. As $D_m\in \cu_{\d_m}$, this shows that $\pi_{\d,\d_m}''\sset(e^*)\in \cu_{\d_m}$, for all $m \in \omega$. Now for the more general statement, fix $\a\in X$. Find $m\in \o$ with $\d_m\ge \a$. By (\ref{i74}) of Lemma \ref{t30}, there is $L$ so that $\forall k\in \sset(e^*)[L]\ [\pi_{\d,\a}(k)=\pi_{\d_m,\a}(\pi_{\d,\d_m}(k))]$. Put $A=\sset(e^*)[L]$. Since $\pi_{\d,\d_m}''A\in \cu_{\d_m}$, $\pi_{\d_m,\a}''\pi_{\d,\d_m}''A\in \cu_{\a}$. Hence $\pi_{\d_m,\a}''\pi_{\d,\d_m}''A\subset \pi_{\d,\a}''A\subset \pi_{\d,\a}''\sset(e^*)$, implying $\pi_{\d,\a}''\sset(e^*)\in \cu_{\a}$, which proves (\ref{i93}).

For the proof of (\ref{i94}), consider function $\psi:\o\to\o$ defined in the following way: for $k<\o$ let $\psi(k)=\sset(e^*)(k)$. It is clear that $\seq{\pi,\psi,d^*}$ is a normal triple.
\end{proof}

The last claim proves the lemma.
\end{proof}
When $\CH$ is replaced by $\MA$, the statement of Lemma \ref{t30} needs to be generalized as follows.
$\delta$ is allowed to be any ordinal with $\cf(\delta) < \c$, and the decreasing sequence $\langle {d}_{j}: j < \omega \rangle$ is replaced by the decreasing sequence $\langle {d}_{j}: j < \cf(\delta) \rangle$.
This version can be proved under $\MA$ by taking a suitably generic filter over a poset consisting of finite approximations to ${d}^{\ast}$, ${e}^{\ast}$, and $\pi$ together with some finite side conditions.
The exact definition of this poset can be formulated by examining the proofs of Lemmas \ref{t7} and \ref{t30}.

Lemma \ref{t101} will be used in the proof that the poset ${\Q}^{\delta}$ is countably closed.
The requirement in Lemma \ref{t7} that $F_n = C_n \cap \pi_{n+1,n}''C_{n+1}$ will be used crucially in this proof.
\begin{lemma}\label{t101}
Let $\seq{\mathcal U_n:n<\o}$ be a sequence of distinct rapid P-points.
Assume that $\seq{\pi_{n,m}:m\le n<\o} \subset \BS$ is a sequence so that $\pi_{n,n}=\id$ ($n<\o$) and:
\begin{enumerate}[series=elemma]
\item\label{i101} $\forall m\le n<\o\ \forall a\in\mathcal U_n\ [\pi''_{n,m}a\in\mathcal U_m]$;
\item\label{i102} $\forall m\le n\le k<\o\ \exists a\in\mathcal U_k\ \forall l\in a\ [\pi_{k,m}(l)=\pi_{n,m}(\pi_{k,n}(l))]$;
\item\label{i103} $\forall m\le n<\o\ \exists a\in\mathcal U_n\ \forall x,y\in a\ [x\le y\sledi \pi_{n,m}(x)\le \pi_{n,m}(y)]$.
\end{enumerate}
Then for every $e\in \P$ there is a sequence of maps in $\o^{\o}$, $\seq{\pi_n:n<\o}$, satisfying:
\begin{enumerate}[resume=elemma]
\item\label{i104} $\forall n<\o\ [\pi''_n \sset(e)\in \mathcal U_n]$;
\item\label{i105} $\forall m\le n<\o\ \forall^{\infty} k\in \sset(e)\ [\pi_m(k)=\pi_{n,m}(\pi_n(k))]$;
\item\label{i106} for every $n<\o$ there are $\psi_n\in \o^{\o}$ and $b_n\in \P$ such that $e\le b_n$ and that $\seq{\pi_n,\psi_n,b_n}$ is a normal triple.
\end{enumerate}
\end{lemma}

\begin{proof}
Define $E_k=\o$, for every $k<\o$. Let $M$ be a countable elementary submodel of $H_{(2^{\c})^+}$ containing $\bar\pi=\seq{\pi_{n,m}:m\le n<\o}$, $\seq{\cu_n:n<\o}$. For $m<\o$, let $D_m\in \cu_m$ be such that $D_m\subset^* A$ for every $A\in \cu_m\cap M$. Now Lemma \ref{t7} applies to $M$, function $f=\id$, sequences $\bar \pi$, $\bar D=\seq{D_m:m<\o}$, $\seq{\cu_n:n<\o}$ and $\seq{E_n:n<\o}$. Let sequences $\seq{F_n:n<\o}$, $\seq{C_n:n<\o}$, $\seq{g'(n):n<\o}$ and $\bar K=\seq{K_{m,n}:m\le n<\o}$ be as in Lemma \ref{t7}. Denote $\rho=\seq{\bar D,\bar K,\bar \pi}$, $R_n$, $z^n_j$, $x^n_j$ and $m(n,j)$ ($j<R_n$) as in the conclusion of Lemma \ref{t7}. 

For each $n < \omega$, define ${I}_{n} = \{{L}^{\rho}_{n} + j: j < {R}_{n}\}$.
Recall from the proof of Lemma \ref{t30} that $\langle {I}_{n}: n \in \omega \rangle$ is an interval partition of $\omega$.
Fix $m < \omega$.
For $n < m$ and $j < {R}_{n}$, define ${\psi}_{m}({L}^{\rho}_{n} + j) = 0$, while for $m \leq n$ and $j < {R}_{n}$, define ${\psi}_{m}({L}^{\rho}_{n} + j ) = {\pi}_{n, m}({z}^{n}_{j})$.
Thus ${\psi}_{m} \in {\omega}^{\omega}$ and we claim that it is increasing.
It suffices to consider the following two cases.
Case $1$ is when $n < \omega$, $j \leq j' < {R}_{n}$ and we wish to compare ${\psi}_{m}({L}^{\rho}_{n} + j)$ and ${\psi}_{m}({L}^{\rho}_{n} + j')$.
If $n < m$, then both these values are $0$.
If $m \leq n$, then ${\psi}_{m}({L}^{\rho}_{n} + j) = {\pi}_{n, m}({z}^{n}_{j}) \leq {\pi}_{n, m}({z}^{n}_{j'}) = {\psi}_{m}({L}^{\rho}_{n} + j')$ because ${z}^{n}_{j} \leq {z}^{n}_{j'}$ and because ${z}^{n}_{j}, {z}^{n}_{j'} \in {C}_{n}$.
Now we come to case $2$, which is when we wish to compare ${\psi}_{m}({L}^{\rho}_{n} + j)$ and ${\psi}_{m}({L}^{\rho}_{n + 1} + j')$, for some $n < \omega$, $j < {R}_{n}$, and $j' < {R}_{n + 1}$.
First, if $n < m$, then ${\psi}_{m}({L}^{\rho}_{n} + j) = 0 \leq {\psi}_{m}({L}^{\rho}_{n + 1} + j')$.
So assume that $m \leq n$.
Since ${z}^{n}_{j} \in {F}_{n}$, there exists $z \in {C}_{n + 1}$ with ${\pi}_{n + 1, n}(z) = {z}^{n}_{j}$.
By a property of ${C}_{n + 1}$ from Lemma \ref{t7}, ${\pi}_{n + 1, m(n, j)}(z) = {\pi}_{n, m(n, j)}({\pi}_{n + 1, n}(z)) = {\pi}_{n, m(n, j)}({z}^{n}_{j}) = {x}^{n}_{j} \leq {D}_{m(n, j)}\left( {K}_{m(n, j), n + 1} - 1 \right)$.
It follows from (\ref{i54}) of Lemma \ref{t7} applied to $n + 1$ that $z < {F}_{n + 1}(g'(n + 1)) \leq {z}^{n + 1}_{j'}$.
Since ${z}^{n + 1}_{j'} \in {C}_{n + 1}$, ${\pi}_{n, m}({z}^{n}_{j}) = {\pi}_{n, m}({\pi}_{n + 1, n}(z)) = {\pi}_{n + 1, m}(z) \leq {\pi}_{n + 1, m}({z}^{n + 1}_{j'})$.
So ${\psi}_{m}({L}^{\rho}_{n} + j) = {\pi}_{n, m}({z}^{n}_{j}) \leq {\pi}_{n + 1, m}({z}^{n + 1}_{j'}) = {\psi}_{m}({L}^{\rho}_{n + 1} + j')$.
Thus we have proved that ${\psi}_{m}$ is increasing.

Now for each $m < \omega$, define ${\pi}_{m} \in {\omega}^{\omega}$ as follows.
Let $k \in \omega$.
If $k \notin \sset(e)$, then set ${\pi}_{m}(k) = 0$; else let $l \in \omega$ be unique such that $k \in e(l)$, and set ${\pi}_{m}(k) = {\psi}_{m}(l)$.
We check that (\ref{i104})--(\ref{i106}) are satisfied.
We begin with (\ref{i105}).
Fix $m \leq l < \omega$.
Consider any $k \in \sset(e)\lp {L}^{\rho}_{l} \rp$.
Then $k \in e({L}^{\rho}_{n} + j)$, for some $l \leq n < \omega$ and $j < {R}_{n}$.
So ${\pi}_{m}(k) = {\psi}_{m}({L}^{\rho}_{n} + j) = {\pi}_{n, m}({z}^{n}_{j})$ and ${\pi}_{l}(k) = {\psi}_{l}({L}^{\rho}_{n} + j) = {\pi}_{n, l}({z}^{n}_{j})$.
Since ${z}^{n}_{j} \in {C}_{n}$, ${\pi}_{n, m}({z}^{n}_{j}) = {\pi}_{l, m}({\pi}_{n, l}({z}^{n}_{j}))$.
Therefore, ${\pi}_{m}(k) = {\pi}_{n, m}({z}^{n}_{j}) = {\pi}_{l, m}({\pi}_{n, l}({z}^{n}_{j})) = {\pi}_{l, m}({\pi}_{l}(k))$, as needed for (\ref{i105}).

Next we prove (\ref{i104}). Fix $m<\o$. We will show $D_m\subset^* \pi_m''\sset(e)$. As the sequence $\seq{K_{m,n}:m\le n<\o}$ is strictly increasing with $n$, it suffices to show that for each $n\ge m$, $D_m[K_{m,n},K_{m,n+1})\subset \pi_m''\sset(e)$. Let $n\!\ge\! m$ and $u\!\in\! D_m[K_{m,n},K_{m,n+1})$ be given. Apply the same argument as in the proof of (\ref{i93}) of Claim \ref{t92} to find $m'$ and $u'$ so that $m\le m'\le n$, $u'\in D_{m'}[K_{m',n},K_{m',n+1})$, $\pi_{m',m}(u')=u$, and $u'\in H^{\rho}_{m',n}\subset \Delta^{\rho}_n$. Let $j<R_n$ be such that $x^n_j=u'$. Note that $m(n,j)=m'$. Also $z^n_j\in F_n\subset C_n$. So by a property of $C_n$ from Lemma \ref{t7}, $\pi_{n,m}(z^n_j)=\pi_{m',m}(\pi_{n,m'}(z^n_j))=\pi_{m',m}(u')=u$. Now if $k\in e(L^{\rho}_n+j)$, then since $m\le n$, by definition, $\pi_m(k)=\pi_{n,m}(z^n_j)=u$. Thus $u\in \pi_m''\sset(e)$. This proves $D_m\subset^*\pi_m''\sset(e)$, which proves (\ref{i104}) because $D_m\in \cu_m$.

We still have to prove (\ref{i106}).
Fix $m < \omega$.
We have already defined ${\psi}_{m}$ and proved that it is increasing.
Let ${b}_{m} = e$.
By definition of ${\pi}_{m}$, ${\pi}_{m}''{b}_{m}(l) = \{{\psi}_{m}(l)\}$, for each $l < \omega$, and ${\pi}_{m}(k) = 0$, for all $k \in \omega \setminus \sset({b}_{m})$.
Also $\ran({\psi}_{m})$ is infinite because ${\pi}_{m}''\sset(e) \in {\cu}_{m}$ and ${\pi}_{m}''\sset(e) \subset \ran({\psi}_{m})$.
Therefore $\langle {\pi}_{m}, {\psi}_{m}, {b}_{m} \rangle$ is a normal triple and $e \leq {b}_{m}$, as needed.
\end{proof}
In the context of $\MA$, the statement of Lemma \ref{t101} will be modified as follows.
The sequence $\langle {\cu}_{n}: n < \omega \rangle$ will be replaced with the sequence $\langle {\cu}_{\alpha}: \alpha < \lambda \rangle$, where $\lambda$ is a cardinal $ < \c$.
Moreover each ${\cu}_{\alpha}$ will be assumed to be a rapid ${P}_{\c}$-point.
And, of course, there will be a map ${\pi}_{\beta, \alpha}$, for each $\alpha \leq \beta < \lambda$.
The sequence $\langle {\pi}_{n}: n < \omega \rangle$ in the conclusion of Lemma \ref{t101} will be replaced by the sequence $\langle {\pi}_{\alpha}: \alpha < \lambda \rangle$.
This version can be proved under $\MA$ by taking a suitably generic filter over a poset consisting of finite approximations to the sequence $\langle {\pi}_{\alpha}: \alpha < \lambda \rangle$ together with some finite side conditions.
Its exact definition can be gotten by looking at the proofs of Lemmas \ref{t7} and \ref{t101}.

The next lemma will also be used in the proof that the poset ${\Q}^{\delta}$ is countably closed.
It is like a simple special case of Lemma \ref{t30} in spirit, 
but does not directly follow from the statement of Lemma \ref{t30}.
\begin{lemma}\label{t36}
Let $\cu$ be a rapid P-point, $\pi$ a mapping in $\o^{\o}$ and $\seq{d_m:m<\o}$ a decreasing sequence of conditions in $\P$ such that $\pi''\sset(d_n)\in\cu$ for every $n<\o$. Suppose that there are $b\in \P$ and $\psi\in\o^{\o}$ so that $\seq{\pi,\psi,b}$ is a normal triple and $d_0\le b$. Then there is $d\in \P$ such that $\pi''\sset(d)\in \cu$ and $d\le d_n$ for every $n<\o$.
\end{lemma}

\begin{proof}
First we define sequence of numbers $n_k$ ($k<\o$) as follows: $n_0$ is minimal such that $d_0\le_{n_0} b$, while $n_{k+1}=\max\set{l,n_k}$ for $l$ minimal such that $d_{k+1}\le_l d_k$. By Remark \ref{r5}(\ref{i204}) we have $d_{k}\le_{n_k} b$, $d_k\le_{n_k}d_l$ for $l\le k$ and consequently $\sset(d_{k+1})\lp n_{k+1}\rp\subset \sset(d_k)\lp n_k\rp$ for $k<\o$. Let $C_k=\pi''\sset(d_k)\lp n_k\rp$ for $k<\o$ and notice that $C_{k+1}\subset C_k$ and $C_k\in \cu$ for $k<\o$. So since $\cu$ is a rapid ultrafilter, by Lemma \ref{t33}, for every $k<\o$ there is $D_k\in\cu$ such that for every $n<\o$ there is $m\ge 2(n+1)$ such that $D_k(n)=C_k(m)$. Because $\cu$ is a P-point there is $D\in \cu$ such that $D\subset^* D_k$ for every $k<\o$ and $D\subset D_0$. For every $k,l<\o$ define set $F^k_l=\set{m<\o:\pi''d_k(m)=\set{D(l)}}$. By Lemma \ref{t14}, if $D(l)\in C_k$, then $\dc^k_l=\max(F^k_l)$ and $\bc^k_l=\min(F^k_l\setminus n_k)$ are well defined. For a fixed $k$ and $l_1\!<\!l_2$ such that $D(l_1),D(l_2)\in C_k$, again by Lemma \ref{t14}, we have $\dc^k_{l_1}<\bc^k_{l_2}\le\dc^k_{l_2}$. Also, if $l_1\!<\!l_2\!<\!\o$, $k_1\!<\!k_2\!<\!\o$, $D(l_1)\!\in\! D_{k_1}$ and $D(l_2)\!\in\! D_{k_2}$, then it is easy to see that $\max(d_{k_1}(\dc^{k_1}_{l_1}))<\min(d_{k_2}(\dc^{k_2}_{l_2}))$.
%%To see that the last observation is true, first note that $D(l_1),D(l_2)\in C_{k_1}$, so $\dc^{k_1}_{l_1}<\bc^{k_1}_{l_2}\le \dc^{k_1}_{l_2}$. Next, since $\dc^{k_2}_{l_2}\ge n_{k_2}$ and $d_{k_2}\le_{n_{k_2}}d_{k_1}$ there is $\dc\ge \dc^{k_2}_{l_2}$ such that $d(k_2)(\dc^{k_2}_{l_2})\subset d_{k_1}(\dc)$. Also $\pi''d_{k_1}(\dc)=\set{D(l_2)}$ because $d_{k_1}\le_{n_{k_1}} b$ and $\dc\ge \dc^{k_2}_{l_2}\ge n_{k_2}\ge n_{k_1}$. Therefore $\dc\in F^{k_1}_{l_2}\setminus n_{k_1}$, so $\dc\ge \bc^{k_1}_{l_2}>\dc^{k_1}_{l_1}$.
Now, by induction on $k$, we construct numbers $g(k)$ and sets $d(m)$ for $m<g(k)$ so that for $k<\o$:
\begin{enumerate}
\item\label{i602} $\dc^k_{g(k)}\ge g(k)$;
\item\label{i606} $D[g(k)]\subset D_{k}$;
\item\label{i603} if $k>0$ then $\forall l\in [g(k-1),g(k))\ [d(l)\in [d_{k-1}(\dc^{k-1}_l)]^{l+1}]$.
\item\label{i604} if $k\!>\!0$ then $g(k)\!>\!g(k-1)$ and $\forall l\!<\!g(k)-1\ [\max(d(l))\!<\!\min(d(l+1))]$.
\end{enumerate}
Let $g(0)=0$ and note that (\ref{i602}-\ref{i604}) are satisfied. So fix $k\in\o$ and assume that for every $m\le k$ numbers $g(m)$ are defined, and that for every $l<g(k)$ sets $d(l)$ are defined. Let $X_k$ be the minimal number such that $X_k>g(k)$ and $D[X_k]\subset D_{k+1}$ and define $g(k+1)=2X_k$. First note that since $X_k>g(k)$ we have that $g(k+1)=2X_k>g(k)$. Since by inductive hypothesis $D[g(k)]\subset D_{k}$ and $\dc^k_{g(k)}\ge g(k)$, Lemma \ref{t14} implies that for $g(k)\le l<l'<g(k+1)$ we have $\dc^k_l<\dc^k_{l'}$. So we can pick $d(l)\in [d_{k}(\dc^k_l)]^{l+1}$ for $g(k)\le l<g(k+1)$. Now we prove that (\ref{i602}-\ref{i604}) hold. To prove (\ref{i602}) note that by the choice of $X_k$ we know that $D[X_k,g(k+1))\subset D_{k+1}$. So $D(g(k+1))=D_{k+1}(m)$ for some $m\ge g(k+1)-X_k=X_k$. This implies that $D(g(k+1))=C_{k+1}(m')$ for some $m'\ge 2m\ge 2X_k\ge g(k+1)$. Now by Lemma \ref{t14} applied to $C_{k+1}$, $d_{k+1}$ and $\seq{\pi,\psi,b}$ we have $\dc^{k+1}_{g(k+1)}\ge g(k+1)$. Condition (\ref{i606}) follows from the fact that $g(k+1)=2X_k\ge X_k$ and $D[X_k]\subset D_{k+1}$. Condition (\ref{i603}) holds by construction. To see that (\ref{i604}) is true we distinguish three cases: either $l>g(k)-1$ or $l<g(k)-1$ or $l=g(k)-1$. If $l<g(k)-1$ then it follows from the inductive hypothesis and the fact that $g(0)=0$. If $l=g(k)-1$ then because $g(k-1)<g(k)$ we have $g(k)-1\ge g(k-1)$ so $k>0$. By (\ref{i606}) applied to $k-1$ and $k$ we know $D(g(k)-1)\in D_{k-1}$ and $D(g(k))\in D_k$ so the statement follows from the observation in the first paragraph that $\max(d_{k-1}(\dc^{k-1}_{g(k)-1}))<\min(d_k(\dc^k_{g(k)}))$. If $l>g(k)-1$ then it follows from the facts that $\dc^k_l<\dc^k_{l+1}$.
\end{proof}

\section{Adding an ultrafilter on top} \label{sec:top}
In this section, for a given $\d<\o_2$, we introduce the poset for adding a rapid P-point $\cu_{\d}$ together with a sequence of maps $\langle {\pi}_{\delta, \alpha}: \alpha \leq \delta \rangle$ on top of an already constructed $\d$-generic sequence of P-points $\seq{\cu_{\a}:\a<\d}$ and Rudin-Keisler maps $\seq{\pi_{\b,\a}:\a\le \b<\d}$.
So fix a $\d<\o_2$ and a $\d$-generic sequence $S=\seq{\seq{c^{\a}_i:\a<\d \wedge i<\d},\seq{\pi_{\b,\a}:\a\le \b<\d}}$ for the rest of this section.

We briefly explain the idea behind the definition of ${\Q}^{\delta}$ given below.
We would like a generic filter for ${\Q}^{\delta}$ to produce two sequences $\bar{C} = \langle {c}^{\delta}_{i}: i < \c \rangle$ and $\bar{\pi} = \langle {\pi}_{\delta, \alpha}: \alpha \leq \delta \rangle$ which, when added to $S$, will result in a $\delta + 1$-generic sequence.
Conditions in ${\Q}^{\delta}$ are essentially countable approximations to such objects.
The first coordinate of the condition $q$ will be an element of $\bar{C}$, and the fourth coordinate fixes $\bar{\pi}$ on a countable subset of $\delta$.
Clauses (\ref{i16}), (\ref{i17}), and (\ref{i15}) below say that the maps that have already been determined by $q$ work in accordance with clauses (\ref{i5}), (\ref{i0}), and (\ref{i6}) of Definition \ref{d22}.
Clause \ref{i13} below says that ${X}_{q}$, which is the countable set on which $\bar{\pi}$ has been fixed, always has a maximal element unless ${X}_{q}$ is cofinal in $\delta$.
This assumption will simplify some arguments.
\begin{definition}\label{d14}
Let $\Q^{\d}$ be the set of all $q=\seq{c_q,\g_q,X_q,\seq{\pi_{q,\a}:\a\in X_q}}$ such that:
\begin{enumerate}
\item\label{i11} $c_q\in\P$;
\item\label{i12} $\g_q\le \d$;
\item\label{i13} $X_q\in[\d]^{\le\o}$ is such that $\g_q=\sup(X_q)$ and $\g_q\in X_q$ iff $\g_q<\d$;
\item\label{i14} $\pi_{q,\a}$ ($\a\in X_q$) are mappings in $\o^{\o}$ such that:
\begin{enumerate}
\item\label{i16} $\pi''_{q,\a}\sset(c_q)\in \mathcal U_{\a}$;
\item\label{i17} $\forall \a,\b\in X_q\ \left[\a\le \b\sledi \forall^{\infty}k\in \sset(c_q)\ [\pi_{q,\a}(k)=\pi_{\b,\a}(\pi_{q,\b}(k))]\right];$
\item\label{i15} there is $\psi_{q,\a}\in\o^{\o}$ and $b_{q,\a}\ge c_q$ such that $\seq{\pi_{q,\a},\psi_{q,\a},b_{q,\a}}$ is a normal triple;
\end{enumerate}
\end{enumerate}
Let the ordering on $\Q^{\d}$ be given by: $q_1\le q_0$ if and only if
\[\TS
c_{q_1}\le c_{q_0}\ \mbox{and}\ X_{q_1}\supset X_{q_0}\ \mbox{and for every}\ \a\in X_{q_0},\ \pi_{q_1,\a}=\pi_{q_0,\a}.
\]
\end{definition}

In the situation where $\CH$ is replaced by $\MA$, ${\Q}^{\delta}$ would consist of approximations of size $< \c$ instead of countable ones.
Thus ${X}_{q}$ would be a set of size less than $\c$.
\begin{remark}\label{r2}
It is easy to check that $\seq{\Q^{\d},\le}$ defined in this way is a partial order. Note also that $\Q^{\d}\neq 0$. Namely, if $\d=0$, then we can take $q=\seq{c,0,0,0}$ for any $c\in\P$. If $\d\neq 0$, then let $q=\seq{c_q,\g_q,X_q,\seq{\pi_{q,\a}:\a\in X_q}}$ be such that: $c_q$ is arbitrary in $\P$; $\g_q=0$; $X_{q}=\set{0}$; $\pi_{q,0}\in\o^{\o}$ is given by: for $k\in \sset(c_q)$ let $\pi_{q,0}(k)=n$ for $k\in c_q(n)$, while $\pi_{q,0}(k)=0$ otherwise. First note that conditions (\ref{i11}-\ref{i13}) of Definition \ref{d14} are satisfied. Because $\pi_{q,0}''\sset(c_q)=\o$ we know that (\ref{i16}) holds. It is also easy to see that $\seq{\pi_{q,0},\id,c_q}$ is a normal triple by definition of $\pi_{q,0}$ so condition (\ref{i15}) is true. To see that condition (\ref{i17}) is also true note that $\pi_{0,0}=\id$ by Definition \ref{d22}(\ref{i3}). So $q\in \Q^{\d}$.
\end{remark}

\begin{remark}\label{r101}
Let $q=\seq{c_q,\g_q,X_q,\seq{\pi_{q,\a}:\a\in X_q}}\in \Q^{\d}$. Let $c_{q'}\in \P$ be such that $c_{q'}\le c_q$. Then $q'=\seq{c_{q'},\g_q,X_q,\seq{\pi_{q,\a}:\a\in X_q}}$ satisfies conditions (\ref{i11}), (\ref{i12}), (\ref{i13}), (\ref{i17}) and (\ref{i15}) of Definition \ref{d14}. Moreover, if $q'$ also satisfies Definition \ref{d14}(\ref{i16}), then $q'\in \Q^{\d}$ and $q'\le q$.
\end{remark}
Instead of forcing with the poset ${\Q}^{\delta}$, we would like to build a sufficiently generic filter over it in the ground model itself.
${\Q}^{\delta}$ needs to be countably closed for this to be feasible.
We prove this fact next.
The next lemma is the crux of the whole construction.
We briefly sketch the idea of its proof.
So suppose that $\langle {q}_{n}: n 
\in \omega \rangle$ is a decreasing sequence of conditions in ${\Q}^{\delta}$.
We want to find a lower bound.
There are four natural cases to consider.
We start with the simpler ones.
The most trivial case is when $\delta = 0$.
Then we just have a decreasing sequence in $\P$ and bounding them is easy.
Next, it could be the case that for all $n \in \omega$, ${\gamma}_{{q}_{n}} = \gamma$, for some fixed $\gamma < \delta$.
Then we essentially have a fixed ultrafilter ${\cu}_{\gamma}$, a descending sequence in $\P$, and a fixed map taking each element of this sequence into ${\cu}_{\gamma}$.
We wish to find a bound for this sequence in $\P$ whose image is still in ${\cu}_{\gamma}$.
Lemma \ref{t36} is set up precisely to handle this situation, so we apply it.
The third case is when the ${\gamma}_{{q}_{n}}$ form an increasing sequence converging to $\delta$.
Then we have a decreasing sequence in $\P$, some countable cofinal $Y \subset \delta$, and a sequence of maps taking members of the decreasing sequence in $\P$ to various ultrafilters indexed by $Y$.
We would like to find a lower bound for this decreasing sequence in $\P$ whose images under each of the given maps are in the corresponding ultrafilters.
This is almost like the situation in Lemma \ref{t30}, expect that $e$ and its associated maps are missing.
So we first apply Lemma \ref{t101} to find these things, and then apply Lemma \ref{t30} to them.
The final and trickiest case is when the ${\gamma}_{{q}_{n}}$ form an increasing sequence converging to some $\mu < \delta$.
Then the ultrafilter ${\cu}_{\mu}$ must have been constructed to anticipate this situation.
This is where clause (\ref{i8}) of Definition \ref{d22} enters.
We have a decreasing sequence in $\P$, a countable cofinal $Y \subset \mu$, and a sequence of maps as before.
We would like to find a lower bound for this decreasing sequence in $\P$ as well as a new map associated with ${\cu}_{\mu}$ in such a way that the images of this lower bound under all of the maps, both old and new, are in the corresponding ultrafilters.
Clause (\ref{i8}) of Definition \ref{d22} says precisely that this is possible.  
\begin{lemma}\label{t91}
For any decreasing sequence of conditions $\seq{q_n:n\!<\!\o}$ in $\Q^{\d}$ there is $q\!\in\! \Q^{\d}$ so that $\forall n\!<\!\o [q\!\le\! q_n]$. Moreover, if $\forall n<\o [X_{q_{n+1}}\!=\!X_{q_n}]$, then $X_{q}\!=\!X_{q_0}$.
\end{lemma}

\begin{proof}
Assume that we are given a decreasing sequence of conditions $\seq{q_n:n<\o}$ in $\Q^{\d}$, i.e. $q_{n+1}\le q_n$ for $n<\o$. Define $Y=\bigcup_{n<\o}X_{q_n}$ and $\g=\sup(Y)$. Note that $Y\in [\d]^{\le\o}$. Also, if $\forall n<\o\ [X_{q_{n+1}}=X_{q_n}]$, then $Y=X_{q_0}$. So the moreover part of the lemma holds as long as we find $q$ such that $X_q = Y$. We will consider two cases: either $\g\in Y$ or $\g\notin Y$.

\vskip2mm\noindent
Case I: $\g\in Y$. Then there is $n_0<\o$ such that $\g\in X_{q_{n_0}}$. So $\g=\g_{q_{n_0}}$ and note that $\g<\d$ because $X_{q_{n_0}}\subset \d$. Notice that $\g_{q_{n+1}}\ge \g_{q_n}$ for every $n<\o$, so $\g_{q_n}=\g_{q_{n_0}}$ for every $n\ge n_0$. Also, by Definition \ref{d14} we know that $\g\in X_{q_n}$ for $n\ge n_0$. We apply Lemma \ref{t36} in such a way that: $d_n$ in Lemma \ref{t36} is $c_{q_{n+n_0}}$ ($n<\o$); $\cu$ is $\cu_{\g}$; $\pi$ is $\pi_{q_{n_0},\g}$; $\psi$ is $\psi_{q_{n_0},\g}$ and $b$ is $b_{q_{n_0},\g}$.
It is easy to see that the hypotheses of Lemma \ref{t36} are satisfied. So there is $d\in\P$ such that $\pi''\sset(d)\in\cu_{\g}$ and $d\le d_n$ for every $n<\o$. Now we will prove that the condition $q=\seq{d,\g,Y,\seq{\pi_{q,\a}:\a\in Y}}$ is as required, where $\pi_{q,\a}$ is $\pi_{q_{n+n_0},\a}$ for any $n<\o$ such that $\a\in X_{q_{n+n_0}}$. To show that $q\in \Q^{\d}$ note that conditions (\ref{i11}-\ref{i13}) of Definition \ref{d14} are clearly satisfied. To prove Definition \ref{d14}(\ref{i17}), fix $\a,\b\in Y$ such that $\a\le \b$. There is $n<\o$ such that $\a,\b\in X_{q_{n_0+n}}$. Since $\sset(d)\subset^*\sset(c_{q_{n_0+n}})$ and $\pi_{q,\a}=\pi_{q_{n_0+n},\a}$ and $\pi_{q,\b}=\pi_{q_{n_0+n},\b}$, by Definition \ref{d14}(\ref{i17}) for $q_{n_0+n}$ we have $\forall^{\infty}k\in \sset(d)\ [\pi_{q,\a}(k)=\pi_{\b,\a}(\pi_{q,\b}(k))]$ as required. To see that Definition \ref{d14}(\ref{i16}) is true take arbitrary $\b\in Y$. First notice that $\pi_{q_{n_0},\g}''\sset(d)\in \cu_{\g}$. So (\ref{i16}) is true in case $\b=\g$. If $\b<\g$ consider the set $Z=\pi_{\g,\b}''(\pi_{q,\g}''\sset(d))$. It belongs to $\cu_{\b}$ by already proved (\ref{i16}) for $\g$ and Definition \ref{d22}(\ref{i5}). However, by already proved (\ref{i17}) we have $Z\subset^*\pi_{q,\b}''\sset(d)$ which implies that $\pi_{q,\b}''\sset(d)\in \cu_{\b}$. We still have to prove (\ref{i15}). Take arbitrary $\a\in Y$ and let $n<\o$ be such that $\a\in X_{q_{n_0+n}}$. We know that $\seq{\pi_{q_{n_0+n},\a},\psi_{q_{n_0+n},\a},b_{q_{n_0+n},\a}}$ is a normal triple, that $d\le c_{q_{n_0+n}}\le b_{q_{n_0+n},\a}$ and that $\pi_{q,\a}=\pi_{q_{n_0+n},\a}$. So $d\le b_{q_{n_0+n},\a}$ and $\seq{\pi_{q,\a},\psi_{q_{n_0+n},\a},b_{q_{n_0+n},\a}}$ is as required.

\vskip2mm\noindent
Case II: $\g\notin Y$. Therefore $Y\subset \g$. In this case either $\g=0$ or $\g$ is a limit ordinal such that $\cf(\g)=\o$. So there are three subcases: either $\g=0$ or $\g<\d$ and $\cf(\g)=\o$ or $\g=\d$ and $\cf(\g)=\o$.

\vskip1mm\noindent
Subcase IIa: $\g=0$. Since $Y\subset \g$ we have $Y=0$, so $X_{q_0}=0$ and $\g_{q_0}=0$ and $\g_{q_0}\notin X_{q_0}$. So $\d=\g_{q_0}=0$. In this case all the conditions $q_n$ ($n<\o$) are of the form $q_n=\seq{c_{q_n},0,0,0}$. So it is enough to construct condition $c_q\le c_{q_n}$ ($n<\o$) because in that case $q=\seq{c_q,0,0,0}$ will satisfy $q\le q_n$ for every $n<\o$, and also the moreover part of the lemma. For $n<\o$ let $k_n$ be such that $c_{q_{n+1}}\le_{k_n} c_{q_n}$. Define $m_0=0$ and $m_{n+1}=\max\set{k_n,\max(c_{q_n}({m}_{n}))+2}$ for $n<\o$. Let $c_q(n)=c_{q_n}(m_n)$ for $n<\o$. It is obvious that $c_q\in \P$ and $c_q\le c_{q_n}$ for every $n<\o$.

\vskip1mm\noindent
Subcase IIb: $\cf(\g)=\o$ and $\g<\d$. We apply Definition \ref{d22}(\ref{i8}) as follows: $\mu$ is $\g$, $X$ is $Y$ and $d_n$ is $c_{q_n}$ ($n<\o$).
For $\a\in Y$ let $n<\o$ be minimal such that $\a\in X_{q_n}$. Then we consider $\pi_{\a}$ to be $\pi_{q_n,\a}$, $\psi_{\a}$ to be $\psi_{q_n,\a}$ and $b_{\a}$ to be $b_{q_n,\a}$ - note that if $m\le n$ then $\sset(c_{q_n})\subset^*\sset(c_{q_m})$, so $\pi_{q_n,\a}''\sset(c_{q_m})\in \cu_{\a}$, while if $m>n$, then $\pi_{q_n,\a}=\pi_{q_m,\a}$ and Definition \ref{d14}(\ref{i16}) implies $\pi_{q_n,\a}''\sset(c_{q_m})\in \cu_{\a}$, and so Definition \ref{d22}(\ref{i21}) holds; Definition \ref{d22}(\ref{i23}) is true because $\seq{\pi_{\a},\psi_{\a},b_{\a}}$ is a normal triple and $d_n=c_{q_n}\le b_{q_n,\a}=b_{\a}$; to show that Definition \ref{d22}(\ref{i22}) is satisfied, pick $\a,\b\in Y$ such that $\a\le \b$, let $n<\o$ be minimal such that $\a\in X_{q_n}$, let $m<\o$ minimal such that $\b\in X_{q_m}$ and assume $n\le m$ (case $m\le n$ is symmetric). Then $\a,\b\in X_{q_m}$ so according to Definition \ref{d14}(\ref{i17}) for $q_m$ we have $\forall^{\infty}k\in \sset(d_m)\ [\pi_{{q}_{m},\a}(k)=\pi_{\b,\a}(\pi_{{q}_{m},\b}(k))]$.

Hypothesis of Definition \ref{d22}(\ref{i8}) is satisfied as explained above. So there are $i^*<\c$, $d^*\in \P$ and $\pi,\psi\in\o^{\o}$ which satisfy the conclusion of Definition \ref{d22}(\ref{i8}). Now define condition $q=\seq{d^*,\g,Y\cup\set{\g},\seq{\pi_{q,\a}:\a\in Y\cup\set{\g}}}$, where for $\a\in Y$, $\pi_{q,\a}$ is $\pi_{q_n,\a}$ for the minimal $n<\o$ such that $\a\in X_{q_n}$, while $\pi_{q,\g}$ is $\pi$. When we prove $q\in \Q^{\d}$ it will follow easily that $q\le q_n$ for $n<\o$. So we check conditions (\ref{i11}-\ref{i14}) of Definition \ref{d14}. The only non-trivial condition is (\ref{i14}). First we show (\ref{i17}). Take any $\a,\b\in Y$ such that $\a\le \b$. There are two cases, either $\b=\g$ or $\b\neq \g$. If $\b=\g$, then by Definition \ref{d22}(\ref{i32}) we have $\forall^{\infty}\in \sset(d^*)\ [\pi_{q,\a}(k)=\pi_{\g,\a}(\pi_{q,\g}(k))]$ as required. If $\b<\g$, then pick $n<\o$ such that $\a,\b\in X_{q_n}$. Then since $\sset(d^*)\subset\sset(c_{q_n})$, by Definition \ref{d14}(\ref{i17}) applied to $q_n$ we have $\forall^{\infty} k \in \sset(d^*)\ [\pi_{q,\a}(k)=\pi_{\b,\a}(\pi_{q,\b}(k))]$ as required. Next we prove (\ref{i16}). Let $\a\in Y$. If $\a=\g$, then by Definition \ref{d22}(\ref{i31}) we have $\pi_{q,\g}''\sset(d^*)=\sset(c^{\g}_{i^*}) \in \cu_{\g}$. If $\a<\g$, then by already proved (\ref{i17}) we have $\pi_{\g,\a}''\br{\pi_{q,\g}''\sset(d^*)}\subset^*\pi_{q,\a}''\sset(d^*)$. This together with Definition \ref{d22}(\ref{i5}) gives $\pi_{q,\a}''\sset(d^*)\in \cu_{\a}$ as required. We still have to prove (\ref{i15}). Take arbitrary $\a\in Y$. If $\a=\g$, then $\seq{\pi_{q,\g},\psi,d^*}$ is itself a normal triple. If $\a<\g$ let $n<\o$ be minimal such that $\a\in X_{q_n}$. Then $\seq{\pi_{q,\a},\psi_{q_n,\a},b_{q_n,\a}}$ is a normal triple and $d^*\le c_{q_n}\le b_{q_n,\a}$ as required.

The situation from the moreover part of the lemma does not occur in this subcase.
To see this, suppose otherwise.
Then $Y=X_{q_0}$ and $\g=\g_{q_0}$. 
Since $\g<\d$, by Definition \ref{d14}(\ref{i13}) $\gamma = \g_{q_0}\in X_{q_0}\subset Y$, a contradiction to Case II.

\vskip1mm\noindent
Subcase IIc: $\cf(\g)=\o$ and $\g=\d$. Choose $\seq{\g_n:n<\o}$ such that $\sup\set{\g_n:n<\o}=\d$, and $\g_n<\g_{n+1}$ and $\g_n\in Y$, for every $n<\o$. Now we apply Lemma \ref{t101} as follows: $\cu_n$ is $\cu_{\g_n}$ ($n<\o$) - note that the $\cu_{\g_n}$'s are distinct rapid P-points; $\pi_{m,n}$ is $\pi_{\g_m,\g_n}$ ($n\le m<\o$) - note that by Definition \ref{d22}(\ref{i3}) conditions (\ref{i101}-\ref{i103}) of Lemma \ref{t101} are satisfied.

As we have explained above, hypothesis of Lemma \ref{t101} is satisfied, so there are $e\in\P$ and maps $\pi_{\d,\g_n}$ ($n<\o$) such that
\begin{enumerate}[resume]
\item\label{i4001} $\forall n<\o\ [\pi''_{\d,\g_n} \sset(e)\in \mathcal U_{\g_n}]$;
\item\label{i4002} $\forall n\le m<\o\ \forall^{\infty} k\in \sset(e)\ [(\pi_{\d,\g_n}(k)=\pi_{\g_m,\g_n}(\pi_{\d,\g_m}(k))]$;
\item\label{i4003} for every $n<\o$ there are $\psi_{\d,\g_n}$ and $b_{\d,\g_n}$ such that $\seq{\pi_{\d,\g_n},\psi_{\d,\g_n},b_{\d,\g_n}}$ is a normal triple and $e\le b_{\d,\g_n}$.
\end{enumerate}

We will apply Lemma \ref{t30} as follows: $d_n$ is $c_{q_n}$ for $n<\o$, $e$ is $e$, $\d$ is $\d$ and $f=\id$ - note that $\cf(\d)=\o$; $X$ is $\set{\g_n:n<\o}$ - note that $\d=\sup(X)$ as required in Lemma \ref{t30}; the $\pi_{\d,\a}$ are $\pi_{\d,\a}$, for $\a\in X$ - note that Lemma \ref{t30}(\ref{i72}) is true by (\ref{i4001}-\ref{i4003}); for $n<\o$, $\pi_{\g_n}$ is $\pi_{q_m,\g_n}$, $b_{\g_n}$ is $b_{q_m,\g_n}$, $\psi_{\g_n}$ is $\psi_{q_m,\g_n}$ for the minimal $m < \omega$ such that ${\gamma}_{n} \in {X}_{{q}_{m}}$. We have to show that Lemma \ref{t30}(\ref{i77}-\ref{i79}) are satisfied. First we prove (\ref{i77}). Fix $n<\o$ and let $m < \omega$ be minimal such that $\g_n\in X_{q_m}$. We will show that $\forall j<\o [\pi_{\g_n}''\sset(c_{q_j})=\pi_{q_m,\g_n}''\sset(c_{q_j})\in \cu_{\g_n}]$. There are two cases: either $j\le m$ or $j>m$. If $j\le m$, then $\sset(c_{q_m})\subset^* \sset(c_{q_j})$ and by Definition \ref{d14}(\ref{i16}) applied to $q_m$, $\pi_{{q}_{m},\g_n}''\sset(c_{q_j})\in \cu_{\g_n}$. If $j>m$, then $\g_n\in X_{q_j}$ and ${\pi}_{{q}_{j}, {\gamma}_{n}} = {\pi}_{{q}_{m}, {\gamma}_{n}}$; so we have that $\pi_{{q}_{m},\g_n}''\sset\left({c}_{{q}_{j}}\right)\in \cu_{\g_n}$. Next, we prove (\ref{i78}). Fix $n\le m<\o$. Let $k < \omega$ be minimal such that $\g_n\in X_{q_k}$ and $l < \omega$ minimal such that $\g_m\in {X}_{{q}_{l}}$. Define $j=\max\set{k,l}$. Then $\g_n,\g_m\in X_{q_j}$, and ${\pi}_{{q}_{k}, {\gamma}_{n}} = {\pi}_{{q}_{j}, {\gamma}_{n}}$ and ${\pi}_{{q}_{l}, {\gamma}_{m}} = {\pi}_{{q}_{j}, {\gamma}_{m}}$. By Definition \ref{d14}(\ref{i17}) applied to $q_j$ we have that $\forall^{\infty}k\in \sset(c_{q_j})\ [{\pi}_{{q}_{j}, {\gamma}_{n}}(k)= \pi_{\g_m,\g_n}({\pi}_{{q}_{j}, {\gamma}_{m}}(k))]$. Hence $j$ witnesses (\ref{i78}).
Finally for (\ref{i79}), fix $n < \omega$ and let $m < \omega$ be minimal such that $\g_n\in X_{q_m}$. Since $q_m$ satisfies Definition \ref{d14}(\ref{i15}) we know that $c_{q_m}\le b_{q_m,\g_n}$ and $\seq{\pi_{q_m,\g_n},\psi_{q_m,\g_n},b_{q_m,\g_n}}$ is a normal triple. So (\ref{i79}) is witnessed by $j = m$.

As explained above, the assumptions of Lemma \ref{t30} are satisfied, so there are $e^*,d^*\in\P$ and $\pi,\psi\in\o^{\o}$ which satisfy conditions (\ref{i84}-\ref{i83}) in the conclusion of Lemma \ref{t30}. Consider $q=\seq{d^*,\d,Y,\seq{\pi_{q,\a}:\a\in Y}}$, where for $\alpha \in Y$, $\pi_{q,\a}=\pi_{q_m,\a}$ for the minimal $m<\o$ such that $\a\in X_{q_m}$.
Note that for each $n < \omega$, ${\pi}_{{\gamma}_{n}} = {\pi}_{q, {\gamma}_{n}}$.
If we prove that $q\in \Q^{\d}$ it will follow easily that $q\le q_n$ for $n<\o$ and that $q$ satisfies the moreover part of the lemma. So we check the properties (\ref{i11}-\ref{i14}) of Definition \ref{d14}. Conditions (\ref{i11}-\ref{i13}) are clearly satisfied. We prove (\ref{i16}-\ref{i15}). First we show that (\ref{i17}) is true. Let $\a,\b\in Y$ be such that $\a\le\b$.
Let $m < \omega$ and $k < \omega$ be minimal with $\alpha \in {X}_{{q}_{m}}$ and $\beta \in {X}_{{q}_{k}}$ respectively.
Put $l = \max\{m, k\}$, and note that ${\pi}_{q, \alpha} = {\pi}_{{q}_{l}, \alpha}$ and ${\pi}_{q, \beta} = {\pi}_{{q}_{l}, \beta}$. 
So by definition \ref{d14}(\ref{i17}) applied to ${q}_{l}$ and by the fact that $\sset({d}^{\ast}) \; {\subset}^{\ast} \; \sset({c}_{{q}_{l}})$, $\forallbutfin {k}^{\ast} \in \sset({d}^{\ast})\left[ {\pi}_{q, \alpha}({k}^{\ast}) = {\pi}_{\beta, \alpha}({\pi}_{q, \beta}({k}^{\ast}))\right]$, as required. 
Now we prove (\ref{i16}).
Fix $\a\in Y$, and let $n < \omega$ be such that ${\gamma}_{n} \geq \alpha$.
Note that ${\pi}_{\delta, {\gamma}_{n}}''\sset({e}^{\ast}) \in {\cu}_{{\gamma}_{n}}$ and that ${\pi}_{\delta, {\gamma}_{n}}''\sset({e}^{\ast}) \; {\subset}^{\ast} \; {\pi}_{{\gamma}_{n}}''\sset({d}^{\ast})$.
Thus ${\pi}_{q, {\gamma}_{n}}''\sset({d}^{\ast}) \in {\cu}_{{\gamma}_{n}}$, and so ${\pi}_{{\gamma}_{n}, \alpha}''{\pi}_{q, {\gamma}_{n}}''\sset({d}^{\ast}) \in {\cu}_{\alpha}$.
By (\ref{i17}), ${\pi}_{{\gamma}_{n}, \alpha}''{\pi}_{q, {\gamma}_{n}}''\sset({d}^{\ast}) \; {\subset}^{\ast} \; {\pi}_{q, \alpha}''\sset({d}^{\ast})$, whence ${\pi}_{q, \alpha}''\sset({d}^{\ast}) \in {\cu}_{\alpha}$ as needed.
Finally for (\ref{i15}), fix $\a\in Y$ and let $m < \omega$ be minimal such that $\a\in X_{q_m}$. Then setting ${b}_{q, \alpha} = {b}_{{q}_{m}, \alpha}$ and ${\psi}_{q, \alpha} = {\psi}_{{q}_{m}, \alpha}$ fulfills (\ref{i15}).
\end{proof}
${\Q}^{\delta}$ is required to be $< \c$ closed when carrying out the constructing under $\MA$.
This can be proved in the same way as Theorem \ref{t91} by using the appropriate generalizations of the lemmas from Section \ref{sec:mainlemmas} and the regularity of $\c$, which follows from $\MA$.

We next turn towards showing that various sets are dense in ${\Q}^{\delta}$.
These are the dense sets we will want to meet when building our ``sufficiently generic'' filter for ${\Q}^{\delta}$.
Meeting these dense sets will ensure that the sequences $\langle {c}^{\delta}_{i}: i < \c \rangle$ and $\langle {\pi}_{\delta, \alpha}: \alpha \leq \delta \rangle$, which we intend to read off from the generic filter, will satisfy the conditions of Definition \ref{d22} when they are added to $S$.
The first density condition states that for each $q \in {\Q}^{\delta}$, there is a $q' \leq q$ such that ${c}_{q'}$ is a ``fast'' subsequence of ${c}_{q}$.
This is needed to ensure that ${\cu}_{\delta}$ is rapid, and it will also play a role in ensuring that it is an ultrafilter.
\begin{lemma}\label{t5}
For $q\in \Q^{\d}$ and strictly increasing $f\in\o^{\o}$ there is $q'\le q$ such that ${X}_{q} = {X}_{q'}$ and that for every $n<\o$ there is $m\ge f(n)$ so that $c_{q'}(n)=c_q(m)$. Moreover, there is $q''\le q'$ such that for every $n<\o$ we have $c_{q''}(n)\in [c_{q'}(n)]^{n+1}$.
\end{lemma}

\begin{proof}
We first show how to get $q'$.
We will distinguish two cases: when $\g_q=\d$ and when $\g_q<\d$.

\vskip2mm\noindent
Case I: $\g_q=\d$. We know that $X_q\subset \d$, $\sup(X_q)=\g_q=\d$ and $\abs{X_q}\le \o$, so either $\d=0$ or $\d$ is a limit ordinal with $\cf(\d)=\o$.

\vskip1mm\noindent
Subcase Ia: $\g_q=\d$ and $\d=0$.
In this case $q$ is of the form $\seq{c_q,0,0,0}$.
For every $n < \omega$, let ${c}_{q'}(n) = {c}_{q}(f(n))$.
Then ${c}_{q'} \in \P$ because $f$ is strictly increasing.
Also it is clear that ${c}_{q'} \leq {c}_{q}$.
Consequently $q'=\seq{c_{q'},0,0,0}\le q$ is as required.

\vskip1mm\noindent
Subcase 1b: $\g_q=\d$ and $\cf(\d)=\o$. We apply Lemma \ref{t30} in such a way that: $e$ is $c_q$ and $\d$ is $\d$; $d_n$ is $c_q$ for $n<\o$; $f$ is $f$; $X$ is $X_q$; maps $\pi_{\a}$ are maps $\pi_{q,\a}$ ($\a\in X_q$); maps $\pi_{\d,\a}$ are maps $\pi_{q,\a}$ ($\a\in X_q$). The conditions of Lemma \ref{t30} are clearly satisfied. Hence, there is $e^*\le_0 c_q$ such that for every $n<\o$ there is $m\ge f(n)$ so that $e^*(n)\subset c_q(m)$. We will construct numbers $k_n$ by induction on $n$ so that for every $n<\o$ there is $m\ge n$ so that $e^*(m)\subset c_q(k_n)$ and that $\sset(e^*)\subset \bigcup_{n<\o}c_q(k_n)$. Let $k_0$ be such that $e^*(0)\subset c_{q}(k_0)$. Now assume that numbers $k_m$ have been chosen for every $m\le n$, and define $k_{n+1}$ as follows: let $l$ be maximal such that $e^*(l)\subset c_q(k_n)$ and define $k_{n+1}$ as the unique number such that $e^*(l+1)\subset c_q(k_{n+1})$. Now for every $n<\o$ define $c_{q'}(n)=c_{q}(k_n)$. We will prove that the condition $q'=\seq{c_{q'},\d,X_q,\seq{\pi_{q,a}:\a\in X_q}}$ is as required. Since for every $n$, $e^*(n)\subset c_q(m)$ for $m\ge f(n)$ we have that $k_n\ge f(n)$, so $\forall n<\o\ \exists l\ge f(n)\ c_{q'}(n)=c_q(l)$, as required in the statement of the lemma. By Remark \ref{r101}, in order to prove $q'\in \Q^{\d}$ and $q'\le q$ it is enough to prove that $q'$ satisfies Definition \ref{d14}(\ref{i16}). So pick $\a\in X_q$. Since $\sset(e^*)\subset \bigcup_{n<\o}c_q(k_n)=\sset(c_{q'})$ we know that $\pi''_{q,\a}\sset(e^*)\subset^* \pi''_{q,\a}\sset(c_{q'})$, but since $\pi''_{q,\a}\sset(e^*)\in \cu_{\a}$ we have $\pi''_{q',\a}\sset(c_{q'})\in \cu_{\a}$ as required.

\vskip2mm\noindent
Case II: ${\gamma}_{q} < \delta$.
Note that by Definition \ref{d14}(\ref{i13}) ${\gamma}_{q} \in {X}_{q}$.
Let ${n}_{0}$ be such that ${c}_{q} \; {\leq}_{{n}_{0}} \; {b}_{q, {\gamma}_{q}}$.
Then $a = {\pi}_{q, {\gamma}_{q}}'' \sset({c}_{q})\lp {n}_{0} \rp \in {\cu}_{{\gamma}_{q}}$.
Now by Lemma \ref{t14}, for each $n < \omega$, ${m}_{n} = \max\{m < \omega: {\pi}_{q, {\gamma}_{q}}''{c}_{q}(m) = \{a(n)\} \}$ is well-defined and ${m}_{n} < {m}_{n + 1}$.
As ${\cu}_{{\gamma}_{q}}$ is rapid, there is $Y \in {\cu}_{{\gamma}_{q}}$ such that $Y \subset a$ and for each $n \in \omega$, there is ${l}_{n} \geq f(n)$ such that $Y(n) = a({l}_{n})$.
Now it is clear that for each $n \in \omega$, ${m}_{{l}_{n}} \geq {l}_{n} \geq f(n)$.
Define ${c}_{q'}(n) = {c}_{q}({m}_{{l}_{n}})$.
It is clear that ${c}_{q'} \in \P$ and that ${\pi}_{q, {\gamma}_{q}}''\sset({c}_{q'}) = Y$.
So by Remark \ref{r101}, we will finish the proof by showing that $q'=\seq{c_{q'},\g_q,X_q,\seq{\pi_{q,\a}:\a\in X_q}}$ satisfies Definition \ref{d14}(\ref{i16}). So let $\a\in X_q$. We know that $\pi''_{q,\g_q}\sset(c_{q'})=Y\in \cu_{\g_q}$ so $\pi''_{\g_q,\a}Y\in \cu_{\a}$. Now we have that $\pi''_{q,\a}\sset(c_{q'})=^*\pi''_{\g_q,\a}(\pi''_{q,\g_q}\sset(c_{q'}))=\pi''_{\g_q,\a}Y\in \cu_{\a}$ as required.

To get $q''$, define $c_{q''}$ as follows: for every $n<\o$ pick an arbitrary $c_{q''}(n)\in [c_{q'}(n)]^{n+1}$. 
This is possible because $\left| {c}_{q'}(n) \right| \geq n + 1$.
Let $q''=\seq{c_{q''},\g_q,X_q,\seq{\pi_{q,\a}:\a\in X_q}}$. 
To see that $q''\in \Q^{\d}$ note that conditions (\ref{i11}-\ref{i13}), (\ref{i17}), and (\ref{i15}) of Definition \ref{d14} are clearly satisfied.
Condition (\ref{i16}) holds because for every $\a\in X_q$, $\pi''_{q,\a}\sset(c_{q''}) \; {=}^{\ast} \; \pi''_{q,\a}\sset(c_{q'})$.
\end{proof}
The next lemma ensures that for any given $X \in \cp(\omega)$, every condition in ${\Q}^{\delta}$ has an extension that ``decides'' $X$.
This will make ${\cu}_{\delta}$ into an ultrafilter.
\begin{lemma}\label{t3}
For every $q\in\Q^{\d}$ and for every $X\in \cp(\o)$ there is $q'\le q$ such that $X_{q'}=X_q$ and that $\sset(c_{q'})\subset X$ or $\sset(c_{q'})\subset \o\setminus X$.
\end{lemma}

\begin{proof}
In the same way as in the proof of Lemma \ref{t5} we distinguish the following cases: either $\g_q=\d=0$ or $\g_q=\d$ and $\cf(\d)=\o$ or $\g_q<\d$.

\vskip2mm\noindent
Case I: $\g_q=\d$. As already mentioned this case has two subcases.

\vskip1mm\noindent
Subcase Ia: $\g_q=\d=0$. In this case $q$ is of the form $q=\seq{c_q,0,0,0}$. For $i=0,1$ consider the sets $X_i=\set{n<\o: \abs{c_q(n)\cap X^i}\ge (n+1)/2}$. Note $X_0\cup X_1=\o$ so either $X_0$ or $X_1$ infinite. Assume without loss of generality that $X_0$ is infinite. Then $\abs{c_q(X_0(2n+1))\cap X}\ge n+1$ for every $n<\o$. Define $c_{q'}\in \P$ as follows: for $n<\o$ let $c_{q'}(n)=[c_q(X_0(2n+1))\cap X]^{n+1}$. It easy to see that $c_{q'}\in \P$ and $c_{q'}\le c_q$. So for $q'=\seq{c_{q'},0,0,0}$ we have $q'\le q$ and $\sset(c_{q'})\subset X$. If we assumed $X_1$ is infinite, then we would obtain $\sset(c_{q'})\subset \o\setminus X$.

\vskip1mm\noindent
Subcase Ib: $\g_q=\d$ and $\cf(\d)=\o$. First according to Lemma \ref{t5} there is $q'\le q$ such that $X_{q'}=X_q$ and that for every $n<\o$ there is $m\ge 2^{n+1}$ such that $c_{q'}(n)=c_q(m)$. Note that this implies that for every $n<\o$ we have $\abs{c_{q'}(n)}\ge 2^{n+1}$. Let us consider two sets $A_i=\set{n<\o:\abs{c_{q'}(n)\cap X^i}\ge 2^n}$ ($i=0,1$). Fix $\a\in X_{q'}$. Because $q'\in \Q^{\d}$ we have $\pi''_{q',\a}\sset(c_{q'})\in \cu_{\a}$. So since $A_0\cup A_1=\o$ we have that
\[\TS
\pi''_{q',\a}\br{\bigcup_{n\in A_0}c_{q'}(n)}\cup\pi''_{q',\a}\br{\bigcup_{n\in A_1}c_{q'}(n)}\in \cu_{\a}.
\]
Since $\cu_{\a}$ is an ultrafilter, there is $i_{\a}\in 2$ such that $\pi''_{q',\a}\br{\bigcup_{n\in A_{i_{\a}}}c_{q'}(n)}\in \cu_{\a}$. Now that we have defined $i_{\a}$ for every $\a\in X_{q'}$, pick ordinals $\b_n\in X_{q'}$ so that the sequence $\seq{\b_n:n<\o}$ is strictly increasing and cofinal in $\d$.
There is $K\in [\o]^{\o}$ and $i\in \set{0,1}$ so that $i_{\b_n}=i$ for every $n\in K$. Now pick any $\b\in X_{q'}$. Because $K$ is infinite and $\seq{\b_n:n<\o}$ is cofinal in $\d$, there is $n\in K$ so that $\b_n>\b$. We know that $\pi''_{q',\b_n}\br{\bigcup_{n\in A_i}c_{q'}(n)}\in \cu_{\b_n}$. But according to Definition \ref{d22}(\ref{i5}) and Definition \ref{d14}(\ref{i17})we have
\[\TS
\pi''_{q',\b}\br{\bigcup_{n\in A_i}c_{q'}(n)}\supset^* \pi''_{\b_n,\b}\br{\pi''_{q',\b_n}\br{\bigcup_{n\in A_i}c_{q'}(n)}}\in \cu_{\b}
\]
which shows that for every $\b\in X_{q'}$ we have that $\pi''_{q',\b}\br{\bigcup_{n\in A_i}c_{q'}(n)}\in\cu_{\b}$. Now define $d\in\P$ as follows: for every $n<\o$ pick arbitrary $d(n)\in [c_{q'}(A_i(n))\cap X^i]^{n+1}$. The sequence $d=\seq{d(n):n<\o}$ belongs to $\P$ because $A_i$ was chosen in such a way that for $n<\o$ we have $\abs{c_{q'}(A_i(n))\cap X^i}\ge 2^n\ge n+1$. Finally, we will show that $q''=\seq{d,\g_{q'},X_{q},\seq{\pi_{q',\a}:\a\in X_{q'}}}$ is as required (note $X_{q''}=X_{q'}=X_q$). It is enough to prove that $q''\in \Q^{\d}$, because then $q''\le q$ and $\sset(c_{q''})=\sset(d)\subset X^i$ easily follows. By Remark \ref{r101} it is enough to show that Definition \ref{d14}(\ref{i16}) is satisfied. We show that $\pi''_{q'',\b}\sset(d)\supset^* \pi''_{q',\b}\br{\bigcup_{n\in A_i}c_{q'}(n)}\in\cu_{\b}$ holds for $\b\in X_{q'}$. Consider the set $C=\pi''_{q',\b}\br{\bigcup_{n\in A_i}c_{q'}(n)}\setminus \pi''_{q'',\b}\sset(d)$. Let $m<\o$ be such that $c_{q'}\le_m b_{q',\b}$. Note that for any $n\ge m$ we have $\pi''_{q'',\b}d(n)=\pi''_{q',\b}c_{q'}(A_i(n))$. This implies that $C\subset \pi''_{q',\b}\br{\bigcup_{n<A_i(m)}c_{q'}(n)}$ which shows that $\abs{C}<\o$ as required.

\vskip2mm\noindent
Case II: ${\gamma}_{q} < \delta$.
Let $q' \leq q$ be such that ${X}_{q'} = {X}_{q}$ and that for each $n \in \omega$, ${c}_{q'}(n) = {c}_{q}(m)$ for some $m \geq 2n + 1$.
Note that ${\gamma}_{q} = {\gamma}_{q'} \in {X}_{q'}$ and that for each $n \in \omega$, $\left| {c}_{q'}(n) \right| \geq 2n + 2$.
For $i \in 2$, let ${X}_{i} = \left\{n \in \omega: \left| {X}^{i} \cap {c}_{q'}(n) \right| \geq n + 1\right\}$.
Note that $\omega = {X}_{0} \cup {X}_{1}$.
Therefore $\left( {\pi}_{q', {\gamma}_{q}}'' {\bigcup}_{n \in {X}_{0}}{{c}_{q'}(n)} \right) \cup \left( {\pi}_{q', {\gamma}_{q}}'' {\bigcup}_{n \in {X}_{1}}{{c}_{q'}(n)} \right) = {\pi}_{q', {\gamma}_{q}}''\sset({c}_{q'}) \in {\cu}_{{\gamma}_{q}}$.
Fix $i \in 2$ such that ${\pi}_{q', {\gamma}_{q}}''{\bigcup}_{n \in {X}_{i}}{{c}_{q'}(n)} \in {\cu}_{{\gamma}_{q}}$.
Then ${X}_{i}$ is infinite and $\left| {c}_{q'}({X}_{i}(k)) \cap {X}^{i} \right| \geq {X}_{i}(k) + 1 \geq k + 1$, for each $k \in \omega$.
Choose ${c}_{q''}(k) \in {\left[ {c}_{q'}({X}_{i}(k)) \cap {X}^{i} \right]}^{k + 1}$.
Then ${c}_{q''} = \langle {c}_{q''}(k): k \in \omega \rangle \in \P$ and ${c}_{q''} \leq {c}_{q'}$.
Moreover, ${\pi}_{q', {\gamma}_{q}}''{\bigcup}_{n \in {X}_{i}}{{c}_{q'}(n)} \; {\subset}^{\ast} \;{\pi}_{q', {\gamma}_{q}}''\sset({c}_{q''})$.
Thus ${\pi}_{q', {\gamma}_{q}}''\sset({c}_{q''}) \in {\cu}_{{\gamma}_{q}}$.
Furthermore, ${\pi}_{{\gamma}_{q}, \alpha}'' {\pi}_{q', {\gamma}_{q}}''\sset({c}_{q''}) \; {\subset}^{\ast} \; {\pi}_{q', \alpha}''\sset({c}_{q''})$, for each $\alpha \in {X}_{q'}$. 
So we also have that ${\pi}_{q', \alpha}''\sset({c}_{q''}) \in {\cu}_{\alpha}$, for each $\alpha \in {X}_{q'}$.
Therefore by Remark \ref{r101} $q'' = \langle {c}_{q''}, {\gamma}_{q}, {X}_{q'}, \langle {\pi}_{q', \alpha}: \alpha \in {X}_{q'} \rangle \rangle$ is as required.
\end{proof}
We would like it to be the case that for each $\beta < \delta$, there is a $q$ in our ``sufficiently generic'' filter over ${\Q}^{\delta}$ with $\beta \in {X}_{q}$ because we would like to read the map ${\pi}_{\delta, \beta}$ from the filter.
So we next prove that for each $\beta < \delta$, every $q \in {\Q}^{\delta}$ has an extension $q'$ with $\beta \in {X}_{q'}$.
But let us first interject two technical lemmas that are easy to prove.
\begin{lemma}\label{t10}
For $q\in \Q^{\d}$, $\a\in X_q$ and $a\in \cu_{\a}$ there is $q'\le q$ such that $X_{q'}=X_q$ and $\pi_{q',\a}''\sset(c_{q'})\subset a$.
\end{lemma}

\begin{proof}
Consider the set $b=\pi_{q,\a}''\sset(c_q)$. By Definition \ref{d14}(\ref{i16}) $b\in \cu_{\a}$, which implies that $a\cap b\in \cu_{\a}$. Denote $V=\pi_{q,\a}^{-1}(a\cap b)$. By Lemma \ref{t3} there is $q'\le q$ such that $\sset(c_{q'})\subset V$ or $\sset(c_{q'})\subset \o\setminus V$. Assume that $\sset(c_{q'})\subset \o\setminus V$. By Definition \ref{d14}(\ref{i16}) we have $\pi_{q',\a}''\sset(c_{q'})\in \cu_{\a}$  and $\pi_{q',\a}=\pi_{q,\a}$. So $(\pi_{q',\a}''\sset(c_{q'}))\cap (a\cap b)=0$ which is impossible. Hence $\sset(c_{q'})\subset V$ implying $\pi_{q',\a}''\sset(c_{q'})\subset a\cap b$ as required.
\end{proof}

\begin{lemma}\label{t0}
Let $q\in \Q^{\d}$, $\b \in \delta$, and $Y=X_q\cup\set{\b}$. There is $q'\le q$ such that $X_{q'}=X_q$ and that for every $\zeta,\xi,\mu$ satisfying $\mu\in X_{q'}$, $\zeta,\xi\in Y$, and $\zeta\le \xi\le \mu$, there is $N<\o$ such that for every $k,l\in \sset(c_{q'})$ if $N\le k\le l$, then $\pi_{\mu,\zeta}(\pi_{q',\mu}(k))=\pi_{\xi,\zeta}(\pi_{\mu,\xi}(\pi_{q',\mu}(k)))$ and $\pi_{\mu,\xi}(\pi_{q',\mu}(k))\le\pi_{\mu,\xi}(\pi_{q',\mu}(l))$.
\end{lemma}
\begin{proof}
Let $V = \{\langle \zeta, \xi, \mu \rangle: \mu \in {X}_{q} \wedge \zeta, \xi \in Y \wedge \zeta \leq \xi \leq \mu\}$.
$V$ is countable, so let $\{\langle {\zeta}_{n}, {\xi}_{n}, {\mu}_{n}\rangle: n < \omega\}$ enumerate it, possibly with repetitions.
Build by induction on $n$ a decreasing sequence $\langle {q}_{n}: n \in \omega \rangle \subset {\Q}^{\delta}$ such that $\forall n \in \omega \left[{X}_{{q}_{n + 1}} = {X}_{{q}_{n}}\right]$.
Let ${q}_{0} = q$.
Fix $n \in \omega$, and suppose that ${q}_{n} \leq q$ is given.
By the definition of a $\delta$-generic sequence, there exists ${a}_{n} \in {\cu}_{{\mu}_{n}}$ such that $\forall {k}^{\ast} \in {a}_{n}\left[{\pi}_{{\mu}_{n}, {\zeta}_{n}}({k}^{\ast}) = {\pi}_{{\xi}_{n}, {\zeta}_{n}}({\pi}_{{\mu}_{n}, {\xi}_{n}}({k}^{\ast}))\right]$ and $\forall {k}^{\ast}, {l}^{\ast} \in {a}_{n}\left[{k}^{\ast} \leq {l}^{\ast} \implies {\pi}_{{\mu}_{n}, {\xi}_{n}}({k}^{\ast}) \leq {\pi}_{{\mu}_{n}, {\xi}_{n}}({l}^{\ast}) \right]$.
Apply Lemma \ref{t10} to ${q}_{n} \in {\Q}^{\delta}$, ${\mu}_{n} \in {X}_{{q}_{n}}$, and ${a}_{n} \in {\cu}_{{\mu}_{n}}$, to find ${q}_{n + 1} \leq {q}_{n}$ such that ${\pi}_{{q}_{n + 1}, {\mu}_{n}}''\sset({c}_{{q}_{n + 1}}) \subset {a}_{n}$ and ${X}_{{q}_{n + 1}} = {X}_{{q}_{n}}$.
This concludes the construction of $\langle {q}_{n}: n \in \omega \rangle$.
Find $q' \in {\Q}^{\delta}$ such that $\forall n \in \omega\left[q' \leq {q}_{n}\right]$ and ${X}_{q'} = {X}_{{q}_{0}} = {X}_{q}$.
We check that $q'$ is as needed.
Fix $n < \omega$.
As ${\mu}_{n} \in {X}_{q'}$ and $q' \leq {q}_{n + 1}$, there is $N$ such that for all $k, l \in \sset({c}_{q'})$, if $N \leq k \leq l$, then $k, l \in \sset({c}_{{q}_{n + 1}})$ and ${\pi}_{q', {\mu}_{n}}(k) \leq {\pi}_{q', {\mu}_{n}}(l)$.
Fixing any such $k$ and $l$, let ${k}^{\ast} = {\pi}_{q', {\mu}_{n}}(k)$ and ${l}^{\ast} = {\pi}_{q', {\mu}_{n}}(l)$.
Then ${k}^{\ast}, {l}^{\ast}\in {a}_{n}$ and ${k}^{\ast} \leq {l}^{\ast}$.
Therefore, ${\pi}_{{\mu}_{n}, {\zeta}_{n}}({k}^{\ast}) = {\pi}_{{\xi}_{n}, {\zeta}_{n}}({\pi}_{{\mu}_{n}, {\xi}_{n}}({k}^{\ast}))$ and ${\pi}_{{\mu}_{n}, {\xi}_{n}}({k}^{\ast}) \leq {\pi}_{{\mu}_{n}, {\xi}_{n}}({l}^{\ast})$, as needed. 
\end{proof}
\begin{lemma}\label{t2}
For $q\in \Q^{\d}$ and $\b<\d$ there is $q'\in \Q^{\d}$ such that $q'\le q$ and $\b\in X_{q'}$.
\end{lemma}

\begin{proof}
Assume $\b\notin X_q$. According to Lemma \ref{t0} applied to $q$ and $\b$ there is $q^*\le q$ such that $X_{q^*}=X_q$ and that for every $\zeta,\xi,\mu$ satisfying $\mu \in {X}_{{q}^{\ast}}$, $\zeta, \xi \in {X}_{{q}^{\ast}} \cup \set{\b}$ and $\zeta\le \xi\le \mu$ there is $N<\o$ such that for every $k,l\in \sset(c_{q^*})$ if $N\le k\le l$, then $\pi_{\mu,\zeta}(\pi_{q^*,\mu}(k))=\pi_{\xi,\zeta}(\pi_{\mu,\xi}(\pi_{q^*,\mu}(k)))$ and $\pi_{\mu,\xi}(\pi_{q^*,\mu}(k))\le\pi_{\mu,\xi}(\pi_{q^*,\mu}(l))$. In the same way as in the proof of Lemma \ref{t5} we have the following cases: either $\g_{q^*}=\d=0$ or $\g_{q^*}=\d$ and $\cf(\d)=\o$ or $\g_{q^*}<\d$.

\vskip2mm\noindent
Case I: $\g_{q^*}=\d$. As we mentioned above there are two subcases.

\vskip1mm\noindent
Subcase Ia: $\g_{q^*}=\d=0$. Note that in this case the statement is vacuous because there is no $\b<\d$.

\vskip1mm\noindent
Subcase Ib: $\g_{q^*}=\d$ and $\cf(\d)=\o$. Since $\sup(X_{q^*})=\g_{q^*}=\d$ and $\b<\d$, let $\g^*\in X_{q^*}$ be minimal such that $\b\le \g^*$.
Let $m < \omega$ be minimal such that ${c}_{{q}^{\ast}} \; {\leq}_{m} \; {b}_{{q}^{\ast}, {\gamma}^{\ast}}$ and for all $k, l \in \sset({c}_{{q}^{\ast}})\lp m \rp$, if $k \leq l$, then ${\pi}_{{\gamma}^{\ast}, \beta}({\pi}_{{q}^{\ast}, {\gamma}^{\ast}}(k)) \leq {\pi}_{{\gamma}^{\ast}, \beta}({\pi}_{{q}^{\ast}, {\gamma}^{\ast}}(l))$.
Define
\[
q'=\seq{c_{q^*},\d,X_{q^*}\cup\set{\b},\seq{\pi_{q^*,\a}:\a\in X_{q^*}\cup\set{\b}}},
\]
where $\pi_{q^*,\b}$ is as follows: for $k\in \sset(c_{q^*})\lp m\rp$ let $\pi_{q^*,\b}(k)=\pi_{\g^*,\b}(\pi_{q^*,\g^*}(k))$, while $\pi_{q^*,\b}(k)=0$ otherwise.
It suffices to prove that $q'\in \Q^{\d}$ because it is then easy to see that $q'\le {q^*}$ and $\b\in X_{q'}$ hold.
Properties (\ref{i11}-\ref{i13}) are clearly satisfied.
So we check (\ref{i14}).
First it is clear that (\ref{i16}) holds by the definition of ${\pi}_{{q}^{\ast}, \beta}$ and by the fact that ${c}_{q'} = {c}_{{q}^{\ast}}$.
Next, we check (\ref{i17}).
Pick arbitrary $\a, \g \in X_{q'}$ such that $\a \le \g$.
We will distinguish four cases: either ($\a \neq \b$ and $\g \neq \b$), or ($\alpha = \beta = \gamma$), or ($\g = \b$ and  $\alpha \neq \beta$), or ($\a = \b$ and $\gamma \neq \beta$).
First, if $\a \neq \b$ and $\g \neq \b$, then (\ref{i17}) holds because $q^*\in \Q^{\d}$ and $\alpha, \gamma \in {X}_{{q}^{\ast}}$.
Next, if $\alpha = \beta = \gamma$, then (\ref{i17}) trivially holds.
Now assume that $\g = \b$ and $\alpha \neq \beta$.
Then $\alpha \in {X}_{{q}^{\ast}}$.
There exists ${k}_{1}$ such that for each $k \in \sset({c}_{{q}^{\ast}})\lp {k}_{1} \rp$ the following hold:
$\pi_{q^*, \gamma}(k)=\pi_{\g^*, \gamma}(\pi_{q^*,\g^*}(k))$, ${\pi}_{{\gamma}^{\ast}, \alpha}({\pi}_{{q}^{\ast}, {\gamma}^{\ast}}(k)) = {\pi}_{\gamma, \alpha}({\pi}_{{\gamma}^{\ast}, \gamma}({\pi}_{{q}^{\ast}, {\gamma}^{\ast}}(k)))$, and ${\pi}_{{q}^{\ast}, \alpha}(k) = {\pi}_{{\gamma}^{\ast}, \alpha}({\pi}_{{q}^{\ast}, {\gamma}^{\ast}}(k))$.
Thus for every $k\in \sset(c_{q^*})\lp k_1\rp$ we have $\pi_{q^*,\a}(k)=\pi_{\g^*,\a}(\pi_{q^*,\g^*}(k))=\pi_{\gamma,\a}(\pi_{\g^*,\gamma}(\pi_{q^*,\g^*}(k)))=\pi_{\gamma,\a}(\pi_{q^*,\gamma}(k))$ as required.
Finally assume that $\alpha = \beta$ and $\gamma \neq \beta$.
Then $\gamma \in {X}_{{q}^{\ast}}$ and $\beta \leq \gamma$.
By minimality of ${\gamma}^{\ast}$, $\g^*\le \g$.
As before, there exists ${k}_{0} \in \omega$ such that for each $k \in \sset({c}_{{q}^{\ast}})\lp {k}_{0} \rp$ the following hold: $\pi_{q^*,\alpha}(k)=\pi_{\g^*,\alpha}(\pi_{q^*,\g^*}(k))$, ${\pi}_{\gamma, \alpha}({\pi}_{{q}^{\ast}, \gamma}(k)) = {\pi}_{{\gamma}^{\ast}, \alpha}({\pi}_{\gamma, {\gamma}^{\ast}}({\pi}_{{q}^{\ast}, \gamma}(k)))$, and ${\pi}_{{q}^{\ast}, {\gamma}^{\ast}}(k) = {\pi}_{\gamma, {\gamma}^{\ast}}({\pi}_{{q}^{\ast}, \gamma}(k))$.
Thus for every $k \in \sset({c}_{{q}^{\ast}})\lp {k}_{0} \rp$ we have $\pi_{q^*,\alpha}(k)=\pi_{{\gamma}^{\ast},\alpha}(\pi_{q^*,\g^*}(k))=\pi_{\g^*,\alpha}(\pi_{\g,\g^*}(\pi_{q^*,\g}(k)))=\pi_{\g,\alpha}(\pi_{q^*,\g}(k))$ as required.
So (\ref{i17}) holds.
Finally, we check (\ref{i15}). If $\a\in X_{q^*}$ then (\ref{i15}) is true because $q^*\in \Q^{\d}$.
Let us now define ${b}_{{q}^{\ast}, \beta}$ and ${\psi}_{{q}^{\ast}, \beta}$.
Put ${b}_{{q}^{\ast}, \beta} = {c}_{{q}^{\ast}}$.
For each ${m}^{\ast} \geq m$, ${\pi}_{{q}^{\ast}, {\gamma}^{\ast}}$ is constant on ${c}_{{q}^{\ast}}({m}^{\ast})$ because ${c}_{{q}^{\ast}} \; {\leq}_{m} \; {b}_{{q}^{\ast}, {\gamma}^{\ast}}$.
So for each ${m}^{\ast} \geq m$, ${\pi}_{{q}^{\ast}, \beta}$ is constant on ${c}_{{q}^{\ast}}({m}^{\ast})$.
Also ${\pi}_{{\gamma}^{\ast}, \beta} \circ {\pi}_{{q}^{\ast}, {\gamma}^{\ast}}$ is increasing on $\sset({c}_{{q}^{\ast}})\lp m \rp$.
So for each ${m}^{\ast} \geq m$ define ${\psi}_{{q}^{\ast}, \beta}({m}^{\ast}) = {\pi}_{{\gamma}^{\ast}, \beta}({\pi}_{{q}^{\ast}, {\gamma}^{\ast}}(k)) = {\pi}_{{q}^{\ast}, \beta}(k)$, for an arbitrary $k \in {c}_{{q}^{\ast}}({m}^{\ast})$.
When ${m}^{\ast} < m$, ${\pi}_{{q}^{\ast}, \beta}$ is constantly equal to $0$ on ${c}_{{q}^{\ast}}({m}^{\ast})$.
So set ${\psi}_{{q}^{\ast}, \beta}({m}^{\ast}) = 0$, for ${m}^{\ast} < m$.
It is clear that $\langle {\pi}_{{q}^{\ast}, \beta}, {\psi}_{{q}^{\ast}, \beta} , {b}_{{q}^{\ast}, \beta} \rangle$ is a normal triple with ${c}_{{q}^{\ast}} \leq {b}_{{q}^{\ast}, \beta}$.

\vskip2mm\noindent
Case II: $\g_{q^*}<\d$.
Note that ${\gamma}_{{q}^{\ast}} \in {X}_{{q}^{\ast}}$
There are two subcases: when $\b<\g_{q^*}$ and when $\g_{q^*}<\b$.
If $\b<\g_{q^*}$, then define $\pi_{{q^*},\b}$ as follows.
Let $m_1$ be such that the following two things hold: $c_{q^*} {\leq}_{{m}_{1}} b_{q^*,\g_{q^*}}$; and for any $k, l \in \sset({c}_{q^*}) \lp m_1 \rp$, if $k \leq l$, then ${\pi}_{{\gamma}_{{q}^{\ast}}, \beta}({\pi}_{{q}^{\ast}, {\gamma}_{{q}^{\ast}}}(k)) \leq {\pi}_{{\gamma}_{{q}^{\ast}}, \beta}({\pi}_{{q}^{\ast}, {\gamma}_{{q}^{\ast}}}(l))$.
For $k\in \sset(c_{q^*})\lp m_1\rp$, define $\pi_{q^*,\b}(k) = \pi_{\g_{q^*},\b}(\pi_{q^*,\g_{q^*}}(k))$, while for $k \notin \sset(c_{q^*})\lp m_1\rp$, define $\pi_{q^*,\b}(k)=0$.
We will prove that $q'=\seq{c_{q^*},\g_{q^*},X_{q^*}\cup\set{\b},\seq{\pi_{{q^*},\a}:\a\in X_{q^*}\cup\set{\b}}}$ is as required. It is enough to show that $q'\in \Q^{\d}$, because then we will have that $q'\le {q^*}$ and $\b\in X_{q'}$. Clearly, conditions (\ref{i11}-\ref{i13}) are satisfied. To see that (\ref{i14}) is true, note that (\ref{i16}) is clear from the definition of ${\pi}_{{q}^{\ast}, \beta}$ and from that fact that ${c}_{q'} = {c}_{{q}^{\ast}}$.  
Next, we check (\ref{i17}) for $q'$.
Fix $\a,\g\in X_{q'}$ such that $\a\le \g$.
There are again four cases: either ($\b \neq \a$ and $\b \neq \g$), or ($\alpha = \beta = \gamma$), or ($\gamma = \beta$ and $\alpha \neq \beta$), or ($\gamma \neq \beta$ and $\alpha = \beta$).
If $\b\neq\a$ and $\b\neq\g$, then the statement follows directly from Definition \ref{d14}(\ref{i17}) applied to ${q^*}$.
The case when $\alpha  = \beta = \gamma$ is trivial.
Next, consider the case when $\gamma = \beta$ and $\alpha \neq \beta$.
Then $\alpha \in {X}_{{q}^{\ast}}$.
There exists ${k}_{2}$ such that for each $k \in \sset({c}_{{q}^{\ast}})\lp {k}_{2} \rp$ the following hold: ${\pi}_{{q}^{\ast}, \gamma}(k) = {\pi}_{{\gamma}_{{q}^{\ast}}, \gamma}({\pi}_{{q}^{\ast}, {\gamma}_{{q}^{\ast}}}(k))$, ${\pi}_{{\gamma}_{{q}^{\ast}}, \alpha}({\pi}_{{q}^{\ast}, {\gamma}_{{q}^{\ast}}}(k)) = {\pi}_{\gamma, \alpha}({\pi}_{{\gamma}_{{q}^{\ast}}, \gamma}({\pi}_{{q}^{\ast}, {\gamma}_{{q}^{\ast}}}(k)))$, and ${\pi}_{{q}^{\ast}, \alpha}(k) = {\pi}_{{\gamma}_{{q}^{\ast}}, \alpha}({\pi}_{{q}^{\ast}, {\gamma}_{{q}^{\ast}}}(k))$.
Thus for any $k \in \sset({c}_{{q}^{\ast}})\lp {k}_{2} \rp$, ${\pi}_{\gamma, \alpha}({\pi}_{{q}^{\ast}, \gamma}(k)) = {\pi}_{\gamma, \alpha}({\pi}_{{\gamma}_{{q}^{\ast}}, \gamma}({\pi}_{{q}^{\ast}, {\gamma}_{{q}^{\ast}}}(k))) = {\pi}_{{\gamma}_{{q}^{\ast}}, \alpha}({\pi}_{{q}^{\ast}, {\gamma}_{{q}^{\ast}}}(k)) = {\pi}_{{q}^{\ast}, \alpha}(k)$, as needed.
Finally suppose that $\gamma \neq \beta$ and $\alpha = \beta$.
Then $\gamma \in {X}_{{q}^{\ast}}$.
As before, there exists ${k}_{3}$ such that for each $k \in \sset({c}_{{q}^{\ast}})\lp {k}_{3} \rp$ the following hold:
${\pi}_{{q}^{\ast}, \alpha}(k) = {\pi}_{{\gamma}_{{q}^{\ast}}, \alpha}({\pi}_{{q}^{\ast}, {\gamma}_{{q}^{\ast}}}(k))$, ${\pi}_{{\gamma}_{{q}^{\ast}}, \alpha}({\pi}_{{q}^{\ast}, {\gamma}_{{q}^{\ast}}}(k)) = {\pi}_{\gamma, \alpha}({\pi}_{{\gamma}_{{q}^{\ast}}, \gamma}({\pi}_{{q}^{\ast}, {\gamma}_{{q}^{\ast}}}(k)))$, and ${\pi}_{{q}^{\ast}, \gamma}(k) = {\pi}_{{\gamma}_{{q}^{\ast}}, \gamma}({\pi}_{{q}^{\ast}, {\gamma}_{{q}^{\ast}}}(k))$.
Thus for $k \in \sset({c}_{{q}^{\ast}})\lp {k}_{3} \rp$, ${\pi}_{{q}^{\ast}, \alpha}(k) = {\pi}_{{\gamma}_{{q}^{\ast}}, \alpha}({\pi}_{{q}^{\ast}, {\gamma}_{{q}^{\ast}}}(k)) = {\pi}_{\gamma, \alpha}({\pi}_{{\gamma}_{{q}^{\ast}}, \gamma}({\pi}_{{q}^{\ast}, {\gamma}_{{q}^{\ast}}}(k))) = {\pi}_{\gamma, \alpha}({\pi}_{{q}^{\ast}, \gamma}(k))$, as required.
So (\ref{i17}) is checked, and we now check (\ref{i15}) for $q'$.
If $\alpha \in {X}_{{q}^{\ast}}$, then (\ref{i15}) is satisfied for $q'$ because it was satisfied for ${q}^{\ast}$.
It remains to define ${b}_{{q}^{\ast}, \beta}$ and ${\psi}_{{q}^{\ast}, \beta}$.
Put ${b}_{{q}^{\ast}, \beta} = {c}_{{q}^{\ast}}$.
Note that for each ${m}^{\ast} \geq {m}_{1}$, ${\pi}_{{q}^{\ast}, {\gamma}_{{q}^{\ast}}}$ is constant on ${c}_{{q}^{\ast}}({m}^{\ast})$ because ${c}_{{q}^{\ast}} \; {\leq}_{{m}_{1}} \; {b}_{{q}^{\ast}, {\gamma}_{{q}^{\ast}}}$.
So for each ${m}^{\ast} \geq {m}_{1}$, ${\pi}_{{q}^{\ast}, \beta}$ is constant on ${c}_{{q}^{\ast}}({m}^{\ast})$.
Also $ {\pi}_{{\gamma}_{{q}^{\ast}}, \beta} \circ {\pi}_{{q}^{\ast}, {\gamma}_{{q}^{\ast}}}$ is increasing on $\sset({c}_{{q}^{\ast}})\lp {m}_{1} \rp$.
So for ${m}^{\ast} \geq {m}_{1}$, define ${\psi}_{{q}^{\ast}, \beta}({m}^{\ast}) = {\pi}_{{\gamma}_{{q}^{\ast}}, \beta}({\pi}_{{q}^{\ast}, {\gamma}_{{q}^{\ast}}}(k)) = {\pi}_{{q}^{\ast}, \beta}(k)$, for an arbitrary $k \in {c}_{{q}^{\ast}}({m}^{\ast})$.
When ${m}^{\ast} < {m}_{1}$, ${\pi}_{{q}^{\ast}, \beta}$ is constantly equal to $0$ on ${c}_{{q}^{\ast}}({m}^{\ast})$.
So define ${\psi}_{{q}^{\ast}, \beta}({m}^{\ast}) = 0$, for ${m}^{\ast} < {m}_{1}$.
It is clear that $\langle {\pi}_{{q}^{\ast}, \beta}, {\psi}_{{q}^{\ast}, \beta}, {b}_{{q}^{\ast}, \beta} \rangle$ is a normal triple and that ${c}_{{q}^{\ast}} \leq {b}_{{q}^{\ast}, \beta}$.
Hence $q'$ is as required.

Now consider the case when $\b>\g_{q^*}$. For each $\a\in X_{q^*}$, since $\a\le\g_{q^*}<\b$, by Definition \ref{d22}(\ref{i0}) pick $a_{\a}\in \cu_{\b}$ so that $\forall k\in a_{\a}\ [\pi_{\b,\a}(k)=\pi_{\g_{q^*},\a}(\pi_{\b,\g_{q^*}}(k))]$. Since $X_{q^*}$ is countable and $\cu_{\b}$ is a P-point there is $a\in \cu_{\b}$ such that $a\subset^*a_{\a}$ for every $\a\in X_{q^*}$. Then we apply Definition \ref{d22}(\ref{i1}) with $\b$, $\a$ being $\g_{q^*}$, $c_{q^*}$ being $d$, $\pi_1$ being $\pi_{{q^*},\g_{q^*}}$, $b_1$ being $b_{{q^*},\g_{q^*}}$ and $\psi_1$ being $\psi_{{q^*},\g_{q^*}}$ and $a$. Note that hypothesis of Definition \ref{d22}(\ref{i1}) are satisfied. By Definition \ref{d22}(\ref{i1}) there are $b\in \cu_{\b}$, $\pi,\psi\in\o^{\o}$ and $d^*\le_0 c_{q^*}$ so that $b\subset^*a$, $\seq{\pi,\psi,d^*}$ is a normal triple, $\pi''\sset(d^*)=b$ and $\forall k\in \sset(d^*)\ [\pi_{{q^*},\g_{q^*}}(k)=\pi_{\b,\g_{q^*}}(\pi(k))]$. Denote $\pi_{{q^*},\b}=\pi$ and $\psi_{{q^*},\b}=\psi$. Now define $q'=\seq{d^*,\b,X_{q^*}\cup\set{\b},\seq{\pi_{{q^*},\a}:\a\in X_{q^*}\cup\set{\b}}}$. It is easy to see that if we prove that $q'\in \Q^{\d}$, then $q'\le {q^*}$ and $\b\in X_{q'}$ follow. So we check conditions of Definition \ref{d14}. Note that conditions (\ref{i11}-\ref{i13}) are clearly true. We still have to check Definition \ref{d14}(\ref{i14}). First note that (\ref{i15}) is satisfied for $\a\in X_{q^*}$ because $d^*\le c_{q^*}$, while it is true for $\b$ because $\seq{\pi,\psi,d^*}$ is a normal triple. To see that (\ref{i17}) is true let $\a,\g\in X_{q^*}\cup\set{\b}$ be such that $\a\le \g$. There are three cases: either $\a\neq\b$ and $\g\neq\b$ or $\a=\b$ or $\g=\b$. First note that if $\a=\b$, then it must also be $\g=\b$ and the statement holds. If $\a\neq\b$ and $\g\neq\b$ then by Definition \ref{d14} and because $\sset(d^*)\subset^*\sset(c_{q^*})$ we have $\forall^{\infty}k\in \sset(d^*)\ [\pi_{q^*,\a}(k)=\pi_{\g,\a}(\pi_{q^*,\g}(k))]$. If $\g=\b$ then because $b\subset^* a\subset^*a_{\a}$ and $\pi''\sset(d^*)=b$ we have that there is $k_0<\o$ such that for every $k\in \sset(d^*)\lp k_0\rp$ we have $\pi_{q^*,\a}(k)=\pi_{\g_{q^*},\a}(\pi_{q^*,\g_{q^*}}(k))=\pi_{\g_{q^*},\a}(\pi_{\b,\g_{q^*}}(\pi(k)))=\pi_{\b,\a}(\pi(k))=\pi_{\b,\a}(\pi_{q^*,\b}(k))$ as required. To see that (\ref{i16}) is true note that $\pi_{{q^*},\b}''\sset(d^*)=b\in \cu_{\b}$ and consequently $\pi_{\b,\a}''\br{\pi_{q^*,\b}''\sset(d^*)}\in \cu_{\b}$ for any $\a\in X_{q^*}$. Together with already proved (\ref{i17}), this implies $\pi_{\b,\a}''\br{\pi_{q^*,\b}''\sset(d^*)}\subset^*\pi_{{q^*},\a}''\sset(d^*)\in \cu_{\b}$ for $\a\in X_{q^*}$.
\end{proof}
The next lemma ensures that we can ``kill'' unwanted Tukey maps.
That is, if $\beta < \delta$ and $\phi: \cp(\omega) \rightarrow \cp(\omega)$ is a monotone map that is a potential witness for the unwanted Tukey reduction ${\cu}_{\delta} \; {\leq}_{T} \; {\cu}_{\beta}$, then we would like every condition in ${\Q}^{\delta}$ to have an extension forcing that $\phi$ is not such a witness.
\begin{lemma}\label{t4}
For any $q\in \Q^{\d}$, any $\b<\d$ and any monotone $\phi:\cp(\o)\to \cp(\o)$, if for every $A\in \cu_{\b}$, $\phi(A)\neq\ps$ then there is $q'\le q$ such that $\b\in X_{q'}$ and that for every $A\in\cu_{\b}$ we have $\phi(A) \not\subset \sset(c_{q'})$.
\end{lemma}

\begin{proof}
By Lemma \ref{t2} there is $q'\le q$ such that $\b\in X_{q'}$ and by Lemma \ref{t5} there is $q''\le q'\le q$ such that for every $n<\o$ there is $m\ge 2n+1$ such that $c_{q''}(n)={c}_{q'}(m)$. For every $n<\o$ choose sets $d_1(n)$ and $d_2(n)$ which are elements of $[c_{q''}(n)]^{n+1}$ and are such that $d_1(n)\cap d_2(n)=0$. This can be done because $\abs{c_{q''}(n)}\ge 2n+2$. Note that both $q_1=\seq{d_1,\g_{q''},X_{q''},\seq{\pi_{q'',\a}:\a\in X_{q''}}}$ and $q_2=\seq{d_2,\g_{q''},X_{q''},\seq{\pi_{q'',\a}:\a\in X_{q''}}}$ belong to $\Q^{\d}$ and that $q_1,q_2\le q''\le q$. Now we consider two cases: either for every $A\in\cu_{\b}$, $\phi(A) \not\subset \sset(d_1)$, or there is some $A\in \cu_{\b}$ such that $\phi(A)\subset \sset(d_1)$. If for every $A\in\cu_{\b}$, $\phi(A) \not\subset \sset(d_1)$, then $q_1$ is as required. Otherwise, $q_2$ is as required because $\sset(d_1)\cap \sset(d_2)=0$ and $\phi$ is monotone.
\end{proof}
Note that $q'$ forces what we want because it forces $\sset({c}_{q'}) \in {\cu}_{\delta}$.
Hence it forces that the image of ${\cu}_{\beta}$ under $\phi$ is not cofinal in ${\cu}_{\delta}$.
It is also worth noting that the descriptive complexity of $\phi$ plays no role in the proof of Lemma \ref{t4}.
So Theorem \ref{thm:dt} is only needed for bounding the number of relevant maps.

The next lemma is needed for ensuring clause (\ref{i8}) of Definition \ref{d22}, and hence it is only relevant when $\cf(\delta) = \omega$.
It follows by a direct application of Lemma \ref{t30}.
\begin{lemma}\label{t9}
Suppose that $\cf(\d)=\o$, $q\in \Q^{\d}$ is such that $\g_q=\d$, $\seq{d_j:j<\o}$ is a decreasing sequence in $\P$, $X\subset X_q$ is such that $\sup(X)=\d $ and that $\seq{\pi_{\a}:\a\in X}$ is a sequence of maps in $\o^{\o}$ satisfying:
\begin{enumerate}
\item\label{i2001} $\forall \a\in X\ \forall j<\o\ [\pi_{\a}''\sset(d_j)\in \mathcal U_{\a}]$;
\item\label{i2002} $\forall \a,\b\in X\ [\a\le\b\sledi\exists j<\o\ \forall^{\infty} k\in \sset(d_j)\ [\pi_{\a}(k)=\pi_{\b,\a}(\pi_{\b}(k))]]$;
\item\label{i2003} for all $\a\in X$ there are $j<\o$ and $\psi_{\a}\in\o^{\o}$ and $b_{\a}\in \P$ such that $\seq{\pi_{\a},\psi_{\a},b_{\a}}$ is a normal triple and $d_j\le b_{\a}$.
\end{enumerate}
Then there are $q'\le q$, $d^*\in \P$ and $\pi:\o\to\o$ such that:
\begin{enumerate}[resume]
\item\label{i2005} $\forall j<\o\ [d^*\le d_j]$ and $\sset(c_{q'})=\pi''\sset(d^*)$;
\item\label{i2006} $\forall \a\in X\ \forall^{\infty}k\in \sset(d^*)\ [\pi_{\a}(k)=\pi_{q',\a}(\pi(k))]$;
\item\label{i2008} there is $\psi$ for which $\seq{\pi,\psi,d^*}$ is a normal triple.
\end{enumerate}
\end{lemma}

\begin{proof}
We will use Lemma \ref{t30} where: $\d$, $X$, $\seq{d_j:j<\o}$ and $\pi_{\a}$ ($\a\in X$) are as in the statement of this lemma, $f=\id$; $e$ is $c_q$; for $\a\in X$ map $\pi_{\d,\a}$ is $\pi_{q,\a}$ ($\a\in X$). So there are $e^*$, $d^*$ and $\pi$ satisfying properties (\ref{i84}-\ref{i83}) of the conclusion of Lemma \ref{t30}. We will show that $q'=\seq{e^*,\d,X_q,\seq{\pi_{q,\a}:\a\in X_q}}$, $d^*$ and $\pi$ are as required. The conditions (\ref{i2005}-\ref{i2008}) will be witnessed by conditions (\ref{i80}-\ref{i83}) in the conclusion of the Lemma \ref{t30}. By Remark \ref{r101}, in order to finish the proof we only have to show that Definition \ref{d14}(\ref{i16}) holds for $q'$. First assume that $\a\in X$. Then by Lemma \ref{t30}(\ref{i82}), $\pi_{q,\a}''\sset(e^*)\in \cu_{\a}$. Now assume that $\a\in X_q\setminus X$. Let $\a'\in X$ be such that $\a'>\a$. Then $\pi_{\a',\a}''\br{\pi_{q,\a'}''\sset(e^*)}\in \cu_{\a}$. Also we have $\pi_{\a',\a}''\br{\pi_{q,\a'}''\sset(e^*)} \subset^* \pi_{q,\a}''\sset(e^*)$. These observations together give us $\pi_{q,\a}''\sset(e^*)\in \cu_{\a}$ as required.
\end{proof}
Next we show how to make sure that ${\cu}_{\delta}$ is rapid.
This lemma follows from a direct application of Lemma \ref{t5}.
\begin{lemma}\label{t11}
Suppose that $\d<\o_2$, that $q\in \Q^{\d}$ and that $f\in\o^{\o}$ is a strictly increasing function. There is $q'\le q$ such that for every $n<\o$ we have $\sset(c_{q'})(n)\ge f(n)$.
\end{lemma}

\begin{proof}
According to Lemma \ref{t5} there is $q'\le q$ so that for every $n<\o$ there is $m\ge f(s(n+1))$ such that $c_{q'}(n)\in [c_q(m)]^{n+1}$. We will prove that $q'$ is as required. So fix $n<\o$, and let $k<\o$ be such that $\sset(c_{q'})(n)\in c_{q'}(k)$. Equivalently $s(k)\le n<s(k+1)$ which implies $\sset(c_{q'})(n)\ge \sset(c_{q'})(s(k))$. Since for some $m\ge f(s(k+1))$ we have $c_{q'}(k)\subset c_q(m)$ and $\sset(c_{q'})(n)\in c_{q'}(k)$, then by Remark \ref{r5}(\ref{i201}), $\sset(c_{q'})(n)\ge m\ge f(s(k+1))>f(n)$, the last inequality being true because $f$ is an increasing function. So we showed that $\sset(c_{q'})(n)\ge f(n)$ as required.
\end{proof}
We now come to the final density lemma.
This lemma ensures that clause (\ref{i1}) of Definition \ref{d22} can be satisfied during the construction of ${\cu}_{\delta}$.
One of the cases in its proof makes use of Lemma \ref{t30}.
\begin{lemma}\label{t12}
Let $q\in \Q^{\d}$, $\pi_1,\psi_1\in\o^{\o}$, ${b}_{1},d\in\P$, and $\a<\d$ be such that $\seq{\pi_1,\psi_1,{b}_{1}}$ is a normal triple, $d\le {b}_{1}$, and $\pi''_1\sset(d)\in\cu_{\a}$. Then there are $q^*\le q$, $d^*\le d$, $\pi,\psi\in\o^{\o}$ such that $\seq{\pi,\psi,d^*}$ is a normal triple, $\a\in X_{q^*}$, $\pi''\sset(d^*)=\sset(c_{q^*})$ and $\forall k\in \sset(d^*)\ [\pi_1(k)=\pi_{q^*,\a}(\pi(k))]$.
\end{lemma}

\begin{proof}
By Lemma \ref{t2} there is $q_0\le q$ such that $\a\in X_{q_0}$. In the same way as in the proof of Lemma \ref{t5} we have the following cases: either $\g_{q_0}=\d=0$ or $\g_{q_0}=\d$ and $\cf(\d)=\o$ or $\g_{q_0}<\d$.

\vskip2mm\noindent
Case I: $\g_{q_0}=\d$. As we mentioned above there are two subcases.

\vskip1mm\noindent
Subcase Ia: $\g_{q_0}=\d=0$. Then the statement is vacuous because there is no $\a<\d$.

\vskip1mm\noindent
Subcase Ib: $\g_{q_0}=\d$ and $\cf(\d)=\o$. In particular $\d$ is limit ordinal. There is $q_1\le q_0$ such that $X_{q_1}=X_{q_0}$ and that $q_1$ satisfies conclusion of Lemma \ref{t0}. Note $\sup(X_{q_1})=\d$. So pick an increasing sequence $\seq{\a_n:n<\o}$ such that $\a_0=\a$, $\sup\set{\a_n:n<\o}=\d$ and $\a_n\in X_{q_1}$ for $n<\o$. Build by induction sequences $\seq{d_n:n<\o}$ and $\seq{\pi_{\a_n}:n<\o}$ satisfying the following for each $n < \omega$: 
\begin{enumerate}
 \item \label{item:map1}
 $d_0=d$, $\pi_{\a_0}=\pi_1$, and $\forall m\le n\ [d_n\le d_m]$;
 \item \label{item:map2}
 $\pi_{\a_n}''\sset(d_n)\in \cu_{\a_n}$ and $\forall m\le n\ \forall^{\infty}k\in \sset(d_n)\ [\pi_{\a_m}(k)=\pi_{\a_n,\a_m}(\pi_{\a_n}(k))]$;
 \item \label{item:map3}
 if $n > 0$, then there is $\psi_{\a_n}\in \o^{\o}$ such that $\seq{\pi_{\a_n},\psi_{\a_n},d_n}$ is a normal triple. 
\end{enumerate}
Put ${d}_{0} = d$ and ${\pi}_{{\alpha}_{0}} = {\pi}_{1}$. 
Fix $n \in \omega$ and assume that $d_n$ and $\pi_{\a_n}$ are given satisfying (\ref{item:map1})--(\ref{item:map3}).
To get $d_{n+1}$ we apply Definition \ref{d22}(\ref{i1}) with $\a=\a_n$, $\b=\a_{n+1}$, $\pi_1=\pi_{\a_n}$, $d=d_n$, and if $n = 0$, then ${\psi}_{1} = {\psi}_{1}$ and ${b}_{1} = {b}_{1}$, while if $n > 0$, then ${\psi}_{1} = {\psi}_{{\alpha}_{n}}$ and ${b}_{1} = {d}_{n}$.
Note that in all cases the hypothesis of Definition \ref{d22}(\ref{i1}) is satisfied.
Let $a$ in Definition \ref{d22}(\ref{i1}) be $\pi_{q_1,\a_{n+1}}''\sset(c_{q_1})$.
Then there are $b\in \cu_{\a_{n+1}}$, $\psi_{\a_{n+1}}$, $\pi_{\a_{n+1}} \in \BS$ and $d_{n+1}\le d_n$ such that $b \; {\subset}^{\ast} \; \pi_{q_1,\a_{n+1}}''\sset(c_{q_1})$, $\pi_{\a_{n+1}}''\sset(d_{n+1})=b$, $\forall k\in \sset(d_{n+1})\ [\pi_{\a_n}(k)=\pi_{\a_{n+1},\a_n}(\pi_{\a_{n+1}}(k))]$ and that $\seq{\pi_{\a_{n+1}},\psi_{\a_{n+1}}, {d}_{n + 1}}$ is a normal triple.
We will prove that $d_{n+1}$ and $\pi_{\a_{n+1}}$ satisfy (\ref{item:map1})--(\ref{item:map3}).
(\ref{item:map1}) is clear.
Second, we have $\pi_{\a_{n+1}}''\sset(d_{n+1})=b\in \cu_{\a_{n+1}}$.
Next, we check that for every $m\le n+1$, $\forall^{\infty}k\in \sset(d_{n+1})\ [\pi_{\a_m}(k)=\pi_{\a_{n+1},\a_m}(\pi_{\a_{n+1}}(k))]$.
We distinguish two cases: either $m=n+1$ or $m\le n$.
If $m=n+1$, then since $\pi_{\a_{n+1},\a_{n+1}}=\id$, for every $k\in \sset(d_{n+1})$ we have $\pi_{\a_{n+1}}(k)=\pi_{\a_{n+1},\a_{n+1}}(\pi_{\a_{n+1}}(k))$.
If $m\le n$, then it is easy to find a $k_0 \in \omega$ so that for every $k\in \sset(d_{n+1})\lp k_0\rp$ the following hold: $\pi_{\a_n}(k)=\pi_{\a_{n+1},\a_{n}}(\pi_{\a_{n+1}}(k))$, $\pi_{\a_m}(k)=\pi_{\a_n,\a_m}(\pi_{\a_n}(k))$, and $\pi_{\a_{n+1},\a_m}(\pi_{\a_{n+1}}(k)) = \pi_{\a_n,\a_m}(\pi_{\a_{n+1}, {\alpha}_{n}}(\pi_{\a_{n+1}}(k)))$.
Hence for each $k\in \sset(d_{n+1})\lp k_0\rp$, $\pi_{\a_m}(k)=\pi_{\a_n,\a_m}(\pi_{\a_n}(k))=\pi_{\a_n,\a_m}(\pi_{\a_{n+1},\a_n}(\pi_{\a_{n+1}}(k)))=\pi_{\a_{n+1},\a_m}(\pi_{\a_{n+1}}(k))$, as required.
Fourth, $\seq{\pi_{\a_{n+1}},\psi_{\a_{n+1}},d_{n+1}}$ is a normal triple.
So the sequences $\seq{d_n:n<\o}$ and $\seq{\pi_{\a_n}:n<\o}$ are as required. 

Next we apply Lemma \ref{t30} in such a way that $\d$ is $\d$, $f$ is $\id$, $X=\set{\a_n:n<\o}$, $e$ is $c_{q_1}$, $\pi_{\d,\a_n}$ is $\pi_{q_1,\a_n}$ for $n<\o$, $d_n$ is $d_n$ for $n<\o$, $\pi_{\a_n}$ is $\pi_{\a_n}$  for $n<\o$, $b_{\d,\a_n}$ and $\psi_{\d,\a_n}$ are $b_{q_1,\a_n}$ and $\psi_{q_1,\a_n}$ for $n<\o$. Note that $\cf(\d)=\o$, $\sup(X)=\d$, $X\subset \d$ and that Lemma \ref{t30}(\ref{i72}) is satisfied because $q_1$ satisfies Definition \ref{d14}(\ref{i14}). So we still have to prove that condition (\ref{i76}) of Lemma \ref{t30} is satisfied. First we prove (\ref{i78}). Fix $m\le n<\o$. By the construction of $\pi_{\a_n}$ we know that $\forall^{\infty}k\in \sset(d_{n})\ [\pi_{\a_m}(k)=\pi_{\a_n,\a_m}(\pi_{\a_n}(k))]$ as
required. To see (\ref{i79}) take $n<\o$ and note that $\seq{\pi_{\a_n},\psi_{\a_n},d_n}$ is a normal triple. Next we prove (\ref{i77}). We have to prove that for every $m,n<\o$, $\pi_{\a_n}''\sset(d_m)\in \cu_{\a_n}$. We consider two cases: either $m\le n$ or $m>n$. If $m\le n$ note $\sset(d_n)\subset^* \sset(d_m)$ so since by construction $\pi_{\a_n}''\sset(d_n)\in \cu_{\a_n}$ we have $\pi_{\a_n}''\sset(d_m)\in \cu_{\a_n}$. If $m>n$ then by construction $\pi_{\a_m}''\sset(d_m)\in \cu_{\a_m}$. By already proved (\ref{i78}) $\pi_{\a_m,\a_n}''(\pi_{\a_m}''\sset(d_m))=^*\pi_{\a_n}''\sset(d_m)$. By Definition \ref{d22}(\ref{i5}), $\pi_{\a_m,\a_n}''(\pi_{\a_m}''\sset(d_m))\in \cu_{\a_n}$ so we have $\pi_{\a_n}''\sset(d_m)\in \cu_{\a_n}$ as required. 

So all conditions of Lemma \ref{t30} are satisfied. Hence, there are $e'$, $d'$, $\pi'$ and $\psi'$ satisfying conditions (\ref{i84}-\ref{i83}) of the conclusion of Lemma \ref{t30}. In particular, by (\ref{i81}) we know that there is $k_1$ such that $\forall k\in \sset(d')\lp k_1\rp\ [\pi_{\a_0}(k)=\pi_{q_1,\a_0}(\pi'(k))]$. Let $n<\o$ be such that $(\pi')''d'(k_1)\subset e'(n)$, and $k_2$ such that $\sset(e')\lp n+1\rp =(\pi')''\sset(d')\lp k_2\rp$. Since $\seq{\pi',\psi',d'}$ is a normal triple, $k_2$ is well defined and $k_2\ge k_1$. Now let $d^*$ be defined by $d^*(k)=d'(k_2+k)$ ($k<\o$), let $e^*$ be defined by $e^*(k)=e'(n+k+1)$ ($k<\o$). Note that $e^*\le_0 e'\le_0 c_{q_1}$, $\sset(e^*)=^*\sset(e')$ and $d^*\le d'$. Define also $\pi\in \o^{\o}$ as follows: for $k\in \sset(d^*)$ let $\pi(k)=\pi'(k)$ and $\pi(k)=0$ otherwise. Let $\psi\in \o^{\o}$ be defined by $\psi(k)=\psi'(k_2+k)$ for $k<\o$. Note that $\seq{\pi,\psi,d^*}$ is a normal triple and that $\pi''\sset(d^*)=\sset(e^*)$. Define $q^*=\seq{e^*,\d,X_{q_1},\seq{\pi_{q_1,\a}:\a\in X_{q_1}}}$. Since $e^*\le c_{q_1}$, by Remark \ref{r101}, in order to show $q^*\le q_1$ we only have to prove that $q^*$ satisfies Definition \ref{d14}(\ref{i16}). First we show that it holds for all $\a_n$ ($n<\o$). Take $n<\o$. By (\ref{i82}) of the conclusion of Lemma \ref{t30} and because $\sset(e^*)=^*\sset(e')$ we have $\pi_{q_1,\a_n}''\sset(e^*)\in \cu_{\a_n}$. Now we prove (\ref{i16}). Let $\a\in X_{q_1}$. Pick $\a_n\ge \a$. By Remark \ref{r101}, $q^*$ satisfies (\ref{i17}) so $\pi_{\a_n,\a}''(\pi_{q_1,\a_n}''\sset(e^*))=^*\pi_{q_1,\a}\sset(e^*)$. By Definition \ref{d22}(\ref{i5}) we know $\pi_{\a_n,\a}''(\pi_{q_1,\a_n}''\sset(e^*))\in \cu_{\a}$. Hence $\pi_{q_1,\a}''\sset(e^*)\in \cu_{\a}$ as required. We will show that $q^*$, $d^*$, $\pi$ and $\psi$ satisfy conclusion of this lemma. First, $\seq{\pi,\psi,d^*}$ is a normal triple. Second, $\a\in X_{q_0}=X_{q_1}=X_{q^*}$. Third, $\pi''\sset(d^*)=\sset(e^*)$. Fourth, for every $k\in \sset(d^*)$ we know that $k\in \sset(d')\lp k_2\rp$, so $\pi_1(k)=\pi_{\a}(k)=\pi_{q_1,\a}(\pi(k))$. Note that $\pi_1=\pi_{\a}$ by definition of $\pi_{\a_0}$. So $q^*$ is as required.

\vskip2mm\noindent
Case II: $\g_{q_0}<\d$. Let $n_0$ be such that $c_{q_0}\le_{n_0}b_{q_0,\g_{q_0}}$ and that for every $k\in \sset(c_{q_0})\lp n_0\rp$ we have $\pi_{q_0,\a}(k)=\pi_{\g_{q_0},\a}(\pi_{q_0,\g_{q_0}}(k))$

\begin{claim}\label{t201}
There are $d'\le d$, $b\subset \pi_{q_0,\g_{q_0}}''\sset(c_{q_0})\lp n_0\rp$ and $\pi_2,\psi_2\in \o^{\o}$ such that $\pi_2''\sset(d')=b$, $\forall k\in \sset(d')\ [\pi_1(k)=\pi_{\g_{q_0},\a}(\pi_2(k))]$ and that $\seq{\pi_2,\psi_2,d'}$ is a normal triple.
\end{claim}

\begin{proof}
We will consider two cases: either $\a = \g_{q_0}$ or $\a<\g_{q_0}$. If $\a<\g_{q_0}$ then we apply Definition \ref{d22}(\ref{i1}) with $\a=\a$, $\b=\g_{q_0}$, $\pi_1=\pi_1$, $\psi_1=\psi_1$, $b_1=b_1$, $d=d$ and $a=\pi_{q_0,\g_{q_0}}''\sset(c_{q_0})\lp {n}_{0} \rp$. Note that hypothesis of Definition \ref{d22}(\ref{i1}) are satisfied. Hence there are $b\in \cu_{\g_{q_0}}$, $\pi,\psi\in\o^{\o}$ and $d'\le_0 d$ so that $\seq{\pi,\psi,d'}$ is a normal triple, $\pi''\sset(d')=b$ and $\forall k\in \sset(d')\ [\pi_1(k)=\pi_{\g_{q_0},\a}(\pi(k))]$ as required.

If $\a=\g_{q_0}$, first let $l$ be such that $d\le_l b_1$. Put $b=\pi_1''\sset(d)\lp l\rp\cap \pi_{q_0,\g_{q_0}}''\sset(c_{q_0})\lp n_0\rp$ and note $b\in \cu_{\g_{q_0}}$ and $b\subset \pi_{q_0,\g_{q_0}}''\sset(c_{q_0})\lp n_0\rp$. Put $F_n=\set{m<\o: \pi_1''d(m)=\set{b(n)}}$. By Lemma \ref{t14}, $L_n=\max(F_n)$ is well defined and $L_n<L_{n+1}$ is true for $n<\o$. Define $d'$ as follows: for $n<\o$ let $d'(n)=d(L_n)$. Since $L_n<L_{n+1}$ ($n<\o$) we know that $d'\in \P$. Define $\pi_2$ as follows: for $k\notin\sset(d')$ let $\pi_2(k) = {\pi}_{1}(k)$ while $\pi_2(k)=0$ otherwise. Since $L_n\ge l$ for every $n<\o$, we can define $L'_n$ such that $d(L_n)\subset b_1(L'_n)$. Then $\pi_1''b_1(L'_n)=\set{{\psi}_{1}(L'_n)}$. Define $\psi_2(n)=\psi_1(L'_n)$. Then it is easy to see that $\seq{\pi_2,\psi_2,d'}$ is a normal triple. We know that $\pi_2''\sset(d')=b$ because for every $n<\o$, $\pi_2''d'(n)=\pi_1''d(L_n)=\pi_1''b_1(L'_n)=\set{b(n)}$. To see that $\forall k\in \sset(d')\ [\pi_1(k)=\pi_{\g_{q_0},\a}(\pi_2(k))]$ note that $\g_{q_0}=\a$ and $\pi_{\a,\a}=\id$ so for every $k\in \sset(d')$ we have that $\pi_1(k)=\pi_{\a,\a}(\pi_2(k))$ is true.
\end{proof}

Now that we have $b$ with the required properties, since $\cu_{\g_{q_0}}$ is rapid, there is $c\in \cu_{\g_{q_0}}$ so that for every $n<\o$ there is $m\ge t(n+1)$ so that $c(n)=b(m)$. We will build $e^*$, $d^*$, $\psi$, and $\pi$ so that the following hold: $e^*\le_0 c_{q_0}$, $d^*\le_0 d'$, $\pi''\sset(d^*)=e^*$, $\pi_{q_0,\g_{q_0}}''\sset(e^*)=c$, $\forall k\in \sset(d^*)\ [\pi_2(k)=\pi_{q_0,\g_{q_0}}(\pi(k))]$, and $\seq{\pi,\psi,d^*}$ is a normal triple. For each $n<\o$ define $M_n=\max\set{m<\o: \pi_{q_0,\g_{q_0}}''c_{q_0}(m)=\set{c(n)}}$ and $K_n=\max\set{m<\o:\pi_2''d'(m)=\set{c(n)}}$. By Lemma \ref{t14}, $M_n\ge n_0$, $M_{n+1}>M_n$ and $K_{n+1}> K_n$ ($n<\o$). We show that $K_n\ge t(n+1)$ for every $n<\o$. Define $l_n=\max\set{m<\o:\pi_2''d'(m)=b(n)}$ ($n<\o$), and note that by Lemma \ref{t14} numbers $l_n$ are well defined and that $l_{n+1}>l_n$ ($n<\o$). Hence $l_{n+1}\ge n+1$ ($n<\o$). Fix $n<\o$. Then there is $v_n\ge t(n+1)$ such that $c(n)=b(v_n)$. So $K_n=l_{v_n}\ge l_{t(n+1)}\ge t(n+1)$. Define $e^*$ as follows: for $n<\o$ let $e^*(n)\in [c_{q_0}(M_n)]^{n+1}$.
Define $d^*$ as follows.
For each $n' < \omega$ choose a sequence $\langle {d}^{\ast}(n): s(n') \leq n < s(n' + 1) \rangle$ in such a way that for all $s(n') \leq n < s(n' + 1)$, ${d}^{\ast}(n) \in {\left[ d'({K}_{n'}) \right]}^{n + 1}$ and for all $s(n') \leq n < n + 1 < s(n' + 1)$, $\max({d}^{\ast}(n)) < \min({d}^{\ast}(n + 1))$.
Next, define $\pi\in \o^{\o}$ as follows: for $k\notin \sset(d^*)$ let $\pi(k)=0$, while for $k\in \sset(d^*)$ let $\pi(k)=\sset(e^*)(n)$ where $n$ is such that $k\in d^*(n)$. Let $\psi\in \o^{\o}$ be defined as $\psi(n)=\sset(e^*)(n)$ for every $n<\o$. Note that $\seq{\pi,\psi,d^*}$ is a normal triple because $\pi''d^*(n)=\set{\psi(n)}=\set{\sset(e^*)(n)}$ for $n<\o$. To show that $e^*$, $d^*$, $\pi$ and $\psi$ are as required, we still have to show that for every $k\in \sset(d^*)\ [\pi_2(k)=\pi_{q_0,\g_{q_0}}(\pi(k))]$. Fix $k\in \sset(d^*)$. Let $n$ be such that $k\in d^*(n)$ and let $m$ be such that $s(m)\le n<s(m+1)$. Then $k\in d'(K_m)$ so $\pi_2(k)=c(m)$. Also $\pi(k)=\sset(e^*)(n)\in e^*(m)\subset c_{q_0}(M_m)$ so $\pi_{q_0,\g_{q_0}}(\pi(k))=c(m)=\pi_2(k)$ as required. Define $q^*=\seq{e^*,\g_{q_0},X_{q_0},\seq{\pi_{q_0,\a}:\a\in X_{q_0}}}$. Since $e^*\le_0 c_{q_0}$, by Remark \ref{r101}, in order to prove $q^*\in \Q^{\d}$ and $q^*\le q_0$ it is enough to show that $q^*$ satisfies property (\ref{i16}) of Definition \ref{d14}. So let $\b\in X_{q_0}$. There are two cases: either $\b=\g_{q_0}$ or $\b<\g_{q_0}$. If $\b=\g_{q_0}$ then $\pi_{q_0,\g_{q_0}}''\sset(e^*)=c\in \cu_{\b}$. If $\b<\g_{q_0}$ then by Remark \ref{r101} property (\ref{i17}) of Definition \ref{d14} holds for $q^*$ so $\pi_{\g_{q_0},\b}''(\pi_{q_0,\g_{q_0}}''\sset(e^*))=^*\pi_{q_0,\b}''\sset(e^*)$. Also, by Definition \ref{d22}(\ref{i5}), $\pi_{\g_{q_0},\b}''(\pi_{q_0,\g_{q_0}}''\sset(e^*))\in\cu_{\b}$ so $\pi_{q_0,\b}''\sset(e^*)\in\cu_{\b}$. Hence $q^*\in\Q^{\d}$ and $q^*\le q_0$. Finally, we prove that $q^*$ satisfies conclusion of this lemma. By the choice of $q_0$ we have $\a\in X_{q_0}$. By the choice of $d^*$ and $e^*$ we have $\pi''\sset(d^*)=\sset(e^*)$. We already explained why $\seq{\pi,\psi,d^*}$ is a normal triple. So we still have to prove that $\forall k\in \sset(d^*)\ [\pi_1(k)=\pi_{q^*,\a}(\pi(k))]$. By Claim \ref{t201} and since $d^*\le_0 d'$ we have $\forall k\in \sset(d^*)\ [\pi_1(k)=\pi_{\g_{q_0},\a}(\pi_2(k))]$. We also proved $\forall k\in \sset(d^*)\ [\pi_2(k)=\pi_{q_0,\g_{q_0}}(\pi(k))]$. Since $M_n\ge n_0$ for every $n<\o$ and $\pi''\sset(d^*)=\sset(e^*)$, we also have that $\forall k\in \sset(d^*)\ [\pi_{q_0,\a}(\pi(k))=\pi_{\g_{q_0},\a}(\pi_{q_0,\g_{q_0}}(\pi(k)))]$. From these three equations we get $\pi_1(k)=\pi_{\g_{q_0},\a}(\pi_2(k))=\pi_{\g_{q_0},\a}(\pi_{q_0,\g_{q_0}}(\pi(k)))=\pi_{q_0,\a}(\pi(k))$ as required. Hence $q^*$ satisfies conclusion of this lemma.
\end{proof}
The proofs of Lemmas \ref{t5}--\ref{t12} go through with no essential modifications under $\MA$.
Of course the proofs would depend on the fact that ${\Q}^{\delta}$ would be $< \c$ closed and the generalized form of Lemma \ref{t30} would hold in this context.
The inductive construction occurring in Subcase Ib of the proof of Lemma \ref{t12} would need to be of length $\xi$, for some $\xi < \c$.
The limit stages of this inductive construction can be passed by appealing to the generalized forms of Lemma \ref{t30} and Clause (\ref{i8}) of Definition \ref{d22}.
\section{A long chain} \label{sec:long}
We now have all the tools necessary for constructing the desired chain of P-points.
As our construction requires $\CH$, we assume ${2}^{{\aleph}_{0}} = {\aleph}_{1}$ in this section.
The chain of length ${\omega}_{2}$ will be obtained from an ${\omega}_{2}$-generic sequence. 
\begin{theorem}[$\CH$]\label{t1}
There is an $\o_2$-generic sequence.
\end{theorem}

\begin{proof}
We build by induction sequence $\seq{S_{\d'}:\d'\le \o_2}$ such that for each $\d'\le \o_2$:
\begin{enumerate}
\item\label{i211} $S_{\d'}$ is $\d'$-generic;
\item\label{i212} $\forall \g<\d'\ [S_{\g}=S_{\d'}\rest\g]$.
\end{enumerate}
For $\d'=0$, let $S_0=\seq{0,0}$. 
Next assume that $\delta'$ is a limit ordinal and that for every $\g<\d'$, we are given $S_{\g}$ as required. 
Define 
\begin{align*}
S_{\d'}=\seq{\bigcup_{\g<\d'}\seq{c^{\a}_i:\a<\g\wedge i<\c},\bigcup_{\g<\d'}\seq{\pi_{\b,\a}:\a\le\b<\g}}. 
\end{align*}
Remark \ref{r4} ensures that $S_{\d'}$ satisfies (\ref{i211}) and (\ref{i212})\footnote{We consider the sequence $\seq{c^{\a}_i:\a<\g\wedge i<\c}$ as the function from $\g\times \c$ into $\P$ while the sequence $\seq{\pi_{\b,\a}:\a\le \b<\g}$ is considered as the function from $\set{\seq{\a,\b}:\a\le\b<\g}$ into $\o^{\o}$.}. 
Finally assume that $\d'=\d+1$ and that $S_{\d}$ satisfies (\ref{i211}) and (\ref{i212}). Note $\d<\o_2$. In the next paragraph we build $S_{\d+1}$.

First partition $\o_1=T_0\cup T_1\cup T_2\cup T_3\cup T_4$ into five disjoint sets so that $\abs{T_i}=\o_1$ ($i\in 5$). Next we enumerate certain sets. Let $\cp(\o)=\set{X_i:i\in T_0}$. Let $V=\set{f_i:i\in T_1}$, where $V$ is the set of all increasing functions in $\o^{\o}$. Let $\mathtt{T}=\o^{\o}\times \o^{\o}\times \P\times \P\times \d=\set{\mathtt{t}_i:i\in T_2}$ be enumeration of $\mathtt{T}$ in such a way that every element occurs $\o_1$ many times on the list. Let $\Phi\times \d=\set{\seq{\phi_i,\a_i}:i\in T_3}$, where $\Phi$ is the set of all continuous monotone maps $\phi:\cp(\o)\to \cp(\o)$, and note that this enumeration is possible because $\abs{\Phi}=\c$. Let $\Gamma=\set{\mathtt{s}_i:i\in T_4}$ be enumeration of $\Gamma$ such that every element of $\Gamma$ appears $\o_1$ many times, where $\Gamma$ is the set of all $\seq{X,\bar d, \bar \pi, \bar b, \bar \psi}$ such that $X\in [\d]^{\le\o}$, $\bar d\in \P^{\o}$, $\bar \pi\in (\o^{\o})^{X}$, $\bar b\in \P^{X}$, $\bar \psi\in (\o^{\o})^X$. Now we build a decreasing sequence $\seq{q_i:i<\o_1}$ in $\Q^{\d}$. Since $\Q^{\d}\neq 0$ pick arbitrary $q_0\in \Q^{\d}$. Assume that for $i<\o_1$ we already built $\seq{q_j:j<i}$. If $i$ is limit then by Lemma \ref{t91} we choose $q_i$ such that $q_i\le q_j$ ($j<i$). If $i=i_0+1$ then we distinguish five cases. Suppose that $i_0\in T_0$. Then $X_{i_0}\in \cp(\o)$. By Lemma \ref{t3} we pick $q_i\le q_{i_0}$ such that $\sset(c_{q_i})\subset X_{i_0}$ or $\sset(c_{q_i})\subset \o\setminus X_{i_0}$. Suppose that $i_0\in T_1$. Then $f_{i_0}\in\o^{\o}$ is a strictly increasing so by Lemma \ref{t11} there is $q_i\le q_{i_0}$ such that $\forall n<\o\ [\sset(c_{q_i})(n)\ge f_{i_0}(n)]$. Suppose that $i_0\in T_2$. If $\seq{\pi_{i_0},\psi_{i_0},b_{i_0}}$ is a normal triple, $d_{i_0}\le b_{i_0}$ and $\pi_{i_0}''\sset(d_{i_0})\in \cu_{\a_{i_0}}$ then by Lemma \ref{t12} pick $q_i\le q_{i_0}$, $d^*_i\le d_{i_0}$, $\pi^*_i,\psi^*_i\in\o^{\o}$ so that $\seq{\pi^*_i,\psi^*_i,d^*_i}$ is a normal triple, $\a_{i_0}\in X_{q_i}$, $(\pi^*_i)''\sset(d^*_i)=^*\sset(c_{q_i})$ and $\forall k\in \sset(d^*_i)[\pi_{i_0}(k)=\pi_{q_i,\a_{i_0}}(\pi^*_i(k))]$. Otherwise let $q_i=q_{i_0}$. Suppose that $i_0\in T_3$. Then $\phi_{i_0}:\cp(\o)\to\cp(\o)$ is monotone and continuous. If $\phi_{i_0}(A)\neq 0$ for every $A\in \cu_{\a_{i_0}}$ then by Lemma \ref{t4} pick $q_i\le q_{i_0}$ such that $\a_{i_0}\in X_{q_i}$ and $\phi(A)\nsubseteq \sset(c_{q_i})$ for every $A\in \cu_{\a_{i_0}}$. Otherwise let $q_i=q_{i_0}$. Suppose that $i_0\in T_4$. Then $X_{i_0}\in [\d]^{\le\o}$, $\bar d_{i_0}$ and $\bar b_{i_0}$ are decreasing sequences in $\P$ and $\bar \pi_{i_0},\bar \psi_{i_0}\in (\o^{\o})^{X_{i_0}}$. If $\cf(\d)=\o$, $\g_{q_{i_0}}=\d$, $X_{i_0}\subset X_{q_{i_0}}$, $\sup(X)=\d$ and $\bar d_{i_0}$, $\bar \pi_{i_0}$, $\bar \psi_{i_0}$ and $\bar b_{i_0}$ satisfy Lemma \ref{t9}(\ref{i2001}-\ref{i2003}), then by Lemma \ref{t9} pick $q_i\le q_{i_0}$, $d^*_i$ and $\pi^*_i$ satisfying Lemma \ref{t9}(\ref{i2005}-\ref{i2008}). Otherwise let $q_i=q_{i_0}$. Now define $S_{\d'}$ as follows: $\pi_{\d,\d}=\id$, for $\a\le \b<\d$ and $i<\o_1$, $c^{\a}_i$ and $\pi_{\b,\a}$ are as in $S_{\d}$, while for $\a<\d$ and $i<\o_1$, $c^{\d}_i$ is $c_{q_i}$ and $\pi_{\d,\a}$ is $\pi_{q_j,\a}$ where $j$ is minimal such that $\a\in X_{q_j}$.

\begin{claim}\label{t13}
For every $\a<\d$ there is $i<\o_1$ such that $\a\in X_{q_{i}}$.
\end{claim}

\begin{proof}
Take $\a<\d$. Consider the function $\phi:\cp(\o)\to\cp(\o)$ given by $\phi(A)=\o$ for $A\subset \o$. Clearly, $\phi(A)\neq 0$ for $A\in \cu_{\a}$. So there is $i_0\in T_3$ so that $\seq{\phi,\a}=\seq{\phi_{i_0},\a_{i_0}}$. Then for $i=i_0+1$, by choice of $q_i$ we have $\a_{i_0}\in X_{q_i}$.
\end{proof}

Note that by Claim \ref{t13} $\pi_{\d,\a}$ is defined for every $\a<\d$. Namely, if $i<\o_1$ is such that $\a\in X_{q_i}$, then $\pi_{\d,\a}=\pi_{q_i,\a}$. We still have to prove that $S_{\d'}$ is $\d+1$-generic sequence. Only conditions (\ref{i1}) and (\ref{i8}) of Definition \ref{d22} need checking.

To see that (\ref{i1}) holds, take any $\a<\b\le \d$. If $\b<\d$ the statement follows because $S_{\d}$ is $\d$-generic and $S_{\d}=S_{\d+1}\rest\d$. If $\b=\d$ let $\pi_1,\psi_1,d_1,b_1,\a$ be as in the statement of (\ref{i1}) and let $a\in \cu_{\d}$. Since $a\in\cu_{\d}$ there is $j<\c$ such that $c^{\d}_j=c_{q_j}\subset^* a$. Then $\seq{\pi_1,\psi_1,d_1,b_1,\a}=\mathtt{t}_i$ for some $i\in T_2$ such that $i\ge j$. Note that this is true because every element of $\mathtt{T}$ appears $\o_1$ many times in its enumeration. So by the choice of $q_{i+1}$ we know that $c_{q_{i+1}}\subset^* c_{q_j}\subset^*a$ and that $d^*_{i+1}$, $\pi^*_{i+1}$, $\psi^*_{i+1}$ and $b^*_{i+1}$ are such that $\seq{\pi^*_{i+1},\psi^*_{i+1},b^*_{i+1}}$ is a normal triple, $\a\in X_{q_{i+1}}$, $(\pi^*_{i+1})''\sset(d^*_{i+1})=^*\sset(c_{q_{i+1}})$ and $\forall k\in \sset(d^*_{i+1})\ [\pi_1(k)=\pi_{q_{i+1},\a}(\pi^*_{i+1}(k))]$. Denote $b=(\pi^*_{i+1})''\sset(d^*_{i+1})$ and note that $b\subset^* a$. Now $b$, $d^*_{i+1}$, $\pi^*_{i+1}$, $\psi^*_{i+1}$ and $b^*_{i+1}$ witness that Definition \ref{d22}(\ref{i1}) is true in this case also.

Next we show that $S_{\d'}$ satisfies condition (\ref{i8}) of Definition \ref{d22}. Let $\mu$, $X$, $\bar d$, and $\bar \pi$ be as in Definition \ref{d22}(\ref{i8}). Since we require $i^*$ satisfying Definition \ref{d22}(\ref{i8}) to be cofinal in $\c$ let $i_0<\c$ be fixed. If $\mu<\d$, the condition is satisfied because $S_{d'}=S_{\d}\rest\d$ and $S_{\d}$ is $\d$-generic. If $\mu=\d$ then, by Claim \ref{t13} and because every element of $\Gamma$ appears $\c$ many times in its enumeration, there is $i\in T_4$ such that $i\ge i_0$ and $\seq{X,\bar d,\bar \pi,\bar b,\bar \psi}=\mathtt{s}_i$ and $X\subset X_{q_i}$. By the choice of $q_{i+1}$ we know that $d^*_{i+1}$, $\pi^*_{i+1}$ and $c^{\d}_{i+1}=c_{q_{i+1}}$ satisfy conditions (\ref{i2005}-\ref{i2008}) of Lemma \ref{t9} which implies that they also satisfy condition (\ref{i8}) of Definition \ref{d22}. Since $\pi_{\d,\a}=\pi_{q_{i+1},\a}$ for $\a\in X$ this shows that $i+1\ge i_0$ witnesses that Definition \ref{d22}(\ref{i8}) is satisfied.
\end{proof}
For each $\alpha < {\omega}_{2}$, let ${\cu}_{\alpha} = \{a \in \cp(\omega): \exists i < \c [\sset({c}^{\alpha}_{i}) \; {\subset}^{\ast} \; a]\}$.
As we have noted in Section \ref{sec:prelim}, $\langle {\cu}_{\alpha}: \alpha < {\omega}_{2}\rangle$ is a sequence of P-points that is strictly increasing with respect to both ${\leq}_{RK}$ and ${\leq}_{T}$.
Thus the ordinal ${\omega}_{2}$ embeds into the P-points under both orderings.
In fact, the proof of Theorem \ref{t1} shows something slightly more general.
We could have started the construction with a fixed $\delta$-generic sequence for some $\delta < {\omega}_{2}$, and then extended it to an ${\omega}_{2}$-generic sequence in the same way.
So we have the following corollary to the proof.
\begin{corollary}[$\CH$]\label{t15}
Suppose that $\d<\o_2$ and that $S_{\d}$ is a $\d$-generic sequence. Then there is an $\o_2$-generic sequence $S$ such that $S\rest\d=S_{\d}$.
\end{corollary}
When $\CH$ is replaced with $\MA$ and the lemmas from Section \ref{sec:top} have been appropriately generalized, the proof of the natural generalization of Theorem \ref{t1} presents little difficulty.
In the crucial successor step of the construction, ${\omega}_{1}$ can be replaced everywhere with $\c$; all of the sets that need to be enumerated have size $\c$ because ${\c}^{< \c} = \c$ under $\MA$.
The generalizations of the lemmas from Section \ref{sec:top} imply that each condition in ${\Q}^{\delta}$ has an extension that meets some given requirement, and the fact that ${\Q}^{\delta}$ is $< \c$ closed allows us to find lower bounds at the limit steps.
Therefore a ${\c}^{+}$-generic sequence exists under $\MA$.
\input{RK-51.bbl}

%\bibliographystyle{amsplain}
%\bibliography{Bibliography}
\end{document}

%% file: RK-51.bbl
\def\polhk#1{\setbox0=\hbox{#1}{\ooalign{\hidewidth
  \lower1.5ex\hbox{`}\hidewidth\crcr\unhbox0}}}
\providecommand{\bysame}{\leavevmode\hbox to3em{\hrulefill}\thinspace}
\providecommand{\MR}{\relax\ifhmode\unskip\space\fi MR }
% \MRhref is called by the amsart/book/proc definition of \MR.
\providecommand{\MRhref}[2]{%
  \href{http://www.ams.org/mathscinet-getitem?mr=#1}{#2}
}
\providecommand{\href}[2]{#2}